\documentclass[11pt]{amsart}

 \usepackage{amsfonts,graphics,amsmath,amsthm,amsfonts,amscd,amssymb,amsmath,latexsym,multicol,
 mathrsfs}
\usepackage{epsfig,url}
\usepackage{flafter}
\usepackage{fancyhdr}
\usepackage{hyperref}
\hypersetup{colorlinks=true, linkcolor=black}
 \usepackage[matrix, arrow]{xy}

\DeclareMathOperator{\Proj}{Proj}

\DeclareMathOperator{\Spec}{Spec}
\DeclareMathOperator{\Supp}{Supp}

\DeclareMathOperator{\vol}{vol}

\DeclareMathOperator{\Ivol}{Ivol}
\DeclareMathOperator{\sdeg}{sdeg}
\DeclareMathOperator{\Nklt}{Nklt}

 \numberwithin{equation}{subsection}
 \numberwithin{footnote}{subsection}

 \newtheorem{lem}[subsection]{Lemma}
 \newtheorem{prop}[subsection]{Proposition}
 \newtheorem{thm}[subsection]{Theorem}
 \newtheorem{conj}[subsection]{Conjecture}

 \newtheorem{quest}[subsection]{Question}

{
{%_%_%_%_% upright style; roman (non-italic) text
    \newtheoremstyle{upright}%
        {8pt plus2pt minus4pt}%
        {8pt plus2pt minus4pt}%
        {\upshape}%
        {}%
        {\bfseries\scshape}%
        {}%
        {1em}%
        {}%

\theoremstyle{upright}
 \newtheorem{defn}[subsection]{Definition}
 \newtheorem{exa}[subsection]{Example}

}

 \newcommand{\N}{\mathbb N}
 \newcommand{\PP}{\mathbb P}
 \newcommand{\A}{\mathbb A}
 \newcommand{\Q}{\mathbb Q}
 \newcommand{\R}{\mathbb R}
 \newcommand{\Z}{\mathbb Z}
  \newcommand{\C}{\mathbb C}
 \newcommand{\bir}{\dashrightarrow}
 \newcommand{\rddown}[1]{\left\lfloor{#1}\right\rfloor} % round-down

\title{\large M\MakeLowercase{oduli of algebraic varieties}}
\author{\large C\MakeLowercase{aucher} B\MakeLowercase{irkar}}
\date{\today}
\begin{document}
\maketitle

\begin{abstract}
We develop a moduli theory of algebraic varieties and pairs of non-negative Kodaira dimension. 
We define stable minimal models and construct their projective coarse moduli spaces under certain natural conditions. This can be applied to a wide range of moduli problems in algebraic geometry.
\end{abstract}

\tableofcontents

%%%%%%%%%%%%%%%%%%%%%%%%
%%%%%%%%%%%%%%%%%%%%%%%%%%

\section{\bf Introduction}

We work over an algebraically closed field $k$ of characteristic zero unless stated otherwise.\\

In this paper, we investiage compact moduli spaces of algebraic varieties. 
Moduli theory is an old and extensively studied area in algebraic geometry. It is very closely related to birational geometry. According to standard conjectures in birational geometry, any smooth projective variety $W$ of Kodaria dimension $\kappa(W)\ge 0$ is birational to a good minimal variety, that is, a projective variety $X$ with good singularities such that $mK_X$ is base point free for some $m\in \N$. 
From the point of view of classification of algebraic varieties it is then quite natural to focus on the moduli of such minimal varieties.
Not suprisingly, historically, such varieties have been at the centre of attention in moduli theory, especially when $\kappa(W)=\dim W$ and $\kappa(W)=0$.

Riemann constructed moduli spaces $M_g$ of smooth projective curves of fixed genus $g$. It was only in the 1960's that a meaningful compactification of these moduli spaces for $g\ge 2$ was constructed in \cite{DM69} via the introduction of stable nodal curves. 
In \cite{KSB88} a program was initiated to extend the theory to surfaces of general type. The program was completed in \cite{Al94}. They defined the so-called KSBA-stable varieties of general type, which are projective schemes $X$ with semi-log canonical (slc) singularities and ample canonical divisor $K_X$. Slc singularities are analogues of nodal singularities for curves. The moduli theory of KSBA-stable varieties of general type has been extended to higher dimension in full generality by contributions of many people, culminating in \cite{K21} (also see \cite{Al96}). This theory heavily relies on developments in birational geometry, in particular, on existence of minimal models \cite{BCHM10} and boundedness of varieties of general type \cite{HMX18}. The latter says that to get a finite type moduli space one should fix $d=\dim X$ and fix the volume $\vol(K_X)=K_X^d$. 

How about moduli of varieties that are not of general type? In this case, in the absence of a natural polarisation, unlike the case of general type, one needs to add a polarisation. One option is to consider polarisation by an ample invertible sheaf (cf. \cite{V95}) in which case the moduli space usually cannot be compactified in a ``meaningful way". Another option is to consider polarisation by an effective ample divisor (or ``weakly ample" effective divisor as we see below) which is much more suitable for construction of compact moduli spaces. In some special cases the two options coverge.
  
In the literature, the focus has been largely on certain classes of Calabi-Yau varieties. 
An example is that of elliptic curves where one marks a point on the curve. The moduli space of such marked curves can be compactified in a natural way by adding marked degenerations of elliptic curves. Similarly, in higher dimensions abelian varieties with a theta divisor admit a moduli space that can be naturally compactified \cite{Al02}. On the other hand, moduli of log Calabi-Yau varieties relates to moduli of Fano varieties and other interesting classes of varieties. For example, 
\cite{Ha04} studied compactification of moduli of plane curves and in the process studied moduli of certain log Calabi-Yau surfaces. Inspired by these, \cite{KX20} defined polarised log Calabi-Yau varieties which are just stable log    
Calabi-Yau varieties by our definition below.

The purpose of this paper is 
\begin{itemize}
\item to define and study stable varieties of arbitrary non-negative Kodaira dimension that generalises both KSBA-stable varieties and polarised Calabi-Yau varieties, 
\item to establish their boundedness under natural assumptions, and 
\item then construct their projective moduli spaces. 
\end{itemize}
In the Calabi-Yau case, boundedness was established in \cite{B20} which says that to get a finite type moduli space one should fix $d=\dim X$ and fix the volume $\vol(A)=A^d$ of the polarisation. Boundedness in the intermediate Kodaira dimension cases will be established in this paper. To achieve this we introduce new ideas and techniques and also make use of relevant works in recent years including \cite{B19}, \cite{B21a}, \cite{B22}, \cite{B20}, \cite{B21b}, \cite{BH22}, the theory of generalised pairs \cite{BZ16}, and \cite{HMX18}, etc. To construct the moduli spaces we do not rely on existence of moduli spaces of KSBA-stable varieties of general type directly but instead we use the tools in its construction. Our proof in the KSBA-stable general type case is not new and simply reduces to that given in \cite{K21}. Having said that, we still get new moduli functors in the general type case when $K_X$ is nef and big but not ample.  

In its simple form, a \emph{stable minimal model} $X,A$ consists of a connected projective scheme $X$ and a divisor $A\ge 0$ such that   
\begin{itemize}
\item $X$ has slc singularities,
\item $K_X$ is semi-ample, 
\item $K_X+tA$ is ample for some $t>0$, and 
\item $(X,tA)$ is slc for some $t>0$.
\end{itemize}
Note that $A$ does not need to be ample globally and this is what we meant by ``weakly ample". Indeed, if $X\to Z$ is the contraction defined by $K_X$, then $K_X+tA$ being ample for some $t>0$ is equivalent to saying $A$ is ample over $Z$. So $A$ is only a weak kind of polarisation. Moreover, $t$ is not fixed and we think of it as a sufficiently small number.

When $A=0$, $X$ is just a KSBA-stable scheme of general type. And when $K_X\sim_\Q 0$, $X,A$ is a stable Calabi-Yau scheme (that is, a polarised Calabi-Yau). 

In order to get a finite type moduli space of stable minimal models, we fix $d=\dim X$, fix the volume of $A$ along general fibres of $X\to Z$, and fix a polynomial $\sigma$ so that $\sigma(t)=\vol(K_X+tA)$ for small $t>0$. So we need to fix more data compared with the general type and Calabi-Yau cases but this is just a reflection of the fact that we are treating a much more complex situation. 

We will develop the theory in the setting of minimal model pairs $(X,B)$ with slc singularities. In particular, this also allows us to construct moduli spaces of Fano varieties and Fano fibrations that are suitably polarised. 

Quasi-projective moduli spaces of smooth good minimal models $X$ polarised by an ample invertible sheaf $\mathcal{L}$ was constructed in \cite{V95}. Although this is a different approach to moduli but we should emphasize that whatever approach one takes, our boundedness results in this paper and earlier works (\cite{B20}, \cite{B21b}) are key for the full or partial compactification of moduli spaces, e.g. see \cite{Od20},\cite{Od21} for the case of Calabi-Yau varieties. It is also worth mentioning that in \cite{V95}, a double Hilbert polynomial $\mathcal{X}(\omega_X^{m}\otimes \mathcal{L}^{n})$ is fixed to get a finite type moduli space. In our approach, we fix a weaker numerical version, that is, we fix a polynomial $\sigma$ controlling the volume of $K_X+tA$. 

In the rest of this introduction, we define our main concepts in a general form and state our main results concerning boundedness and moduli of stable minimal models as well as other auxilliary results, e.g.  Stein degrees, which are of independent interest. In Section \ref{s-examples} we will look at special cases of moduli spaces and some open problems. 

\vspace{0.5cm}
{\textbf{\sffamily{Stable minimal models.}}} 
We start with defining our stable objects. For the definition of log canonical (lc) and slc pairs, see \ref{ss-pairs} and \ref{ss-slc-pairs}. 

\begin{defn}\label{d-stabl-mmodels-I}
A \emph{stable minimal model} $(X,B),A$ over $k$ consists of $(X,B)$ and an $\R$-divisor $A\ge 0$ where
\begin{itemize}
\item $(X,B)$ is a projective connected slc pair,
\item $K_X+B$ is semi-ample defining a contraction $f\colon X\to Z$, 
\item $K_X+B+tA$ is ample for some $t>0$, i.e. $A$ is ample over $Z$, and 
\item $(X,B+tA)$ is slc for some $t>0$.\
\end{itemize}
We usually denote the model by $(X,B),A$ or more precisely $(X,B),A\to Z$. As pointed out above, $t$ is not fixed and we usually think of it as a sufficiently small number. 
Note that the last condition of the definition just means that $\Supp A$ does not contain any non-klt centre of $(X,B)$. When $(X,B)$ is klt, this condition is automatically satisfied.

We say $(X,B),A$ is \emph{strongly stable} if in addition $K_X+B+A$ is ample. Strongly stable minimal models were defined and studied in \cite{B21b} in the lc case.

A stable minimal model $(X,B),A$ is an \emph{lc stable minimal model} if $(X,B)$ is lc. 
%Klt stable minimal models are similalry defined.
\end{defn}

\begin{exa}
Let $X=\PP^2$ and let $B$ be a nodal cubic curve. Let $A$ be any curve not passing through the node. Then $(X,B),A$ is a stable minimal model.   
\end{exa}

\begin{exa}
A KSBA-stable pair means a stable minimal model $(X,B),A$ with $A=0$. 
\end{exa}

\begin{exa}
A stable Calabi-Yau pair means a stable minimal model $(X,B),A$ with $K_X+B\sim_\R 0$. 
\end{exa}

\begin{exa}
Assume that $(X,B),A$ is an lc stable minimal model of dimension 3 and $X\to Z$ is the corresponding contraction. Then $\dim Z$ is equal to the Kodaira dimension $\kappa(K_X+B)$. We have the following cases:
\begin{itemize}
\item $\kappa(K_X+B)=3$: $X\to Z$ is birational, $(X,B)$ is of general type (but $K_X+B$ may not be ample),

\item $\kappa(K_X+B)=2$: $X\to Z$ is a fibration of relative dimension $1$, an elliptic fibration  or a conic bundle,

\item $\kappa(K_X+B)=1$: $X\to Z$ is a fibration of relative dimension $2$, an abelian fibration or K3 fibration or del Pezzo fibration, etc,

\item $\kappa(K_X+B)=0$: $X\to Z$ is constant, $(X,B)$ is Calabi-Yau.
\end{itemize}
\end{exa}

\begin{exa}
A main source of examples of slc stable minimal models is degeneration of lc stable minimal models. 
A family of lc stable minimal models over a smooth curve can be extended to a family over the compactification of the curve, after a finite base change. For the precise statement, see \ref{l-s.min.model-extension-over-curves}.  All the fibres are then stable minimal models and the special fibres are often not normal. A simple example can be constructed by degenerating an elliptic curve into the union of a conic and a line. 
\end{exa}

\begin{exa} 
Another main source of examples comes from singularities. Let $(V,\Delta)$ be a projective lc pair with $K_V+\Delta$ ample and let $(Y,\Theta)$ be a $\Q$-factorial dlt model. Let $X=\rddown{\Theta}$ and define $K_X+B=(K_Y+\Theta)|_X$. Then $K_X+B$ is semi-ample defining a contraction $X\to Z$. Moreover, we have an induced surjective finite morphism $Z\to N$ where $N$ is the non-klt locus of $(V,\Delta)$. By the connectedness principle, $X\to N$ has connected fibres, so $Z\to N$ is bijective. On the other hand, $N$ is semi-normal \cite{Am03},\cite{F11}, so $Z\to N$ is an isomorphism. 
 Pick any divisor $A$ on $X$ that is ample over $Z$ and such that $(X,B+tA)$ is slc for some $t>0$. Then each connected component of $(X,B),A\to Z$ is a stable minimal model. In particular, in this way we get lots of examples in which $Z$ is not pure dimensional. 

For a specific example, let $V=\PP^3$ and let $S_1,\dots,S_7$ be hyperplanes. Assume $S_1,S_2,S_3$ pass through a line $L$ but that $S_4,\dots,S_7$ are in general position. Let 
$$
\Delta=\sum_{i=1}^3\frac{2}{3}S_i+S_4+\sum_{i=5}^7 \frac{1}{2}S_i.
$$ 
Then $K_V+\Delta$ is ample and $(V,\Delta)$ is lc with non-klt locus $N=S_4\cup L$. Take $Y\to V$ to be the blowup along $L$ and $K_Y+\Theta$ be the pullback of $K_V+\Delta$. Then $X=\rddown{\Theta}=S_4^\sim+E$ where $S_4^\sim$ is the birational transform of $S_4$ and $E$ is the exceptional divisor, and $X\to Z=N$ is the induced morphism. 
\end{exa}

\begin{defn}\label{d-stabl-mmodels-II}
Given a set of data $\Xi$, a $\Xi$-stable minimal model is a stable minimal model satisfying the given data. We then denote the set of all stable minimal models $(X,B),A$ with data $\Xi$ as $\mathcal{S}_{\rm slc}\Xi$. The subset consisting of the lc stable minimal models is denoted by $\mathcal{S}_{\rm lc}\Xi$. The klt case is similarly defined. 

To be more precise, let 
$$
d\in \N, ~~~~\Phi,\Gamma\subset \R^{\ge 0}, ~~~~ \sigma\in \R[t]
$$ 
where $\R[t]$ is the polynomial ring in one variable over $\R$. We will define various classes using this data.

(1) A $(d,\Phi)$-stable minimal model is a stable minimal model $(X,B),A$ such that    
\begin{itemize}
\item $\dim X=d$, and
\item the coefficients of $B$ and $A$ are in $\Phi$.
\end{itemize}
The set of all the $(d,\Phi)$-stable minimal models is denoted $\mathcal{S}_{\rm slc}(d,\Phi)$.

(2) A $(d,\Phi,\Gamma,\sigma)$-stable minimal model is a $(d,\Phi)$-stable minimal model $(X,B),A$ such that    
\begin{itemize}
\item $\vol(A|_F)\in \Gamma$ where $F$ is any general fibre of $f\colon X\to Z$ over any irreducible component of $Z$, and  
\item we have 
$$
\vol(K_X+B+tA)=\sigma(t) \ \ \mbox{for} \ \ 0\le t \ll 1.
$$
\end{itemize}
Let $\mathcal{S}_{\rm slc}(d,\Phi,\Gamma,\sigma)$ denote the set of all $(d,\Phi,\Gamma,\sigma)$-stable minimal models.
\end{defn}

Although $t$ is a variable in $\sigma(t)$ but we often abuse language and use it also to denote real numbers. 

\begin{exa}
A $(d,\Phi,\Gamma,\sigma)$-stable minimal model $(X,B),A$ is a KSBA-stable pair iff $A=0$ iff $\sigma=\vol(K_X+B)$ is constant. 
\end{exa}

\begin{exa}
A $(d,\Phi,\Gamma,\sigma)$-stable minimal model $(X,B),A$ is a stable Calabi-Yau pair iff $K_X+B\sim_\R 0$ iff $\sigma(t)=\vol(A)t^{d}$ is a monomial of degree $d$. Existence of moduli spaces for such pairs was established in the first version of \cite{B20} when $\Phi=c\Z^{\ge 0}$ for a positive rational number $c$ and $\Gamma=\{u\}$ for a positive rational number $u$.     
But the proofs have been moved to this article and integrated into the main theorem on existence of moduli spaces (Theorem \ref{t-moduli-s.min.models}). 
\end{exa}

\begin{exa}
Consider a smooth minimal projective surface $X$ of Kodaira dimension one. Let $X\to Z$ be the contraction defined by $K_X$ and assume that the Iitaka volume $\Ivol(K_X)=v$ (see \ref{ss-vol-Ivol} for the definition of Iitaka volumes). Assume $A$ is a multi-section of $X\to Z$ of degree $u$. Let $\sigma(t)=2uvt+A^2t^2$. Then $(X,0),A\to Z$ is a $(2,\{1\},\{u\},\sigma)$-stable minimal model. 
\end{exa}

\vspace{0.5cm}
{\textbf{\sffamily{Boundedness of stable minimal models.}}} 
We are ready to state our boundedness result which is one of the main results of this paper.  See \ref{ss-bnd-s.min.models} for the definition of boundedness of stable minimal models.

\begin{thm}\label{t-bnd-s-mmodels-slc}
Let $d\in\N$, $\Phi\subset \Q^{\ge 0}$ be a DCC set, $\Gamma\subset \Q^{>0}$ be a finite set, and $\sigma\in\Q[t]$ be a polynomial. Then $\mathcal{S}_{\rm slc}(d,\Phi,\Gamma,\sigma)$ is a bounded family.
\end{thm}

In Proposition \ref{p-bnd-w-s-min-models-DCC} we prove a different boundedness statement, which instead of the condition on the volume of $A$ along the general fibres of $X\to Z$, assumes DCC property of Iitaka volumes after normalising $X$. 

%Boundedness of the subset of $\mathcal{S}_{\rm lc}(d,\Phi,\Gamma,\sigma)$ consisting of the strongly stable minimal models, i.e. those with $K_X+B+A$ ample, was established in 
%\cite[Theorem 1.10]{B21b} which is crucially used below. 

When $\sigma$ is constant, the theorem is boundedness of KSBA-stable minimal models which is \cite[Theorem 1.1]{HMX18}. When $\sigma$ is of the form $ut^d$, the theorem is boundedness of stable Calabi-Yau pairs which is \cite[Corollary 1.8]{B20}. Treating the general case of the theorem 
occupies a large proportion of this paper where some of the main new ideas are introduced. 

It is natural to ask whether the theorem holds if we remove $\sigma$ and instead fix the Iitaka volume of $K_X+B$. It is not hard to see that the theorem would not hold. For example, consider Hirzebruch surfaces $X$ with their $\PP^1$-bundle structure $X\to Z=\PP^1$ and their toric boundary 
$\Delta$, and a section $A$ of $X\to Z$. Let $B$ be the sum of $\Delta$ and a general fibre $F$ of $X\to Z$. Then we can choose $A$ so that $(X,B),A$ is a stable minimal model. But such models are not bounded although $\dim X=2$, $B,A$ are integral, $\vol(A|_F)=1$, and $\Ivol(K_X+B)=1$.    

However, conjecturally we can remove the condition on $\vol(A|_F)$ (see \ref{conj-no-vertical-volume}). 

\vspace{0.5cm}
{\textbf{\sffamily{Moduli of stable minimal models.}}} 
We now define families of minimal models and the corresponding moduli functors. For simplicity we will restrict ourselves to reduced base schemes as we will have enough difficulties to overcome. 
See section \ref{s-locally-stable-families} for the relevant definitions of families of pairs over reduced bases.

\begin{defn}\label{defn-stable-minimal-family}
Let $S$ be a reduced scheme over $k$. 

(1) 
When $S=\Spec K$ for a field $K$, define a stable minimal model over $K$ as in \ref{d-stabl-mmodels-I} replacing $k$ with $K$ and replacing connected with geometrically connected (slc pairs over $K$ are defined in \ref{ss-slc-pairs}). Similarly we define $(d,\Phi,\Gamma,\sigma)$-stable minimal models over $K$.

(2)
For general $S$, a \emph{family of stable minimal models} over $S$ consists of a 
projective morphism $X\to S$ of schemes and $\Q$-divisors $B,A$ on $X$ such that  
\begin{itemize}
\item $(X,B+tA)\to S$ is a locally stable family for every sufficienty small rational number $t\ge 0$, and 

\item  $({X_s},B_s),A_s$ is a stable minimal model over $k(s)$ for each $s\in S$.
\end{itemize}
Here $X_s$ is the fibre of $X\to S$ over $s$ and $B_s,A_s$ are the divisorial pullbacks of $B,A$ to $X_s$.
We will denote the family by $(X,B),A\to S$ (but this should not be confused with the notation $(X,B),A\to Z$ in \ref{d-stabl-mmodels-I}). 

We say the family is a \emph{family of lc stable minimal models} if $({X_s},B_s)$ is lc for each $s\in S$. 

(3)
Now assume $d\in\N$, $c\in \Q^{\ge 0}$, $\Gamma\subset \Q^{>0}$, and $\sigma\in\Q[t]$ is a polynomial.
To simplify notation, put $\Phi_c:=c\Z^{\ge 0}$.
A \emph{family of $(d,\Phi_c, \Gamma,\sigma)$-stable minimal models} over $S$ is a family of stable minimal models $(X,B),A\to S$ such that   
\begin{itemize}
\item $B=cD$ and $A=cN$ where $D,N\ge 0$ are relative Mumford divisors, and

\item $(X_s,B_s),A_s$ is a $(d,\Phi_c, \Gamma,\sigma)$-stable minimal model over $k(s)$ for each $s\in S$.
\end{itemize}
\end{defn}
The reason for taking $\Phi_c=c\Z^{\ge 0}$ rather than a more general set $\Phi$ is that when we consider log fibres $(X_s,B_s)$ of the family we want to have the same type of coefficients. This is ensured for $c\Z^{\ge 0}$ but not for arbitrary DCC sets.

(4)
Define the moduli functor $\mathfrak{S_{\rm slc}}(d,\Phi_c, \Gamma,\sigma)$ on the category of reduced $k$-schemes by setting  
$$
\mathfrak{S_{\rm slc}}(d,\Phi_c, \Gamma,\sigma)(S)=\{\mbox{families of $(d,\Phi_c, \Gamma,\sigma)$-stable minimal models over $S$},
$$
$$
\hspace{9cm} \mbox{ up to isomorphism over $S$}\}.
$$
Define ${\mathfrak{S}_{\rm lc}}(d,\Phi_c, \Gamma,\sigma)$ similarly by replacing slc with lc. 

Given a morphism of reduced schemes $T\to S$, the map 
$$
\mathfrak{S_{\rm slc}}(d,\Phi_c, \Gamma,\sigma)(S)\to \mathfrak{S_{\rm slc}}(d,\Phi_c, \Gamma,\sigma)(T)
$$ 
is defined as in \ref{ss-base-change-fam.s.m.model} (similarly for the lc case).
%\vspace{0.1cm}

\begin{thm}\label{t-moduli-s.min.models}
The functor $\mathfrak{S_{\rm slc}}(d,\Phi_c, \Gamma,\sigma)$ has a projective coarse moduli space ${M_{slc}}(d,\Phi_c, \Gamma,\sigma)$. Moreover, the functor $\mathfrak{S_{\rm lc}}(d,\Phi_c, \Gamma,\sigma)$ has a quasi-projective coarse moduli space ${M_{\rm lc}}(d,\Phi_c, \Gamma,\sigma)$ which is an open subset of ${M_{\rm slc}}(d,\Phi_c, \Gamma,\sigma)$. 
\end{thm}

The closure 
$$
\overline{{M}_{\rm lc}}(d,\Phi_c, \Gamma,\sigma)\subset {M_{\rm slc}}(d,\Phi_c, \Gamma,\sigma)
$$ 
gives a ``meaningful" compactification of ${M_{\rm lc}}(d,\Phi_c, \Gamma,\sigma)$. In fact, we can define limits of lc $(d,\Phi_c, \Gamma,\sigma)$-stable minimal models and their corresponding functor $\overline{\mathfrak{S_{\rm lc}}}(d,\Phi_c, \Gamma,\sigma)$ whose moduli space would be $\overline{{M}_{\rm lc}}(d,\Phi_c, \Gamma,\sigma)$. 
We can also consider the functor $\overline{{M}_{\rm klt}}(d,\Phi_c, \Gamma,\sigma)$ for klt singularities and other functors for appropriate classes of singularities. The relevant moduli spaces can be derived from the theorem.

\vspace{0.5cm}
{\textbf{\sffamily{Stein degree on log Calabi-Yau fibrations.}}}
Let $S\to Z$ be a projective morphism between varieties and let $S\to V\to Z$ be the Stein factorisation. We define the \emph{Stein degree} of $S$ over $Z$ to be $\sdeg(S/Z):=\deg(V/Z)$. If $S\to Z$ is not surjective, this degree is $0$ by convention. 

We are mainly interested in Stein degrees on log Calabi-Yau fibrations. Recall that a log Calabi-Yau fibration $(X,B)\to Z$ consists of an lc pair $(X,B)$ and a contraction $X\to Z$ such that $K_X+B\sim_\R 0/Z$.

\begin{thm}\label{t-bnd-stein-deg-lc-main}
Let $d\in \N$. Let $(X,B)\to Z$ be a log Calabi-Yau fibration of dimension $d=\dim X$. 
Then $\sdeg(I/Z)$ is bounded from above depending only on $d$ for every non-klt centre $I$ of $(X,B)$.
\end{thm}

This and related results in Section \ref{ss-stein-deg} form a crucial ingredient of the boundedness of slc stable minimal models. These are of independent interest and will likely find other applications. We conjecture a more general form of the theorem. 

\begin{conj}\label{conj-stein}
Let $d\in \N$ and let $t\in \R^{>0}$. 
Let $(X,B)\to Z$ be a log Calabi-Yau fibration of dimension $d$ and let $S$ be a horizontal$/Z$ component of $B$ whose coefficient in $B$ is $\ge t$. Then $\sdeg(S/Z)$ is bounded from above depending only on $d,t$.
\end{conj}

Note that the theorem is stated for non-klt centres $I$ which may not be divisors but the theorem can easily be reduced to the case when $I$ is a divisor. Moreover, we conjecture that the generalised version of the conjecture in the context of generalised pairs also holds where we replace the log Calabi-Yau fibration $(X,B)\to Z$ with a generalised log Calabi-Yau fibration $(X,B+M)\to Z$. On the other hand, we can ask a similar question when $S$ is not horizontal over $Z$ in which case we can ask whether $\sdeg(S/T)$ is bounded where $T$ is the image of $S$ in $Z$.
Yet another interesting question to consider is in the setting of slc log Calabi-Yau fibrations in which case we can also consider the Stein degree of the irreducible components of $X$ over $Z$.

We formulate a similar conjecture in the context of non-closed fields which is of interest in arithmetic geometry, e.g. over number fields.  
 
\begin{conj}
Let $d\in \N$ and let $t\in \R^{>0}$. Let $K$ be a field of characteristic zero.
Let $(X,B)$ be a log Calabi-Yau pair over $K$ of dimension $d$ and let $S$ be a component of $B$ whose coefficient in $B$ is $\ge t$. 
Then 
$$
\sdeg(S/\Spec K)=\dim_KH^0(S,\mathcal{O}_S)
$$ 
is bounded from above depending only on $d,t$.
\end{conj} 
 
Of course one can ask whether the same holds in positive characteristic.  
 
\vspace{0.5cm}
{\textbf{\sffamily{Acknowledgements.}}} 
This work was supported by grants from Tsinghua University. Thanks J\'anos Koll\'ar for answering a question about his book on moduli.

%%%%%%%%%%%%%%%%%%%%%%%
%%%%%%%%%%%%%%%%%%%%%%%

\section{\bf Preliminaries}

We work over an algebraically closed field $k$ of characteristic zero. All varieties and schemes are defined over $k$ unless stated otherwise.
Varieties are assumed to be irreducible. 

\subsection{Contractions}
By a \emph{contraction} we mean a projective morphism $f\colon X\to Y$ of schemes 
such that $f_*\mathcal{O}_X=\mathcal{O}_Y$ ($f$ is not necessarily birational). In particular, 
$f$ is surjective and has connected fibres.

\subsection{Divisors}
Let $X$ be a pure dimensional scheme of finite type over $k$ and let $M$ be an $\R$-divisor on $X$, i.e. a finite $\R$-linear combination of prime divisors. We denote the coefficient of a prime divisor $D$ in $M$ by $\mu_DM$.  
    
For a morphism $f\colon X\to Z$ of schemes and an $\R$-Cartier divisor $N$ on $Z$, we sometimes denote $f^*N$ by $N|_X$. An $\R$-Cartier divisor is a finite $\R$-linear combination of Cartier divisors.

For a birational map $X\bir X'$ (resp. $X\bir X''$)(resp. $X\bir X'''$)(resp. $X\bir Y$) of varieties 
whose inverse does not contract divisors, and for 
an $\R$-divisor $M$ on $X$, we usually denote the pushdown of $M$ to $X'$ (resp. $X''$)(resp. $X'''$)(resp. $Y$) 
by $M'$ (resp. $M''$)(resp. $M'''$)(resp. $M_Y$).

Recall that an $\R$-Cartier $\R$-divisor $M$ on a variety $X$ projective over $Z$ 
is said be big if we can write $M\sim_\R A+P$ where $A$ is ample$/Z$ and $P\ge 0$.
When $M$ is $\Q$-Cartier we can replace $\sim_\R$ with $\sim_\Q$ and assume $A,P$ are $\Q$-divisors.

For $l\in \N$ and for two $\Q$-divisors $M,N$ on a variety $X$, the notation $M\sim_lN$ means that $lM\sim lN$.

\subsection{Pairs and singularities}\label{ss-pairs}
A \emph{pair} $(X,B)$ consists of a normal quasi-projective variety $X$ and an $\R$-divisor $B\ge 0$ such that $K_X+B$ is $\R$-Cartier. We call $B$ the \emph{boundary divisor}.

Let $\phi\colon W\to X$ be a log resolution of a pair $(X,B)$. Let $K_W+B_W$ be the 
pullback of $K_X+B$. The \emph{log discrepancy} of a prime divisor $D$ on $W$ with respect to $(X,B)$ 
is defined as 
$$
a(D,X,B):=1-\mu_DB_W.
$$
A \emph{non-klt place} of $(X,B)$ is a prime divisor $D$ over $X$, that is, 
on birational models of $X$,  such that $a(D,X,B)\le 0$, and a \emph{non-klt centre} is the image of such a $D$ on $X$. 

We say $(X,B)$ is \emph{lc} (resp. \emph{klt})(resp. \emph{$\epsilon$-lc}) if 
 $a(D,X,B)$ is $\ge 0$ (resp. $>0$)(resp. $\ge \epsilon$) for every $D$. This means that  
every coefficient of $B_W$ is $\le 1$ (resp. $<1$)(resp. $\le 1-\epsilon$). 
Note that since $a(D,X,B)=1$ for most prime divisors, we necessarily have $\epsilon\le 1$.

A \emph{log smooth} pair is a pair $(X,B)$ where $X$ is smooth and $\Supp B$ has simple 
normal crossing singularities. Assume $(X,B)$ is a log smooth pair and assume $B=\sum_1^r B_i$ is reduced  
where $B_i$ are the irreducible components of $B$. 
A \emph{stratum} of $(X,B)$ is a component of $\bigcap_{i\in I}B_i$ for some $I\subseteq \{1,\dots,r\}$.  
Since $B$ is reduced, a stratum is nothing but an lc centre of $(X,B)$.

\subsection{Semi-log canonical pairs}\label{ss-slc-pairs}
A \emph{semi-pair} $(X,B)$ over a field $K$ of characteristic zero 
(not necessarily algebraically closed) consists of a reduced quasi-projective scheme $X$ of pure dimension and an $\R$-divisor $B\ge 0$ on $X$ satisfying the following conditions: 
\begin{itemize}
\item $X$ is $S_2$ with nodal codimension one singularities,

\item no component of $\Supp B$ is contained in the singular locus of $X$, and 

\item $K_X+B$ is $\R$-Cartier.
\end{itemize} 

We say that $(X,B)$ is \emph{semi-log canonical (slc)} if in addition we have:
\begin{itemize}
\item if $\pi\colon X^\nu\to X$ is the normalisation of $X$ and $B^\nu$ is the sum of the birational 
transform of $B$ and the conductor divisor of $\pi$, then $(X^\nu,B^\nu)$ is lc.
\end{itemize} 
By $(X^\nu,B^\nu)$ being lc we mean after passing to the algebraic closure of $K$, 
$(X^\nu,B^\nu)$ is an lc pair on each of its irreducible components.
The conductor divisor of $\pi$ is the sum of the prime divisors on $X^\nu$ whose 
images on $X$ are contained in the singular locus of $X$. It turns out that $K_{X^\nu}+B^\nu=\pi^*(K_X+B)$ 
for a suitable choice of $K_{X^\nu}$ in its linear equivalence class: 
to see this note that $X$ is Gorenstein outside a codimension $\ge 2$ closed subset, so shrinking $X$ 
we can assume it is Gorenstein and that $X^\nu$ is regular; in this case $B$ is $\R$-Cartier so 
we can remove it in which case the equality follows from 
\cite[5.7]{K13}. 

To simplify the language we usually say $(X,B)$ is an slc pair rather than an slc semi-pair. The word slc indicates that we are in the setting of semi-pairs rather than usual pairs defined in \ref{ss-pairs}. See  \cite[Chapter 5]{K13} for more on slc pairs.

\subsection{b-divisors}
A \emph{b-$\R$-Cartier b-divisor} over a variety $X$ is the choice of  
a projective birational morphism 
$Y\to X$ from a normal variety and an $\R$-Cartier $\R$-divisor $M$ on $Y$ up to the following equivalence: 
 another projective birational morphism $Y'\to X$ from a normal variety and an $\R$-Cartier $\R$-divisor
$M'$ define the same b-$\R$-Cartier  b-divisor if there is a common resolution $W\to Y$ and $W\to Y'$ 
on which the pullbacks of $M$ and $M'$ coincide.  

A b-$\R$-Cartier  b-divisor  represented by some $Y\to X$ and $M$ is \emph{b-Cartier} if  $M$ is 
b-Cartier, i.e. its pullback to some resolution is Cartier.

\subsection{Generalised pairs}\label{ss-gpp}
A \emph{generalised pair} consists of 
\begin{itemize}
\item a normal quasi-projective variety $X$ equipped with a projective
morphism $X\to Z$, 

\item an $\R$-divisor $B\ge 0$ on $X$, and 

\item a b-$\R$-Cartier  b-divisor over $X$ represented 
by some projective birational morphism $X' \overset{\phi}\to X$ and $\R$-Cartier $\R$-divisor
$M'$ on $X$
\end{itemize}
such that $M'$ is nef$/Z$ and $K_{X}+B+M$ is $\R$-Cartier,
where $M:= \phi_*M'$. 

We refer to $M'$ as the \emph{nef part} of the pair.  
Since a b-$\R$-Cartier b-divisor is defined birationally, 
in practice we will often replace $X'$ with a resolution and replace $M'$ with its pullback.
When $Z$ is a point we drop it but say the pair is projective. 

Now we define generalised singularities.
Replacing $X'$ we can assume $\phi$ is a log resolution of $(X,B)$. We can write 
$$
K_{X'}+B'+M'=\phi^*(K_{X}+B+M)
$$
for some uniquely determined $B'$. For a prime divisor $D$ on $X'$ the \emph{generalised log discrepancy} 
$a(D,X,B+M)$ is defined to be $1-\mu_DB'$. 

We say $(X,B+M)$ is 
\emph{generalised lc} (resp. \emph{generalised klt})(resp. \emph{generalised $\epsilon$-lc}) 
if for each $D$ the generalised log discrepancy $a(D,X,B+M)$ is $\ge 0$ (resp. $>0$)(resp. $\ge \epsilon$).

For the basic theory of generalised pairs see \cite[Section 4]{BZ16} and \cite{B20c}.

\subsection{Minimal models, Mori fibre spaces, and MMP}
Let $ X\to Z$ be a
projective morphism of normal quasi-projective varieties and $D$ be an $\R$-Cartier $\R$-divisor
on $X$. Let $Y$ be a normal quasi-projective variety, projective over $Z$, and $\phi\colon X\bir Y/Z$
be a birational map whose inverse does not contract any divisor. 
Assume $D_Y:=\phi_*D$ is also $\R$-Cartier and that 
there is a common resolution $g\colon W\to X$ and $h\colon W\to Y$ such that
$E:=g^*D-h^*D_Y$ is effective and exceptional$/Y$, and
$\Supp g_*E$ contains all the exceptional divisors of $\phi$.

Under the above assumptions we call $Y$ 
a \emph{minimal model} of $D$ over $Z$ if $D_Y$ is nef$/Z$.
On the other hand, we call $Y$ a \emph{Mori fibre space} of $D$ over $Z$ if there is an extremal contraction
$Y\to T/Z$ with $-D_Y$  ample$/T$ and $\dim Y>\dim T$.

If one can run a \emph{minimal model program} (MMP) on $D$ over $Z$ which terminates 
with a model $Y$, then $Y$ is either a minimal model or a Mori fibre space of 
$D$ over $Z$. If $X$ is a Mori dream space, 
eg if $X$ is of Fano type over $Z$, then such an MMP always exists by \cite{BCHM10}.

\subsection{Bounded families of pairs}\label{ss-bnd-couples}
We say a set $\mathcal{Q}$ of normal projective varieties is \emph{birationally bounded} (resp. \emph{bounded}) 
if there exist finitely many projective morphisms $V^i\to T^i$ of varieties  
such that for each $X\in \mathcal{Q}$ there exist an $i$, a closed point $t\in T^i$, and a 
birational isomorphism (resp. isomorphism) $\phi\colon V^i_t\bir X$  where $V_t^i$ is the fibre of $V^i\to T^i$ over $t$.

Next we will define boundedness for couples. 
A \emph{couple} $(X,S)$ consists of a normal projective variety $X$ and a  divisor 
$S$ on $X$ whose coefficients are all equal to $1$, i.e. $S$ is a reduced divisor. 
We use the term couple instead of pair because  
$K_X+S$ is not assumed $\Q$-Cartier and $(X,S)$ is not assumed to have good singularities.
 
We say that a set $\mathcal{P}$ of couples  is \emph{birationally bounded} if there exist 
finitely many projective morphisms $V^i\to T^i$ of varieties and reduced divisors $C^i$ on $V^i$ 
such that for each $(X,S)\in \mathcal{P}$ there exist an $i$, a closed point $t\in T^i$, and a 
birational isomorphism $\phi\colon V^i_t\bir X$ such that $(V^i_t,C^i_t)$ is a couple and 
$E\le C_t^i$ where $V_t^i$ and $C_t^i$ are the fibres over $t$ of the morphisms $V^i\to T^i$ 
and $C^i\to T^i$, respectively, and $E$ is the sum of the 
birational transform of $S$ and the reduced exceptional divisor of $\phi$.
We say $\mathcal{P}$ is \emph{bounded} if we can choose $\phi$ to be an isomorphism. 
  
A set $\mathcal{R}$ of projective pairs $(X,B)$ is said to be {log birationally bounded} (resp. log {bounded}) 
if the set of the corresponding couples $(X,\Supp B)$ is birationally bounded (resp. bounded).
Note that this does not put any condition on the coefficients of $B$, e.g. we are not requiring the 
coefficients of $B$ to be in a finite set.

\subsection{Bounded families of stable minimal models}\label{ss-bnd-s.min.models}
It is not hard to see that a set $\mathcal{P}$ of couples is bounded iff there is a fixed $r\in \N$ such that 
for any $(X,S)$ in $\mathcal{P}$ we can find a very ample divisor $H$ on $X$ so that 
$$
H^d\le r ~~\mbox{and}~~(K_X+S)\cdot H^{d-1}\le r.
$$
Although $K_X+S$ is not necessarily $\Q$-Cartier but the intersection number is well-defined. 

We take a similar approach with stable minimal models to define boundedness. We need to discuss boundedness only when the coefficients are not too small.  
That is, assume that $0$ is not an accummulation point of $\Phi\subset \Q^{\ge 0}$, e.g. when it is a DCC set. A subset $\mathcal{E}\subset \mathcal{S}_{\rm slc}(d,\Phi)$ is said to be a bounded family 
if there is a fixed $r\in \N$ such that 
for any $(X,B),A$ in $\mathcal{E}$ we can find a very ample divisor $H$ on $X$ so that 
$$
H^d\le r ~~\mbox{and}~~(K_X+B+A)\cdot H^{d-1}\le r.
$$
This in particular bounds $(X,\Supp(B+A))$ but also bounds the coefficients of $A$ from above.

\subsection{Volume and Iitaka volume of divisors}\label{ss-vol-Ivol}

Let $X$ be a projective scheme of dimension $d$ over a field $K$. Let $L$ be a nef $\R$-Cartier $\R$-divisor on $X$. The \emph{volume of $L$} is defined as $\vol(L)=L^d$. Note that the volume depends not only on $X,L$ but also on $K$ as $L^d$ is calculated with respect to $K$. 

When $L$ is ample and $\Q$-Cartier, applying the asymptotic Riemann-Roch theorem, for sufficiently divisible $m\ge 0$, we have 
$$
h^0(mL)=\mathcal{X}(mL)=\frac{L^d}{d!}m^{d}+\mbox{lower degree terms in $m$}
$$ 
from which one can derive 
$$
\vol(L)=\lim_{m\to \infty} \frac{h^0({mL})}{m^d/d!}
$$
where $m$ is sufficiently divisible. When $X$ is reduced and $Y$ is the union of its irreducible components of dimension $d$, then $\vol(L)=\vol(L|_Y)$.

Now assume that $L$ is semi-ample, that is, there is a contraction $f\colon X\to Z$ such that $L\sim_\R f^*A$ where $A$ is an ample $\R$-Cartier $\R$-divisor. Then we define the \emph{Iitaka volume of $L$} to be $\Ivol(L)=\vol(A)$. When $\dim Z= 0$, by convention we let $\Ivol(L)=1$.

\begin{lem}\label{l-vol-loc-constant}
Let $X\to S$ be a flat projective morphism of Noetherian schemes with fibres of dimension $d$ and let ${L}$ be a $\Q$-Cartier $\Q$-divisor on $X$ that is ample over $S$. Then $\vol(L|_{X_s})$ considered as a function on $S$ is locally constant. 
\end{lem}
\begin{proof}
It is enough to consider the case when $S$ is connected. In this case, for sufficiently divisible $m\ge 0$, the polynomial $\mathcal{X}(mL|_{X_s})$ is the same for every $s$ as $X\to S$ is flat, so $\vol(L|_{X_s})$ is a constant function by the above discussion.

\end{proof}

\begin{lem}\label{l-base-change-field}
Let $X$ be a projective scheme of dimension $d$ over a field $K$. Let $L$ be a $\Q$-Cartier $\Q$-divisor on $X$. Let $K\subset K'$ be a field extension, $X'=X\times_{\Spec K}\Spec K'$, and let $L'$ be the pullback of $L$ to $X'$. Then the following hold.
\begin{enumerate}
\item $L$ is ample iff $L'$ is ample. 

\item $L$ is semi-ample iff $L'$ is semi-ample. 

\item  If $L$ is ample, then $\vol(L')=\vol(L)$.

\item If $L$ is semi-ample, then $\Ivol(L')=\Ivol(L)$.

\item Assume $h^0(\mathcal{O}_X)=1$ and that $L$ is Cartier. Then $L\sim 0$ iff $L'\sim 0$.

\end{enumerate}
\end{lem}
\begin{proof}
(1) Assume $L$ is ample. Then $mL$ is very ample for some $m\in \N$, so there is a closed embedding $X\to \PP^n_K$ such that $\mathcal{O}_X(mL)\simeq \mathcal{O}_X(1)$ where the latter sheaf is the pullback of $\mathcal{O}_{\PP^n_K}(1)$. Base change to $K'$ gives a closed embedding $X'\to \PP^n_{K'}$ such that $\mathcal{O}_{X'}(mL')\simeq \mathcal{O}_{X'}(1)$ which implies $mL'$ is very ample, hence $L'$ is ample. 

Now assume $L'$ is ample. Let $\mathcal{F}$ be a coherent sheaf on $X$ and $\mathcal{F}'$ its pullback to $X'$. Then $h^i(\mathcal{F}'(mL'))=0$ for some $m\in \N$ and every $i>0$ by the cohomological characterisation of ampleness. Thus $h^i(\mathcal{F}(mL))=0$ for some $m\in \N$ and every $i>0$, hence $L$ is ample.

(2) 
For each sufficiently divisible $m\in \N$, we have  
$$
H^0(\mathcal{O}_{X'}(mL'))\simeq H^0(\mathcal{O}_X(mL))\times_KK'.
$$ 
Then a basis of $H^0(\mathcal{O}_X(mL))$ gives a basis of $H^0(\mathcal{O}_{X'}(mL'))$.
If $x'\in X'$ and if $x$ is its image on $X$, then there exists a section of $\mathcal{O}_X(mL)$ not vanishing at $x$ iff there exists a section of $\mathcal{O}_{X'}(mL')$ not vanishing at $x'$ (this follows by considering the elements of the mentioned bases). So $\mathcal{O}_X(mL)$ is generated by global sections iff $\mathcal{O}_{X'}(mL')$ is generated by global sections, hence $L$ is semi-ample iff $L'$ is semi-ample.

(3) 
This follows from 
$$
\vol(L')=\lim_{m\to \infty} \frac{h^0({mL'})}{m^d/d!}=\lim_{m\to \infty} \frac{h^0({mL})}{m^d/d!}=\vol(L)
$$
where $m$ is sufficiently divisible.

(4)
Since $L$ is semi-ample, it defines a contraction $f\colon X\to Z$ over $\Spec K$ and $L\sim_\Q f^*A$ for some ample $\Q$-Cartier $\Q$-divisor $A$ on $Z$. Here $f_*\mathcal{O}_X=\mathcal{O}_Z$. Let $Z'=Z\times_{\Spec K}\Spec K'$. Then $X'\simeq Z'\times_ZX$. Denote the induced morphism $X'\to Z'$ by $f'$. 
Since $\Spec K'\to \Spec K$ is flat, the induced morphism $Z'\to Z$ is flat. Therefore, by the flat base change theorem, ${f'}_*\mathcal{O}_{X'}=\mathcal{O}_{Z'}$, hence $f'$ is a contraction. Moreover, if $A'$ is the pullback of $A$ via $Z'\to Z$, then $L'\sim_\Q {f'}^*A'$. Thus $f'$ is just the contraction defined by $L'$. But then 
$$
\Ivol(L')=\vol(A')=\vol(A)=\Ivol(L).
$$ 

(5)
$L\sim 0$ means $\mathcal{O}_X(L)\simeq \mathcal{O}_X$ (similarly for $L'$). Clearly $L\sim 0$ implies $L'\sim 0$. Conversely, assume $L'\sim 0$. Then $\mathcal{O}_{X'}(L')$ is generated by global sections, so 
$\mathcal{O}_X(L)$ is also generated by global sections by the above discussions. Thus $L$ is the pullback of an ample divisor $A$ via some contraction $X\to Z$ over $\Spec K$. Since $L'\sim 0$, $Z$ is zero-dimensional, and since $h^0(\mathcal{O}_Z)=h^0(\mathcal{O}_X)=1$, $Z=\Spec K$, hence $A\sim 0$ which implies $L\sim 0$.

\end{proof}

\subsection{Generic semi-ampleness}

\begin{lem}\label{l-generic-semi-ampleness}
Let $X\to Z$ be a projective morphism of Noetherian schemes where $Z$ is integral, and let ${L}$ be a Cartier divisor on $X$. If ${L}$ is semi-ample on the generic fibre of $X\to Z$, then it is semi-ample over some non-empty open subset of $Z$.
\end{lem}
\begin{proof}
Let $\eta$ be the generic point of $Z$ and $X_\eta$ be the generic fibre. Replacing ${L}$ with some ${mL}$ we can assume that 
${L}|_{X_\eta}$ is base point free. We can assume $Z$ is affine, say $Z=\Spec A$. By flat base change of cohomology, 
$$
H^0({L})\otimes_Ak(\eta)\simeq H^0(L|_{X_\eta}).
$$  
Thus if $\alpha_1,\dots, \alpha_r$ generate $H^0({L})$ as an $A$-module, then the restrictions of $\alpha_1,\dots, \alpha_r$ to $X_\eta$ generate $H^0(L|_{X_\eta})$ as a $k(\eta)$-vector space. If $V$ is the set of points where all the $\alpha_i$ vanish, then $V\cap X_\eta=\emptyset$ because ${L}|_{X_\eta}$ is base point free. Therefore, ${L}$ is base point free on $X\setminus V$ which contains an open neighbourhood of the generic fibre as $V$ is a closed subset (with reduced structure). This means that $L$ is free over some non-empty open subset of $Z$.
 
\end{proof}

%%%%%%%%%%%%%%%%%%%%%%%%%%%%%%
%%%%%%%%%%%%%%%%%%%%%%%%%%%%%%

\section{\bf Stein degree on log Calabi-Yau fibrations}\label{ss-stein-deg}

In this section, we prove our main result Theorem \ref{t-bnd-stein-deg-lc-main} on the boundedness of Stein degrees on log Calabi-Yau fibrations. We prove other results which are used in the proof of boundedness of slc stable minimal models.

Recall that given a projective morphism $S\to Z$ between varieties, we let $S\to V\to Z$ be the Stein factorisation and then define the \emph{Stein degree} of $S$ over $Z$ to be $\sdeg(S/Z):=\deg(V/Z)$. If $S\to Z$ is not surjective, this degree is $0$ by convention. 

\begin{exa}
Note that Stein degree is not a birational invariant, that is, if $S'\bir S/Z$ is birational, then it may happen that $\sdeg(S'/Z)$ and $\sdeg(S/Z)$ are not equal (but equality holds if both $S',S$ are normal). For example, let $S\subset \PP^2\times \A^1$ be the 
surface defined by $u^2-tv^2$ where $u,v,w$ are the coordinates on $\PP^2$ and $t$ is the coordinate on $\A^1$. Let $S\to Z=\A^1$ be the projection. Then for each $a\neq 0$, the fibre $X_a$ is the union of two intersecting lines, hence connected. This shows that $\sdeg(S/Z)=1$. But if $S^\nu$ is the normalisation of $S$, then the general fibres of $S^\nu\to Z$ consist of two disjoint curves (the normalisation of $X_a$). Therefore, $\sdeg(S^\nu/Z)=2$.
\end{exa}

\subsection{Boundedness of Stein degree of some non-klt places}
We start with a lemma which helps to compare Stein degrees.

\begin{lem}\label{l-stein-deg-contractions}
Let $S\to Y\to Z$ be surjective projective morphisms of varieties. Then 
\begin{enumerate}
\item we have 
$$
\sdeg(S/Z)\ge \sdeg(Y/Z);
$$

\item if in addition $Y\to Z$ is a contraction with irreducible general fibres (e.g. $Y\to Z$ is a contraction and $Y$ is normal), then 
$$
\sdeg(S/Y)\ge \sdeg(S/Z).
$$
\end{enumerate}
\end{lem}
\begin{proof}
Let $S\to R\to Y$ be the Stein factoristion of $S\to Y$, let $Y\to T\to Z$ be the Stein factoristion of $Y\to Z$, and let $R\to V\to T$ be the Stein factoristion of $R\to T$. Then we get the following diagram 
$$
\xymatrix{
S\ar[rd]\ar[d]^f&&\\
R\ar[r]^a\ar[d]^g\ar[rd]&Y\ar[rd]\ar[d]^h&\\
V\ar[r]^b&T\ar[r]^c&Z
}
$$
where $f,g,h$ are contractions and $a,b,c$ are finite morphisms. In particular, $S\to V\to Z$ is the Stein factorisation of $S\to Z$. Then 
$$
\sdeg(S/Z)=\deg(V/Z)\ge \deg(T/Z)=\sdeg(Y/Z)
$$
which proves (1).

On the other hand, assume that $Y\to Z$ is a contraction with irreducible general fibres. Then in this case $T\to Z$ is an isomorphism.  Moreover, since the general fibres of $Y\to T$ are irreducible, the general fibres of $V\times_TY\to V$ are also irreducible, hence shrinking $T=Z$ we can assume that $V\times_TY$ is irreducible. But then since $R\to Y$ factors through $V\times_TY\to Y$, 
$$
\deg(R/Y)\ge \deg(V\times_TY/Y)=\deg(V/T).
$$
This in turn implies 
$$
\sdeg(S/Y)=\deg(R/Y)\ge \deg(V/T)=\deg(V/Z)=\sdeg(S/Z).
$$
\end{proof}

Recall that a log Calabi-Yau fibration $(X,B)\to Z$ consists of an lc pair $(X,B)$ and a contraction $X\to Z$ such that $K_X+B\sim_\R 0/Z$. Also recall that two log Calabi-Yau fibrations $(X,B)\to Z$ and $(X',B')\to Z$ are \emph{crepant birational} over $Z$ if there exist a birational map $X\bir X'/Z$ and a common resolution $W\to X$ and $W\to X'$ on which the pullbacks of $K_X+B$ and $K_{X'}+B'$ agree.

\begin{lem}\label{l-stein-deg-lc-bnd-comp}
Let $(X,B)\to Z$ be a log Calabi-Yau fibration of dimension $d=\dim X$. Assume that $(X,B)$ is not klt over the generic point $\eta_Z$. Then there exist a crepant birational model $(X',B')$ of $(X,B)$ over $Z$ and a component $S'$ of $\rddown{B'}$ such that 
\begin{itemize}
\item $(X',B')$ is dlt, 

\item $S'$ is horizontal over $Z$, and

\item the Stein degree $\sdeg(S'/Z)$ is bounded from above depending only on $d$. 
\end{itemize}
\end{lem}
\begin{proof}
First assume that $X\to Z$ factors as contractions $X\to Y\to Z$ and that $(X,B)$ is not klt over $\eta_Y$. Then it is enough to prove the lemma for $(X,B)\to Y$: indeed, 
assume that there exist a dlt crepant birational model $(X',B')$ of $(X,B)$ over $Y$ and a component $S'$ of $\rddown{B'}$ such that $S'$ is horizontal over $Y$ and that $\sdeg(S'/Y)$ is bounded from above. Then 
$(X',B')$ is also a crepant birational model of $(X,B)$ over $Z$, and 
$\sdeg(S'/Y)\ge \sdeg(S'/Z)$, by Lemma \ref{l-stein-deg-contractions} as $Y$ is automatically normal since $X$ is normal.

Now replacing $(X,B)$ with a $\Q$-factorial dlt model, we can assume that $(X,B)$ is $\Q$-factorial dlt and $\rddown{B}\neq 0$ over $\eta_Z$. Run an MMP on $K_X+B-\epsilon \rddown{B}$ over $Z$ for some small $\epsilon>0$. Since 
$$
K_X+B-\epsilon \rddown{B}\sim_\R -\epsilon \rddown{B}/Z
$$ 
is not pseudo-effective over $Z$, the MMP ends with a Mori fibre space. Replacing $X$ with the Mori fibre space, we can assume that we have a contraction $X\to Y/Z$ such that $\rddown{B}$ is big over $Y$. By the first paragraph, we can replace $Z$ with $Y$, hence assume  $\rddown{B}$ is big over $Z$. In the process we may have lost the dlt property of $(X,B)$. Let $(X'',B'')$ be a $\Q$-factorial dlt model of $(X,B)$ and let $\alpha$ denote $X''\to X$. Then $\alpha^*\rddown{B}$ is big over $Z$. Since $\Supp \alpha^*\rddown{B}\subset \rddown{B''}$, we deduce that $\rddown{B''}$ is big over $Z$. Thus  we can replace $(X,B)$ with $(X'',B'')$, hence again assume $(X,B)$ is $\Q$-factorial dlt and $\rddown{B}$ is big over $Z$. 

Assume that $B\neq \rddown{B}$ over $\eta_Z$. Run an MMP on 
$$
K_X+B-\epsilon(B-\rddown{B})
$$ 
over $Z$ for some small $\epsilon>0$. The MMP ends with a Mori fibre space $X'\to Y'$. Since $\rddown{B}$ is big over $Z$, $\rddown{B'}$ is big over $Z$, hence $\rddown{B'}$ is big over $Y'$ which shows that some component $S'$ of $\rddown{B'}$ is ample over $Y'$. Moreover, $(X',S')$ is plt because 
$$
(X',B'-\epsilon(B'-\rddown{B'}))
$$ 
is dlt and 
$$
B'-\epsilon(B'-\rddown{B'})\ge \rddown{B'}\ge S'.
$$ 
So the non-klt locus of $(X',S')$ is just $S'$. Moreover, by construction, $B'-\rddown{B'}$ is ample over $Y'$, so $K_{X'}+S'$ is anti-ample over $Y'$. 
Thus the fibres of $S'\to Y'$ are connected by the connected principle (cf. \cite{B20b} or \cite{FS20}). This combined with Lemma \ref{l-stein-deg-contractions} implies that  
$$
\sdeg(S'/Z)\le \sdeg(S'/Y')=1.
$$ 
Since $S'$ is normal, replacing $(X',B')$ with a $\Q$-factorial dlt model would not change $\sdeg(S'/Z)$.

Now assume that $B=\rddown{B}$ over $\eta_Z$. Pick a sufficiently small $\delta>0$ depending on $d$. 
We claim that $X$ is $\delta$-lc over $\eta_Z$.  
Let $(F,B_F)$ be a general log fibre of $(X,B)\to Z$. Then $(F,B_F)$ is lc, $(F,0)$ is klt, $K_F+B_F\sim_\R 0$, and each non-zero coefficient of $B_F$ is $1$.
By \cite[Lemma 2.48]{B19}, we can assume that if $D$ is a prime divisor over $F$ with $a(D,F,B_F)< \delta$, then $a(D,F,B_F)=0$. Assume $E$ is a 
prime divisor over $X$ which is horizontal over $Z$ and with $a(E,X,0)< \delta$.  
Then $E$ gives a prime divisor $D$ over $F$ so that 
$$
a(D,F,B_F)=a(E,X,B)<\delta,
$$ 
hence 
$$
0=a(D,F,B_F)=a(E,X,B).
$$ 
But since $(X,B)$ is dlt, the generic point of the centre of $E$ on $X$ is contained in the log smooth locus of $(X,B)$, so $a(E,X,0)\ge 1$ which shows that such $E$ do not exist. 
%Therefore, extracting all such $E$ 
%we can assume that $X$ has $\delta$-lc singularities over $\eta_Z$ and that the condition $B=\rddown{B}$ over $\eta_Z$ is preserved. Moreover, $\rddown{B}$ is still big over $Z$. 

Next, run an MMP on $K_X$ over $Z$ which ends with a Mori fibre space $X'\to Y'$ as $B=\rddown{B}\neq 0$  over $\eta_Z$. 
Then $X'$ is $\delta$-lc as $X$ is $\delta$-lc, and $\rddown{B'}$ is big over $Y'$ as  $\rddown{B}$ is big over $Z$. Thus $(X',B')$ is not klt over $\eta_{Y'}$. Applying the first paragraph, we can replace $(X,B)\to Z$ with $(X',B')\to Y'$, hence assume that the general fibres $F$ of $X\to Z$ are $\delta$-lc Fano varieties and that $B=\rddown{B}$ (but $(X,B)$ may no longer be dlt). 

In particular, the general fibres $F$ of $X\to Z$ belong to a bounded family, by \cite[Theorem 1.1]{B21a}, and this in turn implies that the general log fibres $(F,B_F)$ belong to a bounded family because $K_F+B_F\sim_\R 0$ and because $B_F$ is reduced. 
Pick a horizontal$/Z$ component $S$ of $\rddown{B}$. Then the general fibres of $S\to Z$ are the intersections $S\cap F$. On the other hand, viewing $S\cap F$ as a subset of $F$, we have $S\cap F=S|_F\subset \Supp B_F$. 
Since $(F,B_F)$ belongs to a bounded family, the number of the irreducible components of $B_F$ is bounded, hence the number of the irreducible components of $S|_F$ is bounded. Thus 
the number of the irreducible components of the general fibres of $S\to Z$ is bounded. 

Let $\pi\colon S^\nu\to S$ be the normalisation of $S$ and let $V^\nu\subset S^\nu$ be the largest closed subset such that $\pi$ is an isomorphism outside $V^\nu$. Then the general fibres of $S^\nu\to Z$ are of pure dimension $\dim S^\nu-\dim Z$. Thus no irreducible component of such fibres is contained in $V^\nu$ as the general fibres of $V^\nu\to Z$ have smaller dimension.
This implies that a general fibre of $S^\nu\to Z$ is birational to a general fibre of $S\to Z$. Therefore, the number of irreducible components of the general fibres of $S^\nu\to Z$ is bounded. 
The latter number is exactly $\sdeg(S^\nu/Z)$, hence $\sdeg(S^\nu/Z)$ is bounded. 

Now it is enough to take a $\Q$-factorial dlt model $(X',B')$ of $(X,B)$ and let $S'$ be the birational transform of $S$. Then $S'$ is birational to $S^\nu$ and they are both normal, hence $\sdeg(S'/Z)=\sdeg(S^\nu/Z)$ is bounded as desired.    
   
\end{proof}

\subsection{Comparing Stein degrees of non-klt centres}

Let $(X,B)$ be a pair and $X\to Z$ be a contraction. By a \emph{horizontal$/Z$ non-klt centre} of $(X,B)$ we mean a non-klt centre which is horizontal over $Z$, and by a \emph{minimal horizontal$/Z$ non-klt centre} of $(X,B)$ we mean a horizontal$/Z$ non-klt centre which is minimal with respect to inclusion among the horizontal non-klt centres.

\begin{lem}\label{l-stein-deg-dlt}
Let $(X,B)\to Z$ be a log Calabi-Yau fibration where $(X,B)$ is dlt. 
Then 
\begin{enumerate}
\item  if $I,J$ are minimal horizontal$/Z$ non-klt centres of $(X,B)$, then there is a birational map $I\bir J/Z$ and 
$$
\sdeg(I/Z)=\sdeg(J/Z);
$$

\item if $I,J$ are horizontal$/Z$ non-klt centres of $(X,B)$ and if $I$ is minimal, then 
$$
\sdeg(I/Z)\ge \sdeg(J/Z).
$$
\end{enumerate}
\end{lem}
\begin{proof}
We prove the lemma inductively, so assume it holds in lower dimension. 
We can shrink $Z$ hence assume that all the non-klt centres are horizontal over $Z$. 
We show that part (2) follows from part (1). Indeed, assume that $I,J$ are non-klt centres and that $I$ is minimal. Pick a minimal non-klt centre $L\subseteq J$. Let $J\to V\to Z$ be the Stein factorisation of $J\to Z$. And let $L\to R\to V$ be the Stein factorisation of the induced morphism $L\to V$. Then $L\to R\to Z$ is the Stein factorisation of $L\to Z$, so 
$$
\sdeg(I/Z)=\sdeg(L/Z)=\deg(R/Z)\ge \deg(V/Z)=\sdeg(J/Z) 
$$ 
where the first equality holds by part (1).

We prove part (1). We use arguments somewhat similar to \cite[Theorem 4.40]{K13}.
Since $(X,B)$ is dlt, the non-klt locus $\Nklt(X,B)=\rddown{B}$.
First assume that $I,J$ are contained in the same connected component of $\rddown{B}$.
Then we can find components $S_1,\dots,S_r$ of $\rddown{B}$ such that $I\subset S_1$, $J\subset S_r$, and $S_j\cap S_{j+1}\neq \emptyset$ for each $1\le j<r$. Pick a minimal non-klt centre $L_1$ of $(X,B)$ inside $S_1\cap S_2$. Define 
$$
K_{S_1}+B_{S_1}=(K_X+B)|_{S_1}
$$ 
by adjunction. Then $(S_1,B_{S_1})$ is dlt. 

Let $S_1\to T_1\to Z$ be the Stein factorisation of $S_1\to Z$. Then $(S_1,B_{S_1})\to T_1$ is a log Calabi-Yau fibration and $I$ and $L_1$ are minimal non-klt centres of $(S_1,B_{S_1})$ which are horizontal over $T_1$. Therefore, by induction on dimension, there is a birational map $I\bir L_1/T_1$ and $\sdeg(I/T_1)=\sdeg(L_1/T_1)$. This implies that 
$$
\sdeg(I/Z)=\sdeg(I/T_1)\deg(T_1/Z)=\sdeg(L_1/T_1)\deg(T_1/Z)=\sdeg(L_1/Z).
$$ 
Next consider a minimal non-klt centre $L_2$ of $(X,B)$ inside $S_2\cap S_3$ and apply induction on $S_2$ to show that there is a birational map $L_1\bir L_2/Z$ and that $\sdeg(L_1/Z)=\sdeg(L_2/Z)$. 
Repeating the argument we eventually prove that there exist minimal non-klt centres $L_i$ of $(X,B)$ and birational maps 
$$
I\bir L_1\bir L_2 \bir \cdots \bir L_{r-1}\bir L_r\bir J/Z
$$
and equalities 
$$
\sdeg(I/Z)=\sdeg(L_1/Z)=\cdots =\sdeg(L_r/Z)=\sdeg(J/Z)
$$ 
as claimed.

 Now assume that $I,J$ are contained in different connected components of $\Nklt(X,B)$. Then we show that there is a birational map $I\bir J/Z$ and that 
$$
\sdeg(I/Z)=\sdeg(J/Z)=1.
$$ 
Since $\Nklt(X,B)$ is not connected, by \cite{B20b} or \cite{FS20}, $(X,B)$ is plt, $\Nklt(X,B)=\rddown{B}$ has exactly two connected components   which are both irreducible components of $\rddown{B}$, say $S,T$, and there is a birational map $S\bir T/Z$. We give a sketch of the argument for convenience (see \cite[Theorem 3.5 and its proof]{B20b} for more details). Since $(X,B)$ is dlt, there is a $\Q$-factorial dlt  model $(X'',B'')$ of $(X,B)$ where $X''\to X$ is small and it is an isomorphism near the generic point of each non-klt centre. Then one runs an MMP on $K_{X''}+B''-\epsilon \rddown{B''}$ over $Z$ for a small $\epsilon>0$ which produces a Mori fibre space $X'\to Y'$. One then argues that the MMP does not contract any connected component of the non-klt locus, and that no two connected components can merge, by the connectedness principle in the birational setting. In particular, $\Nklt(X',B')=\rddown{B'}$ is not connected. From this one can easily deduce that every irreducible component of $\rddown{B'}$ is horizontal$/Y'$, that the general fibres of $X'\to Y'$ are $\PP^1$, and that $\rddown{B'}$ consists of two disjoint prime divisors $S',T'$ and that both $S'\to Y'$ and $T'\to Y'$ are birational. So we get an induced birational map $S'\bir T'/Z$.

Moreover, one shows that $(X',B')$ is plt with $S',T'$ its only non-klt centres as follows: assume this is not the case;  any non-klt centre other than $S',T'$ is vertical over $Y'$ and is of codimension $\ge 2$ in $X'$; extract a non-klt place $R'''$ say via $X'''\to X'$, whose centre is not $S',T'$; then $X'''$ is of Fano type over $Y'$, so we can run an MMP on $-R'''$ over $Y'$ which ends with a good minimal model $X''''$ of $-R'''$ as $R'''$ is vertical over $Y'$; say $X''''\to V'/Y'$ is the contraction defined by $-R''''$; assume the centre of $R'''$ on $X'$ is inside $S'$; then one argues that in the course of the MMP, $R'''$ remains disjoint from the birational transform of $T'$; one then gets a contradiction by arguing that the birational transform $T''''$ dominates $V'$, hence it intersects $R''''$ because $R''''$ is the pullback of some divisor on $V'$.  
 
  Let $S,T\subset X$ be the birational transforms of $S',T'$. Then $S'\bir T'$ induces a birational map $S\bir T$ over $Z$. By the previous paragraph, $(X,B)$ is plt and its only non-klt centres are $S,T$. But then perhaps switching $S,T$, we have $I=S$ and $J=T$. Then since $S,T,S',T',Y'$ are all normal, $S'\to Y'$ and $T'\to Y'$ are birational, and $Y'\to Z$ is a contraction, we have    
$$
\sdeg(I/Z)=\sdeg(S/Z)=\sdeg(S'/Z)=\sdeg(Y'/Z)=1
$$
and similarly 
$$
\sdeg(J/Z)=\sdeg(T/Z)=\sdeg(T'/Z)=\sdeg(Y'/Z)=1.
$$

\end{proof}

\subsection{Stein degree across crepant birational models}
Next we compare Stein degrees of non-klt centres on two crepant birational pairs.

\begin{lem}\label{l-stein-deg-dlt-crep}
Let $(X_1,B_1)\to Z$ and $(X_2,B_2)\to Z$ be crepant birational log Calabi-Yau fibrations where $(X_j,B_j)$ are dlt. Assume $I_j$ is a minimal horizontal$/Z$ non-klt centre of $(X_j,B_j)$. Then there 
is a birational map $I_1\bir I_2/Z$ and 
$$
\sdeg(I_1/Z)=\sdeg(I_2/Z).
$$ 
\end{lem}

\begin{proof} 
We prove the lemma by induction on dimension, so assume it holds in lower dimension.
We can shrink $Z$ hence assume that all the non-klt centres are horizontal over $Z$. 
First assume that there is a component $S_1$ of $\rddown{B_1}$ which is not exceptional over $X_2$, hence it has a birational transform $S_2$.  Define 
$$
K_{S_1}+B_{S_1}=(K_{X_1}+B_1)|_{S_1} ~~\mbox{and}~~K_{S_2}+B_{S_2}=(K_{X_2}+B_2)|_{S_2} 
$$  
by adjunction. Letting $S_1\to T_1\to Z$ be the Stein factorisation of $S_1\to Z$, 
$({S_1},B_{S_1})$ and $({S_2},B_{S_2})$ are crepant birational dlt log Calabi-Yau fibrations over $T_1$. 

Pick a minimal non-klt centre $J_1\subset S_1$ of $(X_1,B_1)$ and a minimal non-klt centre $J_2\subset S_2$ of $(X_2,B_2)$. Then by induction, 
there is a birational map $J_1\bir J_2/Z$ and 
$$
\sdeg(J_1/T_1)=\sdeg(J_2/T_1)
$$ 
which in turn implies 
$$
\sdeg(J_1/Z)=\sdeg(J_2/Z).
$$ 
On the other hand, both $I_1,J_1$ are minimal non-klt centres of $(X_1,B_1)$ and both $I_2,J_2$ are minimal non-klt centres of $(X_2,B_2)$. Therefore, applying Lemma \ref{l-stein-deg-dlt}, there exist birational maps $I_1\bir J_1/Z$ and $I_2\bir J_2/Z$. Thus we have birational maps 
$$
I_1\bir J_1\bir J_2 \bir I_2/Z
$$ 
and we have  
$$
\sdeg(I_1/Z)=\sdeg(J_1/Z)=\sdeg(J_2/Z)=\sdeg(I_2/Z).
$$

Now assume that every component of $\rddown{B_1}$ is exceptional over $X_2$. There is a dlt model $(X_2',B_2')$ of $(X_2,B_2)$ so that some component of $\rddown{B_1}$ is not exceptional over $X_2'$. 
On the other hand, there is also a component of $\rddown{B_2'}$ which is not exceptional over $X_2$.
Thus if $L_2'$ is a minimal non-klt centre of $(X_2',B_2')$, then applying the arguments above to both 
$$
(X_1,B_1)\bir (X_2',B_2')~~\mbox{and}~~ (X_2',B_2')\to (X_2,B_2),
$$ 
we deduce that there exist birational maps $I_1\bir L_2'/Z$ and $L_2'\bir I_2/Z$ which induce a birational map $I_1\bir I_2/Z$, and that   
$$
\sdeg(I_1/Z)=\sdeg(L_2'/Z)=\sdeg(I_2/Z).
$$

\end{proof}

\subsection{Proof of main result}

\begin{proof}[Proof of Theorem \ref{t-bnd-stein-deg-lc-main}]
We can assume that $(X,B)$ is not klt over $\eta_Z$.
We can shrink $Z$ hence assume that all the non-klt centres are horizontal over $Z$ (note that vertical centres have Stein degree $0$). 
By Lemma \ref{l-stein-deg-lc-bnd-comp}, there exist a dlt crepant birational model $(X',B')$ of $(X,B)$ over $Z$ and a normal component $S'$ of $\rddown{B'}$ such that $S'$ is horizontal over $Z$ and  $\sdeg(S'/Z)$ is bounded from above depending only on $d$. 

Let $S'\to T\to Z$ be the Stein factorisation of $S'\to Z$. Then $\deg(T/Z)$ is bounded. By adjunction define 
$$
K_{S'}+B_{S'}=(K_{X'}+B')|_{S'}
$$
and let $J'$ be a minimal non-klt centre of $(S',B_{S'})$. Then $(S',B_{S'})\to T$ is a log Calabi-Yau fibration, so by induction, $\sdeg(J'/T)$ is bounded from above. Therefore, $\sdeg(J'/Z)$ is also bounded.  

On the other hand,
let $(X'',B'')$ be a dlt model of $(X,B)$.  
For any non-klt centre $I$ of $(X,B)$, there is a non-klt centre $I''$ of $(X'',B'')$ mapping onto $I$. 
By Lemma \ref{l-stein-deg-contractions}, 
$$
\sdeg(I''/Z)\ge \sdeg(I/Z),
$$ 
so it is enough to bound $\sdeg(I''/Z)$. If $L''\subset I''$ is a minimal non-klt centre of $(X'',B'')$, then 
$$
\sdeg(I''/Z)\le \sdeg(L''/Z)=\sdeg(J'/Z),
$$ 
where the inequality holds by Lemma \ref{l-stein-deg-dlt}, and the equality holds by Lemma \ref{l-stein-deg-dlt-crep} as $(X'',B'')$ and $(X',B')$ are dlt crepant birational pairs and $I'',J'$ are their respective minimal non-klt centres. Therefore, $\sdeg(I''/Z)$ is bounded depending only on $d$.

\end{proof}

%%%%%%%%%%%%%%%%%%%%%%%
%%%%%%%%%%%%%%%%%%%%%%%

\section{\bf Lc stable minimal models}

In this section, we treat boundedness of stable minimal models in the lc case. The proof of this case is easier than that of the general case of slc stable minimal models (Theorem \ref{t-bnd-s-mmodels-slc}), so it is worth treating it separately. Also  
the arguments in this section can be regarded as preparation for the next two sections.

\begin{thm}\label{t-bnd-s-mmodels-lc}
Let $d\in\N$, $\Phi\subset \Q^{\ge 0}$ be a DCC set, $\Gamma\subset \Q^{>0}$ be a finite set, and $\sigma\in\Q[t]$ be a polynomial. Then $\mathcal{S}_{\rm lc}(d,\Phi,\Gamma,\sigma)$ is a 
bounded family.
\end{thm}

A key ingredient of the proof of this theorem and also the general case (Theorem \ref{t-bnd-s-mmodels-slc}) is the theory of generalised pairs. We start with recalling some relevant notions and results about generalised pairs.

\subsection{Boundedness of generalised pairs}\label{ss-fam-gen-pairs}

Let $d\in \mathbb{N}$, $\Phi\subset \mathbb{R}^{\ge 0}$, and $v\in \R^{>0}$. 
Let $\mathcal{G}_{glc}(d,\Phi)$ be the set of projective 
generalised pairs $(X,B+M)$ with nef part $M'$ such that 
\begin{itemize}
\item $(X,B+M)$ is generalised lc of dimension $d$, 
\item the coefficients of $B$ are in $\Phi$, 
\item $M'=\sum \mu_i M_i'$ where $\mu_i\in \Phi$ and $M_i'$ are nef Cartier, and
\item $K_X+B+M$ is big.
\end{itemize}
Let 
$\mathcal{F}_{glc}(d,\Phi,v)$ be the set of those $(X,B+M)\in \mathcal{G}_{glc}(d,\Phi)$
such that
\begin{itemize}
\item $K_X+B+M$ is ample with volume $\vol(K_X+B+M)=v$.
\end{itemize}
We define $\mathcal{F}_{gklt}(d,\Phi,v)$ similarly by replacing the generalised lc condition with generalised klt.

In birational geometry, one is frequently faced with questions regarding boundedness of generalised pairs. 
The generalised klt case is settled in \cite{B21b}. 

\begin{thm}[{\cite[Theorem 1.4]{B21b}}]\label{t-bnd-gen-klt-pairs}
Let $d\in \N$, $\Phi\subset \R^{\ge 0}$ be a DCC set, and $v\in \R^{>0}$. 
Then the set
$
\mathcal{F}_{gklt}(d,\Phi,v)
$ 
forms a bounded family.
\end{thm}  

Here by boundedness of $\mathcal{F}_{gklt}(d,\Phi,v)$ we mean that there is a fixed $r\in \N$ such that for each $(X,B+M)$ in $\mathcal{F}_{gklt}(d,\Phi,v)$, we can find a very ample divisor $H$ on $X$ with 
$$
H^d\le r ~~\mbox{and}~~(K_X+B+M)\cdot H^{d-1}\le r. 
$$
Boundedness of other sets of generalised pairs is defined similarly.

However, in the theorem, if $\Phi\subset \Q^{\ge 0}$ and $v\in \Q^{>0}$, then the proof of the theorem shows that there is a fixed $l\in\N$ such that $l(K_X+B+M)$ is very ample
for each 
$$
(X,B+M)\in \mathcal{F}_{gklt}(d,\Phi,v).
$$ 
Indeed, in the course of the proof, a divisor $\Theta$ is constructed such that $(X,\Theta)$ is klt, the coefficients of $\Theta$ belongs to a fixed DCC set, 
$$
l(K_X+\Theta)\sim l(1+t)(K_X+B+M)
$$
where $l\in \N$ and $t\in \Q^{>0}$ are fixed, and $(X,\Theta)$ is log birationally bounded. In particular, $\vol(K_X+\Theta)$ is fixed. But then $(X,\Theta)$ belongs to a bounded family by \cite[Theorem 1.6]{HMX14}. Thus we can replace $l$ so that $l(K_X+\Theta)$ and $l(K_X+B+M)$ are very ample.

On the other hand, it turns out that $\mathcal{F}_{glc}(d,\Phi,v)$ is not a bounded family in general \cite[Subsection 5.3]{BH22}. However, the next result says that the generalised pairs that are given by adjunction for fibrations form bounded families under natural conditions, which is the kind of boundedness we need.

First we introduce some notation. Let $d\in \mathbb N$, $\Phi\subset \R^{\ge 0}$, and $u,v\in \R^{>0}$.
A $(d,\Phi,u,v)$-stable minimal model is a stable minimal model $(X,B),A\to Z$ such that    
\begin{itemize}
\item $\dim X=d$,
\item the coefficients of $B$ and $A$ are in $\Phi$, 
\item $\vol(A|_F)=u$ where $F$ is any general fibre of $f\colon X\to Z$ over any irreducible component of $Z$, and 
\item $\Ivol(K_X+B)=v$.
\end{itemize}
The set of all the $(d,\Phi,u,v)$-stable minimal models is denoted $\mathcal{S}_{\rm slc}(d,\Phi,u,v)$. 
Similarly, $\mathcal{S}_{\rm lc}(d,\Phi,u,v)$ denotes the subset consisting of the lc stable minimal models. The notation is different from but consistent with that defined \ref{d-stabl-mmodels-II}.

\begin{thm}[{\cite[Theorem 1.3]{BH22}}]\label{t-bnd-base-adjunction} 
Let $d\in \mathbb N$, $\Phi\subset \Q^{\ge 0}$ be a DCC set, and $u,v\in \Q^{>0}$. 
Then there exists $l\in \N$ such that for any lc stable minimal model 
$$
(X,B),A\overset{f}\to Z\in \mathcal{S}_{lc}(d,\Phi,u,v),
$$
we can write an adjunction formula 
$$
K_X+B\sim_l f^*(K_Z+B_Z+M_Z)
$$ 
such that the corresponding set of generalized pairs $(Z,B_Z+M_Z)$ forms a bounded family. 
Moreover, $l(K_Z+B_Z+M_Z)$ is very ample.
\end{thm}

The last claim in the theorem is not explicitly stated in \cite[Theorem 1.3]{BH22} but it follows from its proof. We include some explanations for convenience. Note also that the theorem is proved in the more general setting that one only requires $A$ to be ample over the generic point $\eta_Z$ of $Z$ and that $A$ does not contain non-klt centres of $(X,B)$ over $\eta_Z$. This allows more flexibility in the proof.

First, one shows that there exist $q\in \N$ and a DCC set $\Psi\subset \Q^{\ge 0}$ such that the adjuction formula can be written as 
$$
K_X+B\sim_q f^*(K_Z+B_Z+M_Z)
$$
where $qM_{Z'}$ is Cartier for a high resolution $Z'\to Z$, and 
$$
(Z,B_Z+M_Z)\in \mathcal{F}_{glc}(\dim Z,\Psi,v).
$$ 

Assume that $(X,B)$ is klt over the generic point $\eta_Z$ of $Z$. Then the proof of this case goes via finding a bounded log smooth dlt pair $(V,\Gamma_V)$ so that $Z$ is the lc model of $(V,\Gamma_V)$ and $K_Z+B_Z+M_Z$ is the corresponding ample divisor on $Z$. More precisely, there is an equality of algebras 
$$
R(q(K_X+B))= R(q(K_V+\Gamma_V)).
$$
Moreover,  the coefficients of $\Gamma_V$ belong to a fixed finite set and $(V,\Gamma_V)$ has a good minimal model. Thus in this case boundedness of $Z$ and the existence of $l$ follow from \cite[Theorem 1.2]{HMX18}. 

Now assume $(X,B)$ is not klt over $\eta_Z$. In this case, one first reduces to the sitution where $(X,B)$ is dlt over $\eta_Z$. Let $S$ be the normalisation of a non-klt centre of $(X,B)$ which maps onto $Z$ and which is minimal with respect to inclusion among such centres, and let $g\colon S\to T$ be the contraction given by the Stein foctorisation of $S\to Z$. Using the dlt condition, there is a divisorial adjunction formula $K_S+B_S=(K_X+B)|_S$ where $(S,B_S)$ is klt over $\eta_T$.  
In particular, 
$$
K_S+B_S\sim_qg^*\pi^*(K_Z+B_Z+M_Z)
$$
where $\pi$ denotes $T\to Z$. 
Letting $A_S=A|_S$, $(S,B_S),A_S\to T$ satisfies the same kind of properties as in \cite[Theorem 1.3]{BH22}. Applying induction we can assume that we have an adjunction formula
$$
K_S+B_S\sim_q g^*(K_T+B_T+M_T)
$$
so that $l(K_T+B_T+M_T)$ is very ample for some fixed $l$. Then 
$$
q\pi^*(K_Z+B_Z+M_Z)\sim q(K_T+B_T+M_T),
$$
hence replacing $l$ by $ql$ we deduce that $l\pi^*(K_Z+B_Z+M_Z)$ is very ample.

The proof of the theorem then goes via constructing a generically Galois finite cover $V\to Z$ with  Galois group $G$ of bounded size such that $V\to Z$ factors through $S\to Z$. Denoting $V\to S$ by $\mu$, we see that 
$$
\Sigma_{h\in G}h^*\mu^*\pi^*l(K_Z+B_Z+M_Z)
$$  
 is the pullback of an ample Cartier divisor on $Z$. But 
$$
\mu^*\pi^*(K_Z+B_Z+M_Z)
$$ 
is $G$-invariant, so the above is just 
$$
l|G|\mu^*\pi^*(K_Z+B_Z+M_Z)
$$ 
implying that 
$$
l|G|(K_Z+B_Z+M_Z)
$$ 
is Cartier. Moreover, using the Galois cover one also shows that 
the linear system of the latter divisor 
is base point free defining a birational contraction. This implies that replacing $l$ with a bounded multiple we can assume that $l(K_Z+B_Z+M_Z)$ is very ample as desired.

\subsection{Bigness and lc thresholds}
We make some further preparations for the proof of \ref{t-bnd-s-mmodels-lc}.

\begin{lem}\label{l-v.ample-div-dominating}
Let $d,r\in\N$ and let $\Phi\subset \Q^{\ge 0}$ be a DCC set. Then there 
exists $l\in \N$ satisfying the following. Assume 
\begin{itemize}
\item $X$ is a normal projective variety of dimension $d$, 
\item $H$ is a very ample divisor, 
\item $A$ is a divisor with coefficients in $\Phi$, and
\item $H^d\le r$ and $A\cdot H^{d-1}\le r$.  
\end{itemize}
Then $lH-A$ is big. 
\end{lem}
\begin{proof}
If the lemma does not hold, then there is a sequence $X_i,H_i,A_i$ of varieties and divisors as in the lemma such that if $l_i\in \N$ is the smallest number so that $l_iH_i-A_i$ is big, then the $l_i$ form a strictly increasing sequence. Since $H_i^d\le r$ and $A_i\cdot H_i^{d-1}\le r$ and since the coefficients of $A_i$ are in the DCC set $\Phi$, the $(X_i,S_i:=\Supp A_i)$ belong to a bounded family, the cefficients of $A_i$ are bounded from above, and the number of components of $A_i$ is bounded. So $A_i\le pS_i$ for some fixed $p$. Thus 
replacing $A_i$ with $S_i$ and replacing $\Phi$ with $\{1\}$, it is enough to treat the case when $A_i$ is reduced. 

Replacing the sequence we can assume that there is a projective morphism $V\to T$ of normal varieties and a reduced divisor $C$ on $V$ such that for each $i$ there is a closed point $t_i\in T$ so that $X_i$ is isomorphic to the fibre $V_{t_i}$ and identifying $X_i$ and $V_{t_i}$ we have $A_i\le C_{t_i}$. Moreover, we can assume that there is an ample$/T$ divisor $L$ on $V$ such that $H_i\sim L|_{X_i}$. 

Now replacing $L$ and replacing $H_i$ accordingly we can assume that $L\sim C+D/T$ where $D\ge 0$. Replacing the sequence we can assume that $\Supp D$ does not contain any of the $X_i=V_{t_i}$. But then 
$$
H_i\sim L|_{X_i}\sim C|_{X_i}+D|_{X_i}.
$$ 
Note that $C,D$ are not necessarily $\Q$-Cartier but their restrictions to general fibres of $V\to T$ are well-defined as Weil divisors and we can assume $X_i$ are general fibres.
Then since $A_i\le C|_{X_i}=C_{t_i}$, we have $H_i\sim A_i+G_i$ for some $G_i\ge D|_{X_i}\ge 0$. Thus $2H_i-A_i$ is big. This is a contradiction.
  
\end{proof}

\begin{lem}\label{l-bnd-lct-on-bnd-family}
Let $d\in\N$, $\Phi\subset \Q^{\ge 0}$ be a DCC set, $u,v\in \Q^{>0}$. Assume that 
$$
\mathcal{E}\subseteq \mathcal{S}_{lc}(d,\Phi,u,v)
$$ 
is a bounded family. Then there exists $\lambda \in \Q^{>0}$ such that for any $(X,B),A$ in $\mathcal{E}$, the pair 
$(X,B+\lambda A)$ is lc.
\end{lem}
\begin{proof}
By definition of boundedness, there is a fixed $r\in \N$ such that 
for any $(X,B),A$ in the family $\mathcal{E}$ we can find a general very ample divisor $H\ge 0$ on $X$ so that 
$$
H^d\le r~~\mbox{and} ~~ (K_X+B+A)\cdot H^{d-1}\le r.
$$
In particular, 
$$
(X,\Supp (B+A+H))
$$ 
belongs to a bounded family of couples. Thus there exist a log resolution $\phi\colon W\to X$ and a reduced divisor $\Sigma$ on $W$ so that $(W,\Sigma)$ is log smooth belonging to a bounded family and $\Sigma$ contains the support of the exceptional divisors of $W\to X$ and the birational transform of $\Supp (B+A+H)$. Moreover, we can assume that there exists a very ample divisor $G$ on $W$ such that $G^d$ and 
$\phi^*H\cdot G^{d-1}$ are bounded. 

On the other hand, by Lemma \ref{l-v.ample-div-dominating}, there is a fixed number $l\in \N$ such that $lH-A$ is big. But then $\phi^*A\cdot G^{d-1}\le l\phi^*H\cdot G^{d-1}$ is bounded which implies that the coefficients of $\phi^*A$ are bounded.  Write 
$$
K_W+B_W=\phi^*(K_X+B)~~ \mbox{and}~~ A_W=\phi^*A.
$$ 
By definition of stable minimal models, $(X,B+tA)$ is slc for some $t>0$, hence it is lc as $X$ is normal, so $A_W$ and $\rddown{B_W}$ have no common components. Moreover, the non-negative coefficients of $B_W$ belong to a fixed finite set depending only on $d,\Phi,u,v$, by \cite[Lemma 6.9]{BH22}\cite[Lemma 8.2]{B21b}.
Therefore, we can find a fixed $\lambda \in \Q^{>0}$ so that the coefficients of $B_W+\lambda A_W$ do not exceed $1$. Thus $(X,B+\lambda A)$ is lc.

\end{proof}

\subsection{Boundedness of stable Calabi-Yau pairs}
Slc stable Calabi-Yau pairs were treated in \cite{B20} where boundedness was essentially established. This is one of the key ingredients of the moduli theory developed in this paper.

\begin{thm}[\cite{B20}]\label{t-bnd-slc-min-model-CY-case}
Let $d\in\N$, $\Phi\subset \Q^{\ge 0}$ be a DCC set, $\Gamma\subset \Q^{>0}$ be a finite set, and $\sigma\in\Q[t]$ be a polynomial.
Consider those  
$$
(X,B),A\to Z \in S_{\rm slc}(d,\Phi,\Gamma,\sigma)
$$
with $\dim Z=0$. Then such $(X,B),A$ form a bounded family.
\end{thm}
\begin{proof}
By assumption, $K_X+B\sim_\Q 0$ and $A$ is ample with $\vol(A)\in \Gamma$, and for every small $t>0$, 
$$
\sigma(t)=\vol(K_X+B+tA)=A^dt^d=\vol(A)t^d.
$$
Since $K_X+B+tA$ is ample for every $t>0$, $\vol(K_X+B+tA)$ is a polynomial in $t$, so the above equalities hold for every $t>0$. 
When the divisor $A$ is integral, boundedness of $(X,B),A$ follows from \cite[Corollary 1.8]{B20}. In general, when $A$ is not necessarily integral, 
by \cite[Theorem 6.4]{B20}, there exists a positive rational number $t>0$ depending only on $d, \Phi, \Gamma$ such that $(X, B+ tA)$ is slc.  
Then by \cite[Theorem 1.1]{HMX18},  $(X, B+tA)$ is bounded. 

\end{proof}

\begin{lem}\label{l-s.min.mod-bnd-gen-fib}
Let $d\in\N$, $\Phi\subset \Q^{\ge 0}$ be a DCC set, $\Gamma\subset \Q^{>0}$ be a finite set, and $\sigma\in\Q[t]$ be a polynomial.
Let 
$$
(X,B),A\to Z \in S_{\rm slc}(d,\Phi,\Gamma,\sigma)
$$
and let $F$ be a general fibre of $X\to Z$ over an irreducible component of $Z$. 
Define $K_F+B_F=(K_X+B)|_F$ and $A_F=A|_F$. Then $(F,B_F),A_F$ is a stable Calabi-Yau pair and it belongs to a bounded family.
\end{lem}
\begin{proof}
By Theorem \ref{t-bnd-slc-min-model-CY-case}, we can assume $\dim Z>0$. 
 Since $(X,B)$ is slc and since $F$ is a general fibre, $(F,B_F)$ is slc. Moreover, $A_F$ is ample as $A$ is ample over $Z$, and $\Supp A_F$ does contain any non-klt centre of $(F,B_F)$ as $\Supp A$ does not contain any non-klt centre of $(X,B)$. In addition, $K_F+B_F\sim_\Q 0$ as $K_X+B\sim_\Q 0/Z$. So $(F,B_F),A_F$ is a stable Calabi-Yau pair.
 By assumption, $\vol(A_F)\in \Gamma$, so   
$$
(F,B_F),A_F\in  S_{\rm slc}(r,\Phi,\Gamma,\theta)
$$  
where $r=\dim F$ and $\theta(t)=\vol(A_F)t^{r}$. Therefore, such $(F,B_F),A_F$ belong to a bounded family by Theorem \ref{t-bnd-slc-min-model-CY-case} because there are finitely many possibilities for $\theta$ as $\Gamma$ is finite. 

\end{proof}

\subsection{Invariants of a stable minimal model}
We will see in the next lemma that we can recover various invariants of a $(d,\Phi,\Gamma,\sigma)$-stable minimal model from the given data $d,\Phi,\Gamma,\sigma$.

\begin{lem}\label{l-s.min.mod.base-d,v-determined-by-u,sigma}
Let $d\in\N$, $\Phi,\Gamma\subset \R^{\ge 0}$, and $\sigma\in\R[t]$ be a polynomial.
Let 
$$
(X,B),A\to Z \in S_{\rm slc}(d,\Phi,\Gamma,\sigma).
$$
Then 
\begin{itemize}
\item  we have 
$$
\sigma(t)=(K_X+B+tA)^d=\sum_{i=0}^d \binom{d}{i} (K_X+B)^{d-i}\cdot A^it^i,
$$
so the intersection numbers $(K_X+B)^{d-i}\cdot A^i$ are determined by $d,\sigma$;
\item  $d-\dim Z$ is the smallest degree of $t$ that appears in $\sigma$; 

\item if $Z$ is irreducible, then 
$$
(K_X+B)^{\dim Z}\cdot A^{d-\dim Z}=\Ivol(K_X+B)\vol(A|_F)
$$ 
where $F$ is a general fibre of $X\to Z$.
\end{itemize}
\end{lem}
\begin{proof}
By assumption, for small $t>0$, 
$$
\sigma(t)=(K_X+B+tA)^d=\sum \binom{d}{i} (K_X+B)^{d-i}\cdot A^it^i.
$$
Since both left hand and right hand side are polynomials in $t$ that agree for all small values of $t$, the two sides coincide. Thus to give $\sigma$ is the same as giving the string of numbers $\binom{d}{i} (K_X+B)^{d-i}\cdot A^i$.  

By assumption, $K_X+B\sim_\R f^*H$ where $f$ denotes $X\to Z$ and $H$ is some ample $\R$-divisor on $Z$.
Then $(K_X+B)^{d-i}\cdot A^i=0$ if $d-i>\dim Z$ because $H^{d-i}=0$. Moreover, the $0$-cycle $H^{\dim Z}$ is numerically the same as $\sum \vol(H|_{Z'})z'$ where $Z'$ runs through the irreducible components of $Z$ of maximal dimension and $z'$ is a general point of $Z'$. Thus 
when $d-i=\dim Z$, the cycle $(K_X+B)^{d-i}$ is numerically equivalent to the cycle $\sum \vol(H|_{Z'})F'$ where $Z'$ runs through the irreducible components of $Z$ of maximal dimension and $F'$ is a general fibre of $X\to Z$ over $Z'$. Thus 
$$
(K_X+B)^{\dim Z}\cdot A^{d-\dim Z}=\sum \vol(H|_{Z'})F'\cdot A^{d-\dim Z}
$$
$$
=\sum \vol(H|_{Z'})\vol(A|_{F'})>0.
$$
This shows that $d-\dim Z$ is the smallest degree of $t$ that appears in $\sigma$.
Moreover, if $Z$ is irreducible, then the above number is $\Ivol(K_X+B)\vol(A|_F)$ 
where $F$ is a general fibre of $X\to Z$. 

\end{proof}

\subsection{Boundedness of lc stable minimal models}

\begin{proof}[Proof of Theorem \ref{t-bnd-s-mmodels-lc}]
\emph{Step 1}.
We do induction on dimension. 
Pick 
$$
(X,B),A\to Z\in \mathcal{S}_{lc}(d,\Phi,\Gamma,\sigma).
$$
By Lemma \ref{l-s.min.mod.base-d,v-determined-by-u,sigma}, $\Ivol(K_X+B)\vol(A|_F)$ is fixed depending only on $d,\sigma$ where $F$ is a general fibre of $X\to Z$. Since $\Gamma$ is finite and since $\vol(A|_F)\in \Gamma$, there are finitely many possibilities for $\Ivol(K_X+B)$. Therefore, we may fix both $u=\vol(A|_F)$ and $v=\Ivol(K_X+B)$.  

Again by Lemma \ref{l-s.min.mod.base-d,v-determined-by-u,sigma}, 
$$
\sigma(t)=\sum_{i=0}^d \binom{d}{i} (K_X+B)^{d-i}\cdot A^it^i
$$
and the intersection numbers $(K_X+B)^{d-i}\cdot A^i$ are uniquely determined by $d,\sigma$.\\

\emph{Step 2}.
Assume that $K_X+B+A$ is ample, that is, $(X,B),A$ is a strongly stable minimal model,
Then for $t\in [0,1]$,  
$$
\vol(K_X+B+tA)=(K_X+B+tA)^d
$$
is a polynomial in $t$ in the interval $[0,1]$. Thus for every $t\in [0,1]$, 
$$
\vol(K_X+B+tA)=\sigma(t)
$$ 
because both sides are polynomials which agree for $t$ sufficiently small. 
Thus by \cite[Theorem 1.10]{B21b}, $(X,B),A$ is bounded. 
In particular, when $d=1$, $K_X+B+A$ is ample, so the theorem holds in this case. So we can assume that $d\ge 2$.\\

\emph{Step 3}.
By definition, $K_X+B$ is semi-ample defining the contraction $f\colon X\to Z$.
If $\dim Z=0$, then $(X,B),A$ is a stable Calabi-Yau pair and such pairs are bounded by Theorem \ref{t-bnd-slc-min-model-CY-case}. We can then assume that $\dim Z>0$.
By Theorem \ref{t-bnd-base-adjunction},  
there exist $l \in \N$ depending only on $d,\Phi, u,v$ 
such that we can write an adjunction formula
$$
K_X+B\sim_l f^*(K_Z+B_Z+M_Z)
$$
such that the generalised pairs $(Z, B_Z + M_Z)$ belong to a bounded family, and  
$$
L:=l(K_Z+B_Z+M_Z)
$$ 
is very ample. \\

\emph{Step 4}.
Let $T$ be a general member of $|L|$ and let $S$ be its pullback to $X$. 
Define 
$$
K_S+B_S=(K_X+B+S)|_S
$$ 
and $A_S=A|_S$. Then $(S,B_S),A_S\to T$ is a stable minimal model: indeed, $(S,B_S)$ is projective lc, $K_S+B_S$ is semi-ample defining the contraction $g\colon S\to T$, and $K_S+B_S+tA_S$ is ample and $(S,B_S+tA_S)$ is lc for every small $t>0$ because $(X,B+S+tA)$ is lc and $K_X+B+S+tA$ is ample for every small $t>0$.

On the other hand, if $G$ is a general fibre of $S\to T$, 
then 
$$
\vol(A_S|_G)=\vol(A|_G)=u
$$ 
because $G$ is among the general fibres of $X\to Z$.
Moreover, for $0\le t\ll 1$, we have
$$
\vol(K_S+B_S+tA_S)=\vol((K_X+B+S+tA)|_S)
$$
$$
=(K_X+B+S+tA)^{d-1}\cdot S
$$
$$
=((l+1)(K_X+B)+tA)^{d-1}\cdot S
$$
$$
=((l+1)(K_X+B)+tA)^{d-1}\cdot l(K_X+B).
$$
This is a polynomial $\theta$ in $t$ whose coefficients are uniquely determined by the intersection numbers 
$(K_X+B)^{d-i}\cdot A^i$ and by $l$. Therefore, $\theta$ is determined by $d,\sigma,l$. 
Thus $(S,B_S),A_S$ is a stable minimal model belonging to  
$$
\mathcal{S}_{lc}(d-1,\Phi,\Gamma,\theta).
$$ 
Therefore,  $(S,B_S),A_S$ belongs to a bounded family by induction.\\ 

\emph{Step 5}.
In the notation of Lemma \ref{l-bnd-lct-on-bnd-family}, $(S,B_S),A_S$ is a stable minimal model belonging to  
$$
\mathcal{S}_{lc}(d-1,\Phi,u,v')
$$ 
for some fixed $v'$ that depends on $d,\sigma,u$: this follows from Lemma \ref{l-s.min.mod.base-d,v-determined-by-u,sigma}.
Then there is a fixed $\tau\in \Q^{>0}$ such that $(S,B_S+\tau A_S)$ is lc, by Lemma \ref{l-bnd-lct-on-bnd-family}.
Then by our choice of $S$, 
$$
(X,B+S+\tau A)
$$ 
is lc near $S$. Then we deduce that    
$(X,B+\tau A)$ is lc over the complement of a finite set of closed points of $Z$: otherwise the non-lc locus of the pair maps onto a positive dimensional closed subset of $Z$ which is intersected by $T$ hence $S$ intersects the non-lc locus, a contradiction.\\ 

\emph{Step 6}.
In this step, assume that $K_X+B+\tau A$ is nef. Then since $K_X+B$ is also nef, $K_X+B+t A$ is nef for any $t\in [0,\tau]$. Then $K_X+B+t A$ is ample for any $t\in (0,\tau)$: otherwise there is an extremal ray $R$ intersecting $K_X+B+tA$ trivially for some $t$; but then $R$ intersects both $K_X+B$ and $K_X+B+\tau A$ trivially; then $R$ intersects $K_X+B+t A$ trivially for any  sufficiently small $t>0$ and this contradicts the assumption that $K_X+B+tA$ is ample for such $t$. So replacing $\tau$ with 
$\frac{\tau}{2}$ we can assume that $K_X+B+\tau A$ is ample.

Therefore, $(X,B),\tau A$ is strongly stable belonging to 
$$
\mathcal{S}_{lc}(d,\Phi\cup \tau \Phi,\tau \Gamma,\sigma(\tau t)),
$$ 
hence $(X,B),\tau A$ belongs to a bounded family, by Step 2, which in turn implies that $(X,B),A$ belongs to a bounded family.\\ 

\emph{Step 7}.
From now on we assume that $K_X+B+\tau A$ is not nef. By Step 5, the non-lc locus of $(X,B+\tau A)$ is contained in finitely many fibres of $X\to Z$.
Let $R$ be a $K_X+B+\tau A$-negative extremal ray. Since $X\to Z$ is the contraction defined by $K_X+B$ and since $K_X+B+\tau A$ is ample over $Z$, 
$R$ is not vertical over $Z$, that is, $(K_X+B)\cdot R>0$. This in particular means that $R$ is not generated by the non-lc locus of $(X,B+\tau A)$, that is, $R$ is not contained in the image of the map of the closed cone of curves $\overline{NE}(\Pi)\to \overline{NE}(X)$ where $\Pi$ is the non-lc locus.
Then by the general cone theorem  \cite{Am03}\cite[Theorem 1.1]{F11}, $R$ is generated by a curve $C$ with 
$$
(K_X+B+\tau A)\cdot C\ge -2d.
$$
Thus from $f^*L\cdot C>0$ we see that 
$$
(K_X+B+2df^*L+\tau A)\cdot C\ge 0.
$$ 
This argument implies that 
$$
K_X+B+2df^*L+\tau A
$$ 
is nef. So in view of $f^*L\sim l(K_X+B)$ we see that 
$$
K_X+B+\frac{\tau}{1+2dl} A\sim_\Q\frac{1}{1+2dl}(K_X+B+2df^*L+\tau A)
$$
is nef. Therefore, replacing $\tau$ with $\frac{\tau}{1+2dl}$, we are reduced to the case 
above when $K_X+B+\tau A$ is nef. 

\end{proof}

%%%%%%%%%%%%%%%%%%%%%%%%%%%%%
%%%%%%%%%%%%%%%%%%%%%%%%%%%%%

\section{\bf Weighted stable minimal models}

In this section, we define weighted stable minimal models which are normal but not necessarily irreducible. The main feature is that each component has an associated weight. Such models naturally appear in our inductive proof of Theorem \ref{t-bnd-s-mmodels-slc} in the next section. Moreover, they help to understand normalisation of slc stable minimal models in terms of Iitaka volumes.

\subsection{Weighted stable minimal models}

A \emph{weighted lc stable minimal model}  is a finite disjoint union 
$$
\Lambda\cdot (X,B),A:=\bigsqcup \lambda_j\cdot (X_j,B_j),A_j
$$ 
where each $(X_j,B_j),A_j$ is an lc stable minimal model and $\lambda_j$ is a positive real number. Here we think of $\Lambda$ as the sequence of the weights $\lambda_j$. The notation $\lambda_j\cdot (X_j,B_j),A_j$ simply means that we consider $(X_j,B_j),A_j$ with weight $\lambda_j$ which in turn means that $(X_j,B_j),A_j$ comes with a specified number $\lambda_j$ attached to it.

We say a set of weighted lc stable minimal models $\Lambda\cdot (X,B),A$ is bounded if the $(X_j,B_j),A_j$ belong to a fixed bounded family, the number of the $X_j$ in $X$ is bounded, and the $\lambda_j$ are bounded from above.  

In practice, when we normalise a usual slc stable minimal model we get a weighted lc stable minimal model where all the weights are $1$. But when we treat such models inductively, non-trivial weights appear as we will see in the proof of Proposition \ref{p-bnd-w-s-min-models-DCC} below. 

Let $d\in \N$, $\Phi,\Psi\subset \R^{\ge 0}$, and $\sigma\in \R[t]$ be a polynomial. 
Let $\mathcal{WS}_{\rm lc}(d,\Phi,\Psi,\sigma)$ denote the set of weighted lc stable minimal models 
$$
\Lambda\cdot (X,B),A=\bigsqcup \lambda_j\cdot (X_j,B_j),A_j
$$
such that 
\begin{itemize}
\item $\dim X_j=d$,
\item the coefficients of $B_j$ and $A_j$ are in $\Phi$,
\item $0<\lambda_j\in \Phi$,
\item $\Ivol(K_{X_j}+B_j)\in \Psi$, and 
\item the volume 
$$
\vol(\Lambda\cdot (X,B+tA)):=\sum \lambda_j\vol(K_{X_j}+B_j+tA_j)=\sigma(t)
$$
for every sufficiently small $t>0$. 
\end{itemize}

We hope that the notation $\mathcal{WS}_{\rm lc}(d,\Phi,\Psi,\sigma)$ is not confused with $\mathcal{S}_{\rm lc}(d,\Phi,\Gamma,\sigma)$. We will use $\Psi$ only for weighted models and use $\Gamma$ for usual (lc or slc) stable minimal models to avoid confusion.

\subsection{Boundedness of weighted stable minimal models}
We now come to the main result of this section. The main point of this result is that we derive boundedness from assumptions on the Iitaka volume rather than from assumptions on the volume of $A$ along the general fibres as in Theorem \ref{t-bnd-s-mmodels-lc}.

\begin{prop}\label{p-bnd-w-s-min-models-DCC}
Let $d\in \N$, $\Phi,\Psi\subset \Q^{\ge 0}$ be DCC sets, and $\sigma\in \Q[t]$ be a polynomial. 
Then 
\begin{enumerate}
\item $\mathcal{WS}_{\rm lc}(d,\Phi,\Psi,\sigma)$ is a bounded family;
\item there exist finite subsets $\Phi^\circ\subset \Phi$ and $\Psi^\circ\subset \Psi$ such that 
$$
\mathcal{WS}_{\rm lc}(d,\Phi,\Psi,\sigma)=\mathcal{WS}_{\rm lc}(d,\Phi^\circ,\Psi^\circ,\sigma);
$$ 
\item there exists $\tau\in \Q^{>0}$ such that for any 
$$
\Lambda\cdot (X,B),A=\bigsqcup \lambda_j\cdot (X_j,B_j),A_j \in \mathcal{WS}_{\rm lc}(d,\Phi,\Psi,\sigma),
$$
the pair $(X_j,B_j+\tau A_j)$ is lc and $K_{X_j}+B_j+\tau A_j$ is ample for each $j$.
\end{enumerate}
\end{prop}

\begin{proof}
\emph{Step 1.}
We apply induction on the dimension $d$. 
Pick a weighted stable minimal model
$$
\Lambda\cdot (X,B),A=\bigsqcup \lambda_j\cdot (X_j,B_j),A_j\in \mathcal{WS}_{\rm lc}(d,\Phi,\Psi,\sigma).
$$
For sufficiently small $t>0$, $K_{X_j}+B_j+tA_j$ is ample for each $j$ as $({X_j},B_j),A_j$ is a stable minimal model. So for such $t$,  
$$
\sigma(t)=\sum \lambda_j\vol(K_{X_j}+B_j+tA_j)
$$
$$
=\sum \lambda_j(K_{X_j}+B_j+tA_j)^d
$$
$$
=\sum \lambda_j\big(\sum_{i=0}^d\binom{d}{i} (K_{X_j}+B_j)^{d-i}\cdot A_j^it^i\big)
$$
$$
=\sum_{i=0}^d\big(\sum \lambda_j\binom{d}{i} (K_{X_j}+B_j)^{d-i}\cdot A_j^i \big)t^i.
$$
 Thus for each $i$, the number 
$$
\sum \lambda_j(K_{X_j}+B_j)^{d-i}\cdot A_j^i
$$
is the coefficient of $t^i$ in $\sigma$ divided by $\binom{d}{i}$, hence it is fixed. 

When $d=1$, this means that both 
$$
\sum \deg \lambda_j(K_{X_j}+B_j) ~~\mbox{and}~~ \sum \deg \lambda_jA_j
$$ 
are fixed numbers, hence there are finitely many possibilities for $\deg (K_{X_j}+B_j)$, for $\deg A_j$, and for $\lambda_j$. So the proposition holds in this case. We will then assume that $d\ge 2$.\\

\emph{Step 2.}
Each $K_{X_j}+B_j$ is semi-ample by definition, hence it defines a contraction $f_j\colon X_j\to Z_j$. 
Let $e=\max\{\dim Z_j\}$. When $d-i>e$, $(K_{X_j}+B_j)^{d-i}=0$ as a cycle, so
$$
\sum \lambda_j(K_{X_j}+B_j)^{d-i}\cdot A_j^i=0.
$$
But for $d-i=e$,
$$
\sum \lambda_j(K_{X_j}+B_j)^{e}\cdot A_j^{d-e}
=\sum_{\dim Z_j=e} \lambda_j(K_{X_j}+B_j)^{e}\cdot A_j^{d-e}
$$
$$
=\sum_{\dim Z_j=e} \lambda_j\Ivol(K_{X_j}+B_j)F_j\cdot A_j^{d-e}
$$
$$
=\sum_{\dim Z_j=e} \lambda_j\Ivol(K_{X_j}+B_j)\vol(A_{F_j})>0  \hspace{1cm} (*)
$$
where $F_j$ is a general fibre of $X_j\to Z_j$ and $A_{F_j}=A_j|_{F_j}$, and we have used the facts that, numerically, the cycle $(K_{X_j}+B_j)^{e}$ is equivalent to the cycle $\Ivol(K_{X_j}+B_j)F_j$ and that $F_j\cdot A_j^{d-e}=A_{F_j}^{d-e}=\vol(A_{F_j})$.

In particular, 
$d-e$ is the smallest number such that the coefficient of $t^{d-e}$ in $\sigma$ is not zero. Thus $e$ is determined by $d,\sigma$.\\

\emph{Step 3.}
By assumption, for each $j$, we have 
$$
0<\lambda_j\in \Phi, \ \ 0<\Ivol(K_{X_j}+B_j)\in \Psi, \ \ \vol(A_{F_j})>0.
$$ 
In the rest of this step, we consider only those $j$ with $\dim Z_j=e$. Then $\vol(A_{F_j})$ is bounded from above depending only on $d,\Phi,\Psi,\sigma$, by the formula $(*)$ in Step 2 and by the fact that the value in $(*)$ is fixed as it is determined by $d,\sigma$ according to Step 1, and by the assumption that $\Phi,\Psi$ are DCC sets. 

On the other hand, defining
$$
K_{F_j}+B_{F_j}:=(K_{X_j}+B_j)|_{F_j}\sim_\Q 0,
$$ 
 $({F_j},B_{F_j}),A_{F_j}$ is a stable Calabi-Yau pair where the coefficients of $B_{F_j},A_{F_j}$ belong to $\Phi$. Therefore, since $\vol(A_{F_j})$ is bounded from above, the lc threshold of $A_{F_j}$ with respect to  $({F_j},B_{F_j})$ is bounded from below, by \cite[Thereom 6.4]{B20}. So we can choose a fixed rational number $u>0$ depending only on $d,\Phi,\Psi,\sigma$ such that $({F_j},B_{F_j}+uA_{F_j})$ is lc. This implies that the volume 
$$
\vol(K_{F_j}+B_{F_j}+uA_{F_j})=\vol(uA_{F_j})=u^{d-e}\vol(A_{F_j})
$$
belongs to a DCC set, by \cite[Theorem 1.3]{HMX14}. This in turn implies that $\vol(A_{F_j})$ belongs to a DCC set. But then by the formula $(*)$, there are finitely many possibilities for the numbers 
$\lambda_j$ and $\Ivol(K_{X_j}+B_j)$ and $\vol(A_{F_j})$. Therefore, by Theorem \ref{t-bnd-base-adjunction}, we can find a fixed $l\in \N$ such that we can write an adjunction formula 
$$
K_{X_j}+B_j\sim_l f_j^*(K_{Z_j}+B_{Z_j}+M_{Z_j})
$$
where $l(K_{Z_j}+B_{Z_j}+M_{Z_j})$ is very ample.\\

\emph{Step 4.}
For $1\le q\le e$, let $\mathcal{C}_q$ be the following statement: there exist $l\in \N$ and finite subsets $\Phi^\circ\subset \Phi$ and $\Psi^\circ\subset \Psi$ depending only on $d,\Phi,\Psi,\sigma$ such that for each $j$ with $\dim Z_j\ge q$, we have:
\begin{itemize}
\item the coefficients of $B_j$ and $\lambda_j$ belong to $\Phi^\circ$,
\item $\Ivol(K_{X_j}+B_j)$ belongs to $\Psi^\circ$, and 
\item we can write an adjunction formula 
$$
K_{X_j}+B_j\sim_l f_j^*(K_{Z_j}+B_{Z_j}+M_{Z_j})
$$
where $l(K_{Z_j}+B_{Z_j}+M_{Z_j})$ is very ample.
\end{itemize}
By the previous step, $\mathcal{C}_{e}$ holds.
Assume that statement $\mathcal{C}_{q+1}$ holds and that $l$ is chosen for $q+1$. We will show in the following steps that statement $\mathcal{C}_{q}$ also holds.\\

\emph{Step 5.}
For each $j$ with $\dim Z_j>q$, $l(K_{Z_j}+B_{Z_j}+M_{Z_j})$ is very ample by assumption. For such $j$, 
pick $q$ general elements $R_{1,j},\dots,R_{q,j}$ of 
$$
|l(K_{Z_j}+B_{Z_j}+M_{Z_j})|
$$ 
and let 
$$
T_j=R_{1,j} \cap \cdots \cap R_{q,j}, \ \ \ S_j=f_j^{-1}T_j.
$$ 
Letting $B_{S_j}=B_j|_{S_j}$ and using adjunction repeatedly, we have 
$$
K_{S_j}+B_{S_j}\sim (K_{X_j}+B_j+\sum f_j^*R_{r,j})|_{S_j}
$$
$$
\sim (K_{X_j}+B_j+ql(K_{X_j}+B_j))|_{S_j}
$$
$$
\sim (1+ql)(K_{X_j}+B_j)|_{S_j}.
$$
Letting $A_{S_j}=A_j|_{S_j}$, then $({S_j},B_{S_j}),A_{S_j}$ is an lc stable minimal model of dimension $d-q$ where the contraction defined by $K_{S_j}+B_{S_j}$ is just $S_j\to T_j$.

Consider the weighted lc stable minimal model 
$$
\Lambda'\cdot (S,B_S),A_S:=\bigsqcup \lambda_j'\cdot (S_j,B_{S_j}),A_{S_j}
$$
where we only consider those $j$ such that 
\begin{itemize}
\item $\dim Z_j>q$ in which case  $(S_j,B_{S_j}),A_{S_j}$ is as in the previous paragraph and $\lambda_j'=\lambda_j$; or 
\item  $\dim Z_j=q$ in which case $S_j=F_j$ is a general fibre of $X_j\to Z_j$, $B_{S_j}=B_j|_{S_j}$, $A_{S_j}=A_j|_{S_j}$ and $\lambda_j'=l^q\lambda_j\Ivol(K_{X_j}+B_j)$.\\
\end{itemize}

\emph{Step 6.}
We claim that 
$$
\Lambda'\cdot (S,B_S),A_S\in \mathcal{WS}_{\rm lc}(d-q,\Phi',\Psi',\sigma')
$$ 
for some DCC sets $\Phi',\Psi' \subset \Q^{\ge 0}$ and polynomial $\sigma'\in \Q[t]$ depending only on $d,\Phi,\Psi,\sigma,q,l$. It is clear that each component of $S$ is of dimension $d-q$. Moreover, by construction, the coefficients of $B_{S_j}$ and $A_{S_j}$ belong to $\Phi$. Also $\lambda_j'=\lambda_j$ for $\dim Z_j>q$ and 
$$
\lambda_j'=l^q\lambda_j\Ivol(K_{X_j}+B_j)
$$ 
for $\dim Z_j=q$, so in any case $\lambda_j'$ belongs to some fixed DCC set $\Phi'$ with $\Phi\subset \Phi'$ because $l^q$ is fixed, $\lambda_j\in \Phi$ and $\Ivol(K_{X_j}+B_j)\in \Psi$.

For $j$ with $\dim Z_j>q$, 
$$
\Ivol(K_{S_j}+B_{S_j})=((1+ql)(K_{Z_j}+B_{Z_j}+M_{Z_j})|_{T_j})^{\dim T_j}
$$
$$
=((1+ql)(K_{Z_j}+B_{Z_j}+M_{Z_j}))^{\dim T_j}\cdot T_j
$$
$$
=((1+ql)(K_{Z_j}+B_{Z_j}+M_{Z_j}))^{\dim Z_j-q}\cdot (l(K_{Z_j}+B_{Z_j}+M_{Z_j}))^q
$$
$$
=l^q(1+ql)^{\dim Z_j-q}(K_{Z_j}+B_{Z_j}+M_{Z_j})^{\dim Z_j}
$$
$$
=l^q(1+ql)^{\dim Z_j-q}\Ivol(K_{X_j}+B_j)
$$
which belongs to a fixed finite set $\Psi'$ as $l,q$ are fixed and as $\Ivol(K_{X_j}+B_j)$ belongs to the finite set $\Psi^\circ$ by assumption.  

On the other hand, for $j$ with $\dim Z_j=q$, $\Ivol(K_{S_j}+B_{S_j})=1$ by convention as $K_{S_j}+B_{S_j}\sim_\Q 0$. So in any case we can assume $\Ivol(K_{S_j}+B_{S_j})\in \Psi'$.

Now let $\sigma'$ be the polynomial such that for sufficiently small $t>0$, 
$$
\sigma'(t)=\sum \lambda_j'\vol(K_{S_j}+B_{S_j}+tA_{S_j}) 
$$
$$
=\sum \lambda_j'(K_{S_j}+B_{S_j}+tA_{S_j})^{d-q}.
$$
For $j$ with $\dim Z_j>q$, the coefficient $a_{i,j}$ of $t^i$ in 
$$
\lambda_j'(K_{S_j}+B_{S_j}+tA_{S_j})^{d-q}
$$ 
is equal to 
$$
\lambda_j\binom{d-q}{i}(K_{S_j}+B_{S_j})^{d-q-i}\cdot A_{S_j}^i
$$
$$
=\lambda_j\binom{d-q}{i} ((1+ql)(K_{X_j}+B_j))^{d-q-i}\cdot A_j^i\cdot S_j
$$
$$
=\lambda_j\binom{d-q}{i}  ((1+ql)(K_{X_j}+B_j))^{d-q-i}\cdot A_j^i\cdot (l(K_{X_j}+B_j))^q
$$
$$
=\binom{d-q}{i}(1+ql)^{d-q-i}l^q\lambda_j(K_{X_j}+B_j)^{d-i}\cdot A_j^i.
$$
Note that for $i=d-q$, this simplifies to $l^q\lambda_j(K_{X_j}+B_j)^{q}\cdot A_j^{d-q}$.

On the other hand, for $j$ with $\dim Z_j=q$, 
$$
\lambda_j'(K_{S_j}+B_{S_j}+tA_{S_j})^{d-q}=\lambda_j'A_{S_j}^{d-q}t^{d-q}
$$
$$
=l^q\lambda_j\Ivol(K_{X_j}+B_j)A_{S_j}^{d-q}t^{d-q}
$$
$$
=l^q\lambda_j\Ivol(K_{X_j}+B_j) S_j\cdot A_j^{d-q}t^{d-q}
$$
$$
=l^q\lambda_j (K_{X_j}+B_j)^q \cdot A_j^{d-q} t^{d-q},
$$
so its coefficient $a_{i,j}$ of $t^i$ is equal to $0$ for $i<d-q$ and equal to $l^q\lambda_j(K_{X_j}+B_j)^q\cdot A_{j}^{d-q}$ for $i=d-q$.

Therefore, for $i<d-q$ the coefficient of $t^i$ in $\sigma'(t)$ is equal to 
$$
\sum_{\dim Z_j\ge q} a_{i,j}=\sum_{\dim Z_j>q} \binom{d-q}{i}(1+ql)^{d-q-i}l^q\lambda_j(K_{X_j}+B_j)^{d-i}\cdot A_j^i 
$$ 
which is fixed depending on $\sigma,d,q,l$ because for $i<d-q$,
$$
\sum_{\dim Z_j\ge 0} \lambda_j(K_{X_j}+B_j))^{d-i}\cdot A_j^i=\sum_{\dim Z_j>q} \lambda_j(K_{X_j}+B_j))^{d-i}\cdot A_j^i
$$ 
is fixed depending only on $d,\sigma$ by Step 1. 

For $i=d-q$, the coefficient of $t^{d-q}$ in $\sigma'(t)$ is 
$$
\sum_{\dim Z_j\ge q} a_{i,j}=\sum_{\dim Z_j\ge q} l^q\lambda_j(K_{X_j}+B_j)^{q}\cdot A_j^{d-q}
=\sum_{\dim Z_j\ge 0} l^q\lambda_j(K_{X_j}+B_j)^{q}\cdot A_j^{d-q}
$$
which is again fixed depending on $\sigma,d,q,l$. This shows that $\sigma'$ is determined by $\sigma,d,q,l$. So the claim of the beginning of this step is proved.\\

\emph{Step 7.}
Now by induction, 
$$
\mathcal{WS}_{\rm lc}(d-q,\Phi',\Psi',\sigma')
$$ 
is a bounded family and there exist finite subsets $\Phi'^\circ\subset \Phi'$ and $\Psi'^\circ\subset \Psi'$ such that 
$$
\Lambda'\cdot (S,B_S),A_S\in \mathcal{WS}_{\rm lc}(d-q,\Phi',\Psi',\sigma')=\mathcal{WS}_{\rm lc}(d-q,\Phi'^\circ,\Psi'^\circ,\sigma').
$$
For $j$ with $\dim Z_j>q$, we know that 
$$
\lambda_j\in \Phi^\circ \ \mbox{and} \ \Ivol(K_{X_j}+B_j)\in \Psi^\circ
$$ 
by assumption where $\Phi^\circ$ and $\Psi^\circ$ are given by statement $\mathcal{C}_{q+1}$ in Step 4. 
For $j$ with $\dim Z_j=q$,  
$$
\lambda_j'=l^q\lambda_j\Ivol(K_{X_j}+B_j)\in \Phi'^\circ,
$$
so since 
$$
\lambda_j\in \Phi \ \mbox{and} \ \Ivol(K_{X_j}+B_j)\in \Psi,
$$ 
we deduce that these numbers belong to a fixed finite set, hence enlarging $\Phi^\circ$ and $\Psi^\circ$, we can assume 
$$
\lambda_j\in \Phi^\circ \ \mbox{and} \ \Ivol(K_{X_j}+B_j)\in \Psi^\circ.
$$ 

On the other hand, again by induction, there exists a fixed rational number $\tau>0$ such that 
$(S_j,B_{S_j}+\tau A_{S_j})$ is lc and $K_{S_j}+B_{S_j}+\tau A_{S_j}$ is ample for each $j$. Then since $K_{S_j}+B_{S_j}$ is nef, $K_{S_j}+B_{S_j}+t A_{S_j}$ is ample for every $t\in (0,\tau]$, 
hence 
$$
\sigma'(t)=\sum \lambda_j'\vol(K_{S_j}+B_{S_j}+t A_{S_j})
$$ 
for all such $t$ because both sides are polynomials which coincide for small $t>0$. Therefore,
$$
\sigma'(\tau)=\sum \lambda_j'\vol(K_{S_j}+B_{S_j}+\tau A_{S_j})
$$ 
is a fixed number.

 Now 
 $$
 \vol(K_{S_j}+B_{S_j}+\tau A_{S_j})
 $$ belongs to a fixed DCC set, by \cite[Theorem 1.3]{HMX14}. This implies that there are finitely many possibilities for the numbers 
$$
\vol(K_{S_j}+B_{S_j}+\tau A_{S_j}).
$$ 
This number is $\tau^{d-q}A_{S_j}^{d-q}$ for $j$ with $\dim Z_j=q$.
Thus $A_{S_j}^{d-q}=\vol(A_{S_j})$ takes finitely many values for such $j$ (recall that for such $j$, $S_j$ is a general fibre of $X_j\to Z_j$). So applying Theorem \ref{t-bnd-base-adjunction}, there is a fixed $n\in \N$ such that, for each $j$ with $\dim Z_j=q$, we can write an adjunction formula
$$
K_{X_j}+B_j\sim_n f_j^*(K_{Z_j}+B_{Z_j}+M_{Z_j})
$$
where $n(K_{Z_j}+B_{Z_j}+M_{Z_j})$ is very ample. 
In particular, this shows that the coefficients of $B_j$ belong to a fixed finite set, hence we can assume they belong to $\Phi^\circ$. Therefore, statement $\mathcal{C}_{q}$ of Step 4 holds, after replacing $l$ with $nl$. 

We should emphasize that up to this point in the proof, we have used the DCC property of the Iitaka volumes $\Ivol(K_{X_j}+B_j)$ only for $j$ with $\dim Z_j\ge q$ (we will refer to this remark in the proof of Theorem \ref{t-bnd-s-mmodels-slc}).\\

\emph{Step 8.}
By the previous steps,  statement $\mathcal{C}_{q}$ holds for every $1\le q\le e$, inductively.
Note that inductively we have also proved that $\vol(A_{F_j})$ belongs to fixed finite set, for $j$ with $\dim Z_j\ge 1$ where $F_j$ is a general fibre of $X_j\to Z_j$.  Let $\Lambda'\cdot (S,B_S),A_S$ be the weighted lc stable minimal model constructed in Step 5 for $q=1$. By Step 7, there exists a fixed rational number $\tau>0$ such that $(S_j,B_{S_j}+\tau A_{S_j})$ is lc, for each $j$. By construction, $S_j$ is the pullback of a hyperplane section on $Z_j$, for $j$ with $\dim Z_j\ge 2$, and $S_j=F_j$ for $j$ with $\dim Z_j=1$. Thus  we deduce that $(X_j,B_{j}+\tau A_{j})$ is lc over the complement of a finite set of closed points of $Z_j$, for each $j$ with $\dim Z_j\ge 1$. Then arguing as in the end of the proof of Theorem \ref{t-bnd-s-mmodels-lc}, by applying the cone theorem, we can replace $\tau$ so that 
 $K_{X_j}+B_{j}+\tau A_{j}$ is ample for each $j$ with $\dim Z_j\ge 1$. On the other hand, for $j$ with $\dim Z_j=0$, 
$$
K_{X_j}+B_{j}+\tau A_{j}\sim_\Q \tau A_{j}
$$
is automatically ample. 

This in particular means that each $(X_j,B_{j}), \tau A_{j}$ is a strongly stable lc minimal model. 
Moreover, 
$$
\sigma(\tau)=\sum \lambda_j\vol(K_{X_j}+B_{j}+\tau A_{j}).
$$
Thus $\vol(K_{X_j}+B_{j}+\tau A_{j})$ is bounded from above for each $j$.
For $j$ with $\dim Z_j\ge 1$, the volumes $\vol(A_j|_{F_j})$ belong to a fixed finite set and the Iitaka volumes $\Ivol(K_{X_j}+B_j)$ belong to the fixed finite set $\Psi^\circ$, by Step 7.  Therefore, applying \cite[Lemma 8.3]{B21b} and decreasing $\tau$ we can assume that $(X_j,B_{j}+\tau A_{j})$ is lc, for $j$ with $\dim Z_j\ge 1$. On the other hand, applying \cite[Theorem 6.4]{B20} and again decreasing $\tau$, we can assume that 
$(X_j,B_{j}+\tau A_{j})$ is lc, for $j$ with $\dim Z_j=0$ as well.

Therefore, the volumes 
$$
\vol(K_{X_j}+B_{j}+\tau A_{j})
$$ 
belong to a DCC set, hence they belong to a fixed finite set by the above expression of $\sigma(\tau)$. Therefore, the number of the $X_j$ is bounded and each $(X_j,B_{j}+\tau A_{j})$ belongs to a bounded family by \cite[Theorem 1.1]{HMX18}. Moreover, we can assume that $\lambda_j$ and the coefficients of $B_j,A_j$ all belong to $\Phi^\circ$ after enlarging $\Phi^\circ$, where $\Phi^\circ, \Psi^\circ$ are given by statement $\mathcal{C}_1$. For $j$ with $\dim Z_j=0$, $\Ivol(K_{X_j}+B_j)=1$, so we can assume it belongs to $\Psi^\circ$.
We have then proved that 
$$
\mathcal{WS}_{\rm lc}(d,\Phi,\Psi,\sigma)=\mathcal{WS}_{\rm lc}(d,\Phi^\circ,\Psi^\circ,\sigma)
$$ 
and that this is a bounded family. We have proved (1),(2),(3) of the proposition.

\end{proof}

%%%%%%%%%%%%%%%%%%%%%%%%%%%%%%
%%%%%%%%%%%%%%%%%%%%%%%%%%%%%%

\section{\bf Slc stable minimal models}
In  this section, we prove our main result on boundedness of stable minimal models in the slc case, that is, Theorem \ref{t-bnd-s-mmodels-slc}. The basic idea is to look at the normalisation of a stable minimal model and try to show that the relevant Iitaka volumes satisfy DCC so that we can apply Proposition \ref{p-bnd-w-s-min-models-DCC}. To achieve the DCC we use boundedness of Stein degrees.

\subsection{Normalisation of an slc stable minimal model}
Given an slc stable minimal model $(X,B),A$, let $X^\nu$ be the normalisation of $X$ and let $X^\nu_j$ be its irreducible components. Let $K_{X^\nu}+B^\nu$ be the pullback of $K_X+B$ and let $K_{X_j^\nu}+B_j^\nu$ be its restriction to $X_j^\nu$. Also let $A^\nu$ be the pullback of $A$ and $A_j^\nu$ its restriction to 
$X_j^\nu$. Then each $(X^\nu_j,B^\nu_j),A_j^\nu$ is an lc stable minimal model. We will think of $(X^\nu,B^\nu),A^\nu$ as the disjoint union of the $(X^\nu_j,B^\nu_j),A_j^\nu$, hence as a weighted stable minimal model where all the wrights are $1$.

\begin{lem}\label{l-s.min.mod-Ivol-over-components}
Let $d\in\N$, $\Phi\subset \Q^{\ge 0}$ be a DCC set, $\Gamma\subset \Q^{> 0}$ be a finite set, and $\sigma\in\Q[t]$ be a polynomial.
Let 
$$
(X,B),A\to Z \in S_{\rm slc}(d,\Phi,\Gamma,\sigma)
$$
and let 
$$
(X^\nu,B^\nu),A^\nu=\bigsqcup (X^\nu_j,B^\nu_j),A_j^\nu
$$ 
be its normalisation and $X_j^\nu\to Z_j$ be the contraction defined by $K_{X_j^\nu}+B_j^\nu$. Then 
\begin{enumerate}
\item if $Z_j\to Z$ surjects onto an irreducible component of $Z$, then $\Ivol(K_{X_j^\nu}+B_j^\nu)$ 
belongs to a fixed DCC set and  $\vol(A_j^\nu|_{F_j^\nu})$ belongs to a fixed finite set where $F_j^\nu$ is a general fibre of $X_i^\nu\to Z_j$; the DCC set and the finite set depend only on $d,\Phi,\Gamma$; 

\item for $j$ with $\dim Z_j=\dim Z$,  $\Ivol(K_{X_j^\nu}+B_j^\nu)$ belongs to a fixed finite set depending only on $d,\Phi,\Gamma,\sigma$.
\end{enumerate}
\end{lem}
\begin{proof}
(1)
Let $F$ be a general fibre of $X\to Z$ over an irreducible component $Z'$ of $Z$. 
Define 
$$
K_F+B_F=(K_X+B)|_F \ \mbox{and} \ A_F=A|_F.
$$ 
Then $(F,B_F),A_F$ belongs to a fixed bounded family, by Lemma \ref{l-s.min.mod-bnd-gen-fib}, depending only on the data $d,\Phi,\Gamma$ (note that $\sigma$ is not used in the proof of \ref{l-s.min.mod-bnd-gen-fib}). Thus if $(F^\nu,B_{F^\nu}),A_{F^\nu}$ is the normalisation of $(F,B_F),A_F$, then the irreducible components of $(F^\nu,B_{F^\nu}),A_{F^\nu}$ are bounded. So there are finitely many possibilities for the volume of $A_{F^\nu}$ on each irreducible component of $F^\nu$. 

On the other hand, if $F$ maps to $z\in Z'$, then 
$(F^\nu,B_{F^\nu}),A_{F^\nu}$ is just the log fibre of $(X^\nu,B^\nu),A^\nu\to Z$ over $z$: this follows from the fact that no irreducible component of $F$ is contained in the singular locus of $X$, hence the fibre of $X^\nu\to Z$ over $z$ maps birationally onto $F$ which means that $F^\nu$ is the fibre of $X^\nu\to Z$ over $z$.
Therefore, if $Z_j$ maps onto $Z'$ and if $F_j^\nu$ is a general fibre of $X^\nu_j\to Z_j$, then 
$F_j^\nu$ is an irreducible component of $F^\nu$ for some choice of $F$, so 
there are finitely many possibilities for $\vol(A_j^\nu|_{F_j^\nu})$ because  $A_j^\nu|_{F_j^\nu}={A_{F^\nu}}|_{F_j^\nu}$. 

If $A^\nu$ is integral, then by \cite[Lemma 7.4]{B21b} and by ACC for lc thresholds \cite{HMX14}, we can find a fixed $l\in \N$ and a fixed DCC set $\Omega$ depending only on $d,\Phi,\Gamma$ such that if $Z_j$ maps onto $Z'$, then we can write an adjunction formula 
$$
K_{X_j^\nu}+B_j^\nu\sim_l f_j^*(K_{Z_j}+B_{Z_j}+M_{Z_j})
$$
where $({Z_j},B_{Z_j}+M_{Z_j})$ is an lc generalised pair, the coefficients of $B_{Z_j}$ belong to $\Omega$ and $lM_{\overline{Z}_j}$ is Cartier where $M_{\overline{Z}_j}$ is the moduli divisor on some high resolution $\overline{Z}_j\to Z_j$. Adding $\frac{1}{l}$ to $\Omega$, we assume that    
$({Z_j},B_{Z_j}+M_{Z_j})$
belongs to $\mathcal{G}_{\rm glc}(e_j,\Omega)$ where $e_j=\dim Z_j$ (see \ref{ss-fam-gen-pairs} for the definition of $\mathcal{G}_{\rm glc}(e_j,\Omega)$). Therefore, 
$$
\Ivol(K_{X_j^\nu}+B_j^\nu)=\vol(K_{Z_j}+B_{Z_j}+M_{Z_j})
$$
belongs to a fixed DCC set, by \cite[Theorem 1.3]{B21b}. In case $A^\nu$ is not integral, we can use a multiple $aA^\nu$ where $a\in \N$ is bounded and $aA^\nu$ is integral over $\eta_{Z_j}$ and then apply the same arguments (note that by the first and second paragraphs, there are finitely many possibilities for the coefficients of $A^\nu$ over $\eta_{Z_j}$).

(2)
Now for small $t>0$, 
$$
\sigma(t)=\vol(K_X+B+tA)
=\sum \vol(K_{X_j^\nu}+B_j^\nu+tA^\nu_j)
$$
$$
=\sum (K_{X_j^\nu}+B_j^\nu+tA^\nu_j)^d.
$$
Expanding the right hand side and arguing similar to Step 1 of the proof of Proposition \ref{p-bnd-w-s-min-models-DCC}, we can see that 
$$
\sigma(t)=\sum_{i=0}^d\big(\sum \binom{d}{i} (K_{X_j^\nu}+B_j^\nu)^{d-i}\cdot {A_j^\nu}^i \big)t^i.
$$
Then taking $i=d-\dim Z$ and using 
$$
(K_{X_j^\nu}+B_j^\nu)^{\dim Z}\cdot ({A_j^\nu})^{d-\dim Z}=\Ivol(K_{X_j^\nu}+B_j^\nu)\vol(A_j^\nu|_{F_j^\nu})
$$ 
for $j$ with $\dim Z_j=\dim Z$ shows that 
$$
 \sum_{\dim Z_j=\dim Z} \binom{d}{d-\dim Z} \Ivol(K_{X_j^\nu}+B_j^\nu)\vol(A_j^\nu|_{F_j^\nu})
$$
is the coefficient of $t^{d-\dim Z}$ in $\sigma$ where $F_j^\nu$ is a general fibre of $X^\nu_j\to Z_j$. 

On the other hand, if $\dim Z_j=\dim Z$, then $Z_j$ maps onto an irreducible component of $Z$ as the induced map $Z_j\to Z$ is finite. So 
 $\Ivol(K_{X_j^\nu}+B_j^\nu)$ belongs to a fixed DCC set and $\vol(A_j^\nu|_{F_j^\nu})$ belongs to a fixed finite set by the above arguments, hence we deduce that $\Ivol(K_{X_j^\nu}+B_j^\nu)$ can take finitely many possible values by the previous paragraph.

\end{proof}

\subsection{Boundedness of slc stable minimal models}

\begin{proof}[Proof of Theorem \ref{t-bnd-s-mmodels-slc}]
The idea is to apply Proposition \ref{p-bnd-w-s-min-models-DCC}. In order to do so we need to ensure that the conditions of the proposition are satisfied which essentially means that after normalisation we need to ensure that the relevant Iitaka volumes belong to a DCC set.

\emph{Step 1}.
Let 
$$
(X,B),A\to Z \in S_{\rm slc}(d,\Phi,\Gamma,\sigma)
$$
and let 
$$
(X^\nu,B^\nu),A^\nu=\bigsqcup (X^\nu_j,B^\nu_j),A_j^\nu
$$ 
be its normaliation which we view as a weighted lc stable minimal model where all the weights are $1$. Let $\Psi$ be the set of all the $\Ivol(K_{X^\nu_j}+B^\nu_j)$ when $(X,B),A$ varies in $S_{\rm slc}(d,\Phi,\Gamma,\sigma)$. Then  
$$
\Lambda\cdot (X^\nu,B^\nu),A^\nu\in \mathcal{WS}_{\rm lc}(d,\Phi,\Psi,\sigma)
$$
where $\Lambda$ is a sequence of $1$'s, one for each $j$. 

Let $f_j\colon X^\nu_j\to Z_j$ be the contraction defined by $K_{X^\nu_j}+B^\nu_j$. We have an induced morphism $Z_j\to Z$ which is finite. Let $V_j$ be the image of $Z_j$ in $Z$. Assume $X_j$ is the irreducible component of $X$ corresponding to $X_j^\nu$.\\ 

\emph{Step 2}.
If $\dim Z=0$, then $(X,B),A$ is bounded, by Theorem \ref{t-bnd-slc-min-model-CY-case}. 
From now on we can assume that $\dim Z>0$. By Lemma \ref{l-s.min.mod-Ivol-over-components}, if $V_j$ is an irreducible component of $Z$, then $\Ivol(K_{X_j^\nu}+B_j^\nu)$ belongs to a fixed DCC subset of $\Psi$ and $\vol(A_j^\nu|_{F_j^\nu})$ belongs to a fixed finite set where $F_j^\nu$ is a general fibre of $X_j^\nu\to Z_j$. 
Moreover, again by the lemma, if $\dim Z_j=\dim V_j=\dim Z$, then there are finitely many possibilities for $\Ivol(K_{X_j^\nu}+B_j^\nu)$. Thus by Theorem \ref{t-bnd-base-adjunction}, for $j$ with $\dim Z_j=\dim Z$, we can find a fixed $l\in \N$ such that we can write an adjunction formula 
$$
K_{X_j^\nu}+B_j^\nu\sim_l f_j^*(K_{Z_j}+B_{Z_j}+M_{Z_j})
$$
where $l(K_{Z_j}+B_{Z_j}+M_{Z_j})$ is very ample.\\

\emph{Step 3.}
Let $e=\dim Z$ which is determined by $d,\sigma$, by Lemma \ref{l-s.min.mod.base-d,v-determined-by-u,sigma}. For $1\le q\le e$, let $\mathcal{D}_q$ be the following statement: there exist $l\in \N$ and a finite subset $\Psi^\circ\subset \Psi$ depending only on $d,\Phi,\Gamma,\sigma$ such that for each $j$ with $\dim Z_j\ge q$, we have:
\begin{itemize}
\item $\Ivol(K_{X_j^\nu}+B_j^\nu)$ belongs to $\Psi^\circ$, and 
\item we can write an adjunction formula 
$$
K_{X_j^\nu}+B_j^\nu\sim_l f_j^*(K_{Z_j}+B_{Z_j}+M_{Z_j})
$$
where $l(K_{Z_j}+B_{Z_j}+M_{Z_j})$ is very ample.
\end{itemize}
 By the previous step,  
$\mathcal{D}_{e}$ holds.
Assume that statement $\mathcal{D}_{q+1}$ holds. We will show in the following steps that statement $\mathcal{D}_{q}$ also holds. If there is no $j$ with $\dim Z_j=q$, then $\mathcal{D}_{q}$ trivially  holds. So assume that there exists $j$ with $\dim Z_j=q$.\\  

\emph{Step 4.}  
Pick $Z_j$ with $\dim Z_j=q$. We claim that $\Ivol(K_{X_j^\nu}+B_j^\nu)$ belongs to a fixed DCC subset $\Psi'\subset \Psi$. We prove this in Steps 4-7. If $V_j$ is a component of $Z$, then the claim follows from Lemma \ref{l-s.min.mod-Ivol-over-components}. So assume $V_j$ is not a component of $Z$ which means $V_j\subsetneq V_i$ for some $i$. To make notation simpler, rearranging the indices, we can assume $j=1$. Then $V_1$ is of dimension $q$ but $V_i$ is of dimension $\ge q+1$. 

Let $F$ be the fibre of $X\to Z$ over a general point $z\in V_1$. 
Then both $X_1$ and $X_i$ intersect $F$ (recall that $X_1,X_i$ are the irreducible components of $X$ corresponding to $X_1^\nu,X_i^\nu$). Since $F$ is connected, rearranging the indices again (but keeping $1$) and perhaps replacing $X_i$, we can assume that there exist components $X_2,\dots,X_{i-1}$ of $X$ such that 
\begin{itemize}
\item $V_j=V_1$ for every $1\le j<i$, and 
\item there is a prime divisor $S_{j,j+1}\subset X_{j}\cap X_{j+1}$ that maps onto $V_1$ for every $1\le j<i$.
\end{itemize}
This can be seen as follows: take a sufficiently small neighbourhood $\tilde{Z}\subset Z$ of the generic point $\eta_{V_1}$ and let $\tilde{X}\subset X$ be the inverse image of $\tilde{Z}$; then the image of every irreducible component of $\tilde{X}$ in $\tilde{Z}$ contains $\eta_{V_1}$, and any two irreducible components of $\tilde{X}$ either do not intersect or the image of any component of their intersection contains $\eta_{V_1}$; next rearrange the indices (but keep $1$) and take $i$ minimal so that $\bigcup_1^i \tilde{X_j}$ is connected and still $V_1\subsetneq V_i$, where $\tilde{X_j}:=X_j\cap \tilde{X}\neq \emptyset$; then possibly rearranging the indices $2,\dots,i-1$, we can assume that $X_{j}\cap X_{j+1}\neq \emptyset$ for every $1\le j<i$; moreover, $V_j=V_1$ for every $1\le j<i$ otherwise we can replace $i$ with a smaller number; on the other hand, since $X$ satisfies Serre's condition $S_2$, removing some codimension $2$ closed subset $X$ stays connected over $\eta_{V_1}$ and all the component intersections are along codimension one subsets; so by the same reasoning we can choose $X_2,\dots,X_i$ so that there is a prime divisor $S_{j,j+1}\subset X_{j}\cap X_{j+1}$ that maps onto $V_1$ for every $1\le j<i$.\\

\emph{Step 5.}  
Since $(X,B)$ is slc, $X$ has nodal singularities in codimension one, so there are prime divisors $T_1\subset X_1^\nu$ and $R_2\subset X_2^\nu$ so that $X_1^\nu\to X_1$ and $X_2^\nu\to X_2$ induce birational morphisms $T_{1}\to S_{1,2}$ and $R_{2}\to S_{1,2}$. 
So we get an induced birational map $T_1\bir R_2$ over $V_1$. Note that $T_1$ is a component of $\rddown{B_1^\nu}$ and $R_1$ is a component of $\rddown{B_2^\nu}$.
Similarly, for every $1\le j<i$, there are prime divisors $T_j\subset X_j^\nu$ and $R_{j+1}\subset X_{j+1}^\nu$ with  birational morphisms $T_j\to S_{j,j+1}$ and $R_{j+1}\to S_{j,j+1}$, and an induced birational map 
$$
\phi_j\colon T_j \bir R_{j+1}/V_1.
$$

Let $(X_j',B_j')$ be a dlt model of $({X_j^\nu},B_j^\nu)$.  Then for each $1\le j<i$ we have the diagram: 
$$
\xymatrix{
X_j'\ar[d] & T_j'\ar[l]\ar[d] \ar@{-->}[rr]& & R_{j+1}'\ar[d]\ar[r] & X_{j+1}'\ar[d]\\
X_j^\nu\ar[ddr]\ar[ddd] & T_j\ar[l]\ar[rd] \ar@{-->}[rr]& & R_{j+1}\ar[ld]\ar[r]& X_{j+1}^\nu\ar[ddl]\ar[ddd]\\
& & S_{j,j+1}\ar[dr] \ar[dl]\ar[d]& &\\
&X_j\ar[r] & X\ar[d] & X_{j+1}\ar[l]&\\
Z_j\ar[rr]&& Z &&Z_{j+1}\ar[ll]
}
$$
where $T_j'\subset X_j'$ is the birational transform of $T_j$ and 
$R_{j+1}'\subset X_{j+1}'$ is the birational transform of $R_{j+1}$. 

Define 
$$
K_{T_{j}'}+B_{T_j'}=(K_{X_j'}+B_j')|_{T_j'}
$$ 
and 
$$
K_{R_{j+1}'}+B_{R_{j+1}'}=(K_{X_{j+1}'}+B_{j+1}')|_{R_{j+1}'}
$$  
by adjunction. Then since both 
$$
K_{X_j'}+B_j' \ \mbox{and} \ K_{X_{j+1}'}+B_{j+1}'
$$ 
are pullbacks of $K_X+B$, the diagram induces a crepant birational map 
$$
\phi_j'\colon (T_j',B_{T_j'})\bir (R_{j+1}',B_{R_{j+1}'})
$$
over $V_1$. Moreover, if $T_j'\to N_j\to V_1$ is the Stein factorisation of $T_j'\to V_1$, then both $(T_j',B_{T_j'})$ and  $(R_{j+1}',B_{R_{j+1}'})$ are dlt log Calabi-Yau fibrations over $N_j$ as $K_X+B\sim_\Q0/Z$. \\ 

\emph{Step 6.} 
 Pick a minimal horizontal$/V_1$ non-klt centre $L_j\subseteq T_j'$ of $(X_j',B_j')$ for $1\le j<i$ and a minimal horizontal$/V_1$ non-klt centre $M_j\subseteq R_j'$ of $(X_j',B_j')$ for $1<j\le i$. Such centres exist because both $T_j'$ and $R_{j}'$ are horizontal$/V_1$ non-klt centres by construction. 
 
 We will show that there is a birational map $L_1\bir M_i$ over $V_1$.
If $L_j=T_j'$ for some $1\le j<i$, then $(T_j',B_{T_j'})$ is klt over the generic point of $V_1$, so $(R_{j+1}',B_{R_{j+1}'})$ is also klt over the same generic point by the crepant birational property of $\phi_j'$, hence in this case $M_{j+1}=R_{j+1}'$, so there is a birational map $\alpha_j\colon L_j\bir M_{j+1}$ by the previous step (which is just $\phi_j'$). However, if $L_j\neq T_j'$, then 
$(T_j',B_{T_j'})$ is not klt over the generic point of $V_1$, so $(R_{j+1}',B_{R_{j+1}'})$ is also not klt over the same generic point which means $M_{j+1}\neq R_{j+1}'$. In this case, by Lemma \ref{l-stein-deg-dlt-crep} (applied over $N_j$), there is again a birational map $\alpha_j\colon L_j\bir M_{j+1}$. 

Moreover,  for  $1< j<i$, both $L_j,M_j$ are minimal horizontal$/V_1$ non-klt centres of $(X_j',B_j')$. 
Also $(X_j',B_j')\to Z_j$ is a dlt log Calabi-Yau fibration and $V_j=V_1$ for $1< j<i$ which ensures that $L_j,M_j$ are horizontal over $Z_j$.
So  for such $j$ there is a birational map $\beta_j\colon M_j\bir L_j$, by Lemma \ref{l-stein-deg-dlt}. 

Combining all these birational maps, we get 
$$
L_1\overset{\alpha_1}\bir M_2\overset{\beta_{2}}\bir L_2\bir \cdots \bir M_{i-1} \overset{\beta_{i-1}}\bir L_{i-1}\overset{\alpha_{i-1}}\bir M_i/V_1,
$$
hence we get a birational map $L_1\bir M_i$ over $V_1$.\\

\emph{Step 7.} 
On the other hand, since we are assuming statement $\mathcal{D}_{q+1}$ of Step 3 and since $\dim V_i>q$, 
\begin{itemize}
\item $\Ivol(K_{X_i^\nu}+B_i^\nu)$ belongs to a fixed finite subset $\Psi^\circ\subset \Psi$, and 
\item we can write an adjunction formula 
$$
K_{X_i^\nu}+B_i^\nu\sim_l f_i^*(K_{Z_i}+B_{Z_i}+M_{Z_i})
$$
where $l(K_{Z_i}+B_{Z_i}+M_{Z_i})$ is very ample. 
\end{itemize}
In particular, if $P_i$ is the image of $M_i$ under the morphism $X_i'\to Z_i$, 
then 
$$
l^{\dim P_i}\vol((K_{Z_i}+B_{Z_i}+M_{Z_i})|_{P_i})\in \N.
$$ 
But then we deduce that the Iitaka volume 
$$
\Ivol((K_{X_i'}+B_i')|_{M_i})=\sdeg(M_i/P_i)\vol((K_{Z_i}+B_{Z_i}+M_{Z_i})|_{P_i})
$$ 
belongs to $\frac{1}{l^{\dim P_i}}\N$. Therefore, the birational map $L_1\bir M_i/V_1$ of Step 6 shows that 
$$
\Ivol((K_{X_1'}+B_1')|_{L_1})=\Ivol((K_{X_i'}+B_i')|_{M_i})
$$
also belongs to $\frac{1}{l^{\dim P_i}}\N$ where the equality follows from the fact that both $(K_{X_1'}+B_1')|_{L_1}$ and $(K_{X_i'}+B_i')|_{M_i}$ are pullbacks of a common ample divisor on $V_1$, up to $\Q$-linear equivalence.

By construction, $(X_1',B_1')\to Z_1$ is a log Calabi-Yau fibration and $L_1$ is a horizontal$/Z_1$ non-klt centre. Then by Theorem \ref{t-bnd-stein-deg-lc-main}, the Stein degree $\sdeg(L_1/Z_1)$ is bounded from above depending only on $d$. Moreover, 
$$
\Ivol(K_{X_1^\nu}+B_1^\nu)=\Ivol(K_{X_1'}+B_1')=\frac{1}{\sdeg(L_1/Z_1)}\Ivol((K_{X_1'}+B_1')|_{L_1}). 
$$
Therefore, by the previous paragraph, we deduce that $\Ivol(K_{X_1^\nu}+B_1^\nu)$ belongs to a fixed DCC subset $\Psi'\subset \Psi$. This proves the claim in the beginning of Step 4.\\ 

\emph{Step 8.} 
We can now use the arguments of the proof of Proposition \ref{p-bnd-w-s-min-models-DCC} to derive statement $\mathcal{D}_{q}$. Indeed, in Steps 1-7 of the proof of the proposition, where in particular we derived statement $\mathcal{C}_{q}$ from statement $\mathcal{C}_{q+1}$, we only used the DCC property of the Iitaka volumes for the components with Kodaira dimension $\ge q$, that is, for $j$ with $\dim Z_j\ge q$ (see the remark at the end of Step 7 of that proof). 
 Thus all those steps also apply to $(X^\nu,B^\nu),A^\nu$ in our setting. Therefore, from statement $\mathcal{C}_q$ we deduce that replacing $l\in \N$ and $\Psi^\circ\subset \Psi$ given in statement $\mathcal{D}_{q+1}$, we can assume that  
  for each $j$ with $\dim Z_j\ge q$, $\Ivol(K_{X_j^\nu}+B_j^\nu)$ belongs to $\Psi^\circ\subset \Psi$ and  we can write an adjunction formula 
$$
K_{X_j^\nu}+B_j^\nu\sim_l f_j^*(K_{Z_j}+B_{Z_j}+M_{Z_j})
$$
where $l(K_{Z_j}+B_{Z_j}+M_{Z_j})$ is very ample. That is, $\mathcal{D}_{q}$ holds. 

Inductively, statement $\mathcal{D}_{q}$ holds for every $q\ge 1$. Thus $\Ivol(K_{X_j^\nu}+B_j^\nu)$ belongs to a fixed finite set for every $j$ with $\dim Z_j\ge 1$. The same holds trivially when $\dim Z_j=0$. We can then apply Proposition \ref{p-bnd-w-s-min-models-DCC} to deduce that the weighted lc stable minimal model $\Lambda\cdot (X^\nu,B^\nu),A^\nu$ belongs to a bounded family. Moreover, again by the proposition, we can choose a fixed rational number $\tau>0$ so that  $(X^\nu_j,B^\nu_j+\tau A^\nu_j)$ is lc and $K_{X^\nu_j}+B^\nu_j+\tau A^\nu_j$ is ample for each $j$. This implies that  
$(X,B+\tau A)$ is slc and $K_X+B+\tau A$ is ample with volume $\sigma(\tau)$. Therefore, $(X,B),A$ belongs to a bounded family, by \cite{HMX18}.

\end{proof}

%%%%%%%%%%%%%%%%%%%%%%%%%%%%
%%%%%%%%%%%%%%%%%%%%%%%%%%%%

\section{\bf Families of locally stable pairs}\label{s-locally-stable-families}

In this section, we recall some definitions and results regarding certain families of pairs. These have appeared in the literature in various levels of generality. Our aim here is not to give a historical account. We follow \cite{K21} which treats the most general setting and where references to the literature can be found.

 Throughout this section we fix natural numbers $d,n$ and a vector $\alpha=(a_1,\dots,a_m)$ with positive rational coordinates (we allow the possibility $m=0$ in which case $\alpha$ is empty).

\subsection{Relative Mumford divisors}\label{ss-Mumford-div}
Let $f \colon X \to S$ be a flat morphism of schemes with $S_2$ fibres of pure dimension. 
A closed subscheme $D \subset X$ is a (effective) \emph{relative Mumford divisor over $S$} \cite[Definition 4.68]{K21} if there is an open subset $U \subset X$
such that
\begin{itemize}
\item codimension of $X_s \setminus U_s$ in $X_s$ is $\ge 2$ for every $s \in S$,
\item $D|_U$ is a relative Cartier divisor,
\item $D$ is the closure of $D|_U$, and
\item $X_s$ is smooth at the generic points of $D_s$ for every $s\in S$.
\end{itemize}

Here $D_s$ is the closure of $D|_{U_s}$ in $X_s$. 
By $D|_U$ being relative Cartier we mean that $D|_U$ is a Cartier divisor on $U$ and 
that its support does not contain any irreducible component of any fibre $U_s$. 

Given a morphism $T\to S$, pulling back $D|_U$ to $T\times_SU$ and taking its 
closure gives a relative Mumford divisor $D_T$ on $T\times_SX$ over $T$ which we 
refer to as the \emph{divisorial pullback} of $D$. In particular, as mentioned above, for each $s\in S$, 
we define $D_s=D|_{X_s}$ to be the closure of $D|_{U_s}$ which is the divisorial 
pullback of $D$ to $X_s$.

Now assume $f \colon X \to S$ is a flat projective morphism of schemes with $S_2$ fibres of pure dimension  where $S$ is reduced.
The functor ${\mathcal MDiv}(X/S)$ on the category of reduced schemes over $S$ is defined by setting 
$$
{\mathcal MDiv}(X/S)(T)=\{\mbox{relative Mumford divisors on $T\times_SX$ over $T$}\}
$$
for any morphism $T\to S$ from a reduced scheme and by using divisorial pullback 
to define 
$$
{\mathcal MDiv}(X/S)(T)\to {\mathcal MDiv}(X/S)(T')
$$ 
for a morphism  
$T'\to T$ of reduced schemes over $S$. Since over reduced bases relative Mumford divisors are the same 
as K-flat divisors \cite[Definition 7.1 and comment 7.4.2]{K21}, 
this functor is represented by a reduced separated scheme ${M\!Div}(X/S)$ over $S$ \cite[Theorem 7.3]{K21}.
 In particular, since this moduli space is a fine moduli 
space (as the functor is represented), there is a universal family over ${M\!Div}(X/S)$ of 
relative Mumford divisors.
Moreover, given a relatively very ample 
divisor $A$ on $X$, restricting the above functor to relative Mumford divisors of degree $l$ 
with respect to $A$, where $l$ is a fixed number, the corresponding moduli space 
${M\!Div}_l(X/S)$ is of finite type over $S$. \\

\subsection{Marked locally stable families}
A \emph{$(d,\alpha)$-marked  locally stable pair} over a field $K$ of characteristic zero is a projective 
geometrically connected slc pair $(X,\Delta)$ 
over $K$ with $\dim X=d$ and a decomposition $\Delta=\sum a_iD_i$ where $D_i\ge 0$ are 
integral divisors. The $D_i$ are not assumed to be $\Q$-Cartier and they are not necessarily distinct. 
Two such marked pairs $(X,\Delta)$ and $(X',\Delta')$ are isomorphic if 
there is an isomorphism $X\to X'$ mapping $D_i$ onto $D_i'$ hence mapping $\Delta$ onto $\Delta'$ preserving the 
decomposition. When we are not concerned with the data $d,\alpha$ and the marking we refer to $(X,\Delta)$ just as a locally stable pair.

In general, $\Delta$ can possibly be written as $\sum a_iD_i$ in several ways, so the underlying 
slc pair can be marked in several ways giving distinct marked pairs. 
For example, let $X=\PP^2$ and let $L,C$ be a line and a smooth 
conic, respectively, intersecting transversally, and let $\Delta=\frac{1}{2}L+\frac{1}{2}C$. 
Letting $\alpha=(\frac{1}{4}, \frac{1}{4})$, we can mark $\Delta$ by taking $D_1=2L,D_2=2C$, or 
taking $D_1=2C,D_2=2L$, or taking $D_1=D_2=L+C$. Thus we have at least three mutually non-isomorphic 
$(2,\alpha)$-marked pairs with the same underlying pair $(X,\Delta)$.    

Next we recall the definition of marked locally stable families \cite[Definitions 4.7, 8.4, and 8.7]{K21} in the projective setting.
Let $S$ be a reduced scheme over $k$. 
A \emph{$(d,\alpha)$-marked locally stable family} $f\colon (X,\Delta)\to S$ is given by a 
morphism $f\colon X\to S$ of schemes and closed subschemes $D_i\subset X$ for $i=1,\dots,m$ where  
\begin{itemize}
\item $f$ is flat and projective and its fibres are reduced, geometrically connected, $S_2$, of pure dimension $d$ and with nodal codimension one singularities,

\item $D_i$ are relative Mumford divisors over $S$,  

\item $\Delta=\sum a_iD_i$,

\item $K_{X/S}+\Delta$ is $\Q$-Cartier, and 

\item for each $s\in S$, $(X_s,\Delta_s)$ is an slc pair over the residue field $k(s)$ where 
$\Delta_s=\sum a_iD_{i,s}$.
\end{itemize} 
In particular, the log fibres $(X_s,\Delta_s)$ are $(d,\alpha)$-marked locally stable pairs. 
Note that in the above setting the dualising sheaf $\omega_{X/S}$ exists and 
commutes with base change \cite[2.68]{K21}. 
Moreover, since codimension one singularities of the fibres of $f$ are nodal,  
there is an open subset $V\subseteq X$ such that codimension of $X_s\setminus V_s$ 
in $X_s$ is $\ge 2$ and such that 
$\omega_{X/S}$ is locally free on $V$; thus   
$\omega_{X/S}$ corresponds to a canonical divisor class $K_{X/S}$; therefore, given the assumptions 
it makes sense to say that $K_{X/S}+\Delta$ is $\Q$-Cartier.

When we are not concerened with the data $d,\alpha$, we just say that $f\colon (X,\Delta)\to S$ is a locally stable family. Moreover, given a morphism $T\to S$ of reduced schemes, we get the induced locally stable family $(X_T,\Delta_T)\to T$ where $X_T=T\times_SX$ and $\Delta_T$ is defined by divisorial pullback.\\

\subsection{Moduli of strongly embedded locally stable pairs}

 We now recall the definition of embedded versions of the above notions \cite[\S 8.5]{K21}. Given a reduced scheme $S$ over $k$, a 
\emph{strongly embedded $(d,\alpha,\PP^n)$-marked locally stable family} 
$$
f\colon (X\subset \PP^n_S,\Delta)\to S
$$
is a $(d,\alpha)$-marked locally stable family $f\colon (X,\Delta)\to S$ together with a closed 
embedding $g\colon X\to \PP^n_S$ such that 
\begin{itemize}
\item $f=\pi g$ where $\pi$ denotes the projection $\PP^n_S\to S$, and 

\item letting $\mathcal{L}=g^*\mathcal{O}_{\PP^n_S}(1)$, we have 
$f_*\mathcal{L}\simeq \pi_*\mathcal{O}_{\PP^n_S}(1)$ and $R^qf_*\mathcal{L}=0$ for $q>0$.
\end{itemize}

\bigskip 

%If in addition,
%\begin{itemize}
%\item for each $s\in S$, $K_{X_s}+\Delta_s$ is ample with $\vol(K_{X_s}+\Delta_s)=v$,
%\end{itemize} 
%then we say the family is a \emph{strongly embedded $(d,\alpha,v,\PP^n)$-marked stable family}.\\

Consider the moduli functor $\mathfrak{MLS^{\rm e}}(d,\alpha,\PP^n)$ of 
strongly embedded $(d,\alpha,\PP^n)$-marked locally stable pairs 
from the category of reduced $k$-schemes to the category of 
sets by setting 
$$
\mathfrak{MLS^{\rm e}}(d,\alpha,\PP^n)(S)=\{\mbox{strongly embedded $(d,\alpha,\PP^n)$-marked}
$$
$$
\hspace{8cm}\mbox{ locally stable families over $S$}\}.
$$
 For a morphism $T\to S$ of reduced schemes, the map 
$$
\mathfrak{MLS^{\rm e}}(d,\alpha,\PP^n)(S)\to \mathfrak{MLS^{\rm e}}(d,\alpha,\PP^n)(T)
$$ 
is given via base change: given a family 
$$
(X\subset \PP^n_S,\Delta=\sum a_iD_i)\to S,
$$ 
we take the divisorial pullback of each $D_i$ to $X_T:=X\times_ST$ which then determines $\Delta_T$ with an $\alpha$-marking, 
hence a strongly embedded $(d,\alpha,\PP^n)$-marked locally stable family 
$$
(X_T\subset \PP^n_T,\Delta_T)\to T.
$$  
By \cite[Proposition 8.47 and Theorem 4.42]{K21}, we have the following. 
 
\begin{thm}\label{t-moduli-embedded-marked-stable-pairs}
The moduli functor $\mathfrak{MLS^{\rm e}}(d,\alpha,\PP^n)$ is represented by a reduced separated  
 scheme ${M\!LS^{e}}(d,\alpha,\PP^n)$ over $k$.
\end{thm}

Letting $M={M\!LS^{e}}(d,\alpha,\PP^n)$, the theorem says that there is a strongly embedded $(d,\alpha,\PP^n)$-marked locally stable family
$$
(X\subset \PP^n_M,\Delta=\sum a_iD_i)\to M
$$ 
which is a universal family for the functor $\mathfrak{MLS^{\rm e}}(d,\alpha,\PP^n)$, that is, every strongly embedded $(d,\alpha,\PP^n)$-marked locally stable family is the pullback of this universal family.

If we restrict the functor $\mathfrak{MLS^{\rm e}}(d,\alpha,\PP^n)$ to families 
$$
(X\subset \PP^n_S,\Delta=\sum a_iD_i)\to S,
$$ 
in which over each point $s\in S$ the log fibres $(X_s\subset \PP^n_{s},\Delta_s)$ have bounded degree 
with respect to $\mathcal{O}_{\PP^n_{s}}(1)$, that is, degree of both $X_s$ and $\Delta_s$ are bounded 
by some fixed number, then the corresponding moduli space (which is an open subscheme of 
${M\!LS^{e}}(d,\alpha,\PP^n)$) is of finite type over $k$. In our situation 
below, the degrees will be bounded. 

We recall a rough sketch of the proof of the theorem for convenience. First we consider $X$ marked with 
$m$ integral divisors $D_1,\dots,D_m\ge 0$ where $X$ is a projective, reduced, connected,  
$S_2$ scheme over $k$ of pure dimension $d$ strongly embedded into $\PP^n$ and no irreducible component of 
any $D_i$ is contained in the singular locus of $X$ which means $D_i$ are Mumford divisors. 
Families of such marked varieties can be defined using Mumford divisors. 
Such $X$ are parametrised by a reduced scheme $V$ derived from the Hilbert scheme of $\PP^n$, admitting a 
universal family ${\bf{V}}\to V$. To get a universal family for marked schemes, 
we use the universal family of Mumford divisors for ${\bf{V}}\to V$ as in \ref{ss-Mumford-div}, say 
$$
{\bf{MDiv}}({\bf{V}}/V) \to {M\!Div}({\bf V}/V).
$$ 
Since there are $m$ marked divisors, we need to take 
$$
{MV}={M\!Div}({\bf V}/V)\times_{V} \cdots \times_{V} {M\!Div}({\bf V}/V)\to \rm V
$$ 
and the induced universal family ${\bf{MV}}\to MV$, where the fibred product is taken $m$ times, 
and $\bf{MV}$ has $m$ marked divisors ${\bf D}_{1},\dots, {\bf D}_{m}$. The fibres which are strongly embedded are given by some open subset $MV^e$. 
Next, applying \cite[Theorem 4.42]{K21} to the universal family over $MV^e$, there is a locally closed partial decomposition ${M\!LS^{e}}(d,\alpha,\PP^n)\to MV^e$ 
such that pulling back the universal family over ${M\!LS^{e}}(d,\alpha,\PP^n)$ 
we get a universal family for strongly embedded $(d,\alpha,\PP^n)$-marked locally stable pairs, 
hence we get a fine moduli space for the functor $\mathfrak{MLS^{\rm e}}(d,\alpha,\PP^n)$. 
It is in this last step where the coordinates of $\alpha$ and the local stability condition come into the 
construction.

%%%%%%%%%%%%%%%%%%%%%%%%%%%%%
%%%%%%%%%%%%%%%%%%%%%%%%%%%%%

\section{\bf Good minimal models in families}

In this section, we aim to study good minimal models in locally stable families, satisfying various properties. We will use these in the next section when constructing moduli of embedded pairs which is in turn used to prove existence of coarse moduli spaces in a later section. 

\subsection{Lc pairs in families}

\begin{lem}\label{l-l.stable-family-lc}
Assume that $(X,B)\to S$ is a locally stable family where $S$ is a reduced $k$-scheme of finite type. 
Let $S'$ be the set of (not necessarily closed) points $s\in S$ such that $(X_s,B_s)$ is lc. Then $S'$ is a locally closed subset of $S$ (with reduced structure).
\end{lem}
\begin{proof} 
 To show that $S'$ is locally closed we need to show that $\overline{S'}\setminus S'$ is a closed subset of $S$ where $\overline{S'}$ is the closure of $S'$ in $S$. 
Replacing $S$ with $\overline{S'}$ and replacing $(X,B)$ with its base change (as defined in the previous section), we can assume that $S'$ is dense in $S$. 
Let $T$ be an irreducible component of $S$ and assume that there is a non-empty open subset $U\subset T$ such that $U\subseteq S'$. Let $R=S\setminus U$. Then $S\setminus S' \subseteq R$, hence 
$R\setminus (S'\cap R)=S\setminus S'$. Thus it is enough to show that $R\setminus (S'\cap R)$ is a closed subset of $R$. 

Let $(X_R,B_R)\to R$ be the locally stable family obtained by base change to $R$. Sicne $R$ is a closed subset in $S$, 
for each $s\in R$, the log fibre of $(X_R,B_R)\to R$ over $s$ coincides with the log fibre of $(X,B)\to S$ over $s$. So $S'\cap R$ is the set of points in $R$ over which the log fibre is lc. 
Thus applying Noetherian induction we can assume that $R\setminus (S'\cap R)$ is closed in $R$, hence $S\setminus S'$ is closed in $S$ as desired. Therefore, replacing $S$ with $T$, we can assume that $S$ is irreducible and then it is enough to show that $S'$ contains some non-empty open subset of $S$. We can shrink $S$ to any non-empty open subset if necessary. In particular, we can assume $S$ is smooth and that $(X,B)$ is slc by \cite[Theorem 4.54]{K13}. 

Let $P$ be the singular locus of $X$. Then $P_s$ is a subset of the singular locus of $X_s$ because $S$ is smooth. 
Since the fibres $X_s$ are normal for $s\in S'$, $P_s$ has codimension $\ge 2$ in $X_s$ for such $s$.
Thus since $S'$ is dense, we can assume that $P$ has codimension $\ge 2$ in $X$ meaning that $X$ is smooth in codimension $1$. But then $X$ is normal and $(X,B)$ is lc as $(X,B)$ is slc. Therefore, $S'$ contains a non-empty open subset of $S$.

\end{proof}

\subsection{Good minimal models in families}

\begin{lem}\label{l-family-dense-lc-s-ample}
$(1)$
Assume that 
\begin{itemize}
\item $X$ is a normal variety and $B\ge 0$ is a $\Q$-divisor, 

\item $X\to S$ is a contraction onto a variety, 

\item $\Pi\subset S$ is a dense set of closed points, and

\item  $(X_s,B_s)$ is an lc pair for general  $s\in \Pi$ where $B_s=B|_{X_s}$.  
\end{itemize} 
Then there exists a non-empty open set $U\subset S$ such that $(X,B)$ is lc over $U$.

$(2)$
If in addition 
\begin{itemize}
\item  $K_{X_s}+B_s$ is semi-ample for each  $s\in \Pi$, 
\end{itemize}
then we can choose $U$ so that $K_X+B$ is semi-ample over $U$.
\end{lem}
\begin{proof} 
(1)
Shrinking $S$ we can assume that $f$ is flat, $S$ is smooth, and that $B_s=B|_{X_s}$ is well-defined for every $s\in S$, where as usual $X_s$ is the fibre of $X\to S$ over $s$. We can assume $(X_s,B_s)$ is lc for every $s\in \Pi$. Note that we are not assuming $K_X+B$ to be $\Q$-Cartier.

Let $\phi\colon X'\to X$ be a log resolution and let $B'$ be the sum of the birational 
transform of $B$ and the prime exceptional divisors of $\phi$ which are horizontal over $S$. Let $s\in S$ 
be a closed point and let $X_s',X_s$ be the fibres of $X'\to S$ and $X\to S$ over $s$. By  \cite[Lemma 2.5]{B20},
shrinking $S$ we can assume that $X_s'$ is smooth and $X_s$ is normal, that $\Supp B'$ does not contain
any irreducible component of any fibre so $B_{s}':=B'|_{X_s'}$ is well-defined, and that $B_{s}'$ 
is the sum of the birational transform of $B_s$ and the reduced exceptional divisor of 
the induced morphism $\psi\colon X_s'\to X_s$. Moreover, we can assume that $(X',B')$ is lc because the horizontal coefficients of $B$ do not exceed $1$, and that $(X',B')$ is relatively log smooth over $S$. In particular, $(X_s',B_s')$ is lc.

Since $(X_s,B_s)$ is lc for each $s\in \Pi$, adding a general sufficiently ample divisor $A$ on $X$ to $B$ we can assume that $K_{X_s}+B_s$ is ample for every $s\in \Pi$ by the cone theorem. 
Then $(X_s,B_s)$ is the lc model of $(X_s',B_s')$ for such $s$. So we can run an MMP on 
$K_{X_s'}+B_{s}'$ ending with a dlt model of $(X_s,B_s)$, by \cite[Theorem 1.8]{B12}, which means that $(X_s',B_s')$ has a good minimal model. Therefore, by \cite[Theorem 1.2]{HMX18}, $(X',B')$ has a good minimal model over $S$, and by \cite[Theorem 1.9]{B12} it can be obtained by running an MMP on $K_{X'}+B'$ over $S$. So $(X',B')$ has an lc model $(Y,B_Y)$ over $S$. Since we added a general sufficiently ample divisor to $B$, the induced map $Y\bir X$ is a morphism by the cone theorem. Shrinking $S$ we can assume that $(Y_s,B_{Y,s})$ is the lc model of $(X_s',B_s')$. But $({X_s},B_s)$ is also the lc model of $(X_s',B_s')$ for $s\in \Pi$ as mentioned above, hence the induced map $Y_s\to X_s$ is an isomorphism for such $s$. Shrinking $S$ again we can assume that $Y\to X$ is an isomorphism because it cannot contract any curve inside $Y_s$ for $s\in \Pi$, which implies that $(X,B)$ is lc. 

(2)
By the above arguments, shrinking $S$ we can assume that $(X,B)$ is lc. 
Taking a $\Q$-factorial dlt model we can assume that $(X,B)$ is $\Q$-factorial dlt. 
Let $(X',B')$ be as above.
Shrinking $S$, for $s\in \Pi$,  $(X_s,B_s)$ is lc and $K_{X_s}+B_s$ is semi-ample by assumption. Then $(X_s,B_s)$ is a good weak lc model of $(X_s',B_s')$. So applying \cite{B12} similar to above, $(X_s',B_s')$ has a good minimal model. Therefore, by \cite[Theorem 1.2]{HMX18}, $(X',B')$ has a good minimal model over $S$. But then we can run an MMP on $K_X+B$ ending with a good minimal model, say $(V,B_V)$, by \cite[Remark 2.8 and Theorem 1.9]{B12} where we use the fact that $(X,B)$ is $\Q$-factorial dlt.
Since $K_{X_s}+B_s$ is nef for $s\in \Pi$, no extremal ray in the MMP intersects $X_s$.
Therefore, shrinking $S$ we can assume that $X\bir V$ is an isomorphism near $X_s$  for $s\in \Pi$. Thus shrinking $S$ again we can assume that $X=V$, hence $K_X+B$ is semi-ample over $S$.

\end{proof}

\begin{lem}\label{l-l.stable-family-s-ample}
Assume that $(X,B)\to S$ is a locally stable family where $S$ is a reduced $k$-scheme of finite type. 
Let $S'$ be the set of (not necessarily closed) points $s\in S$ such that $K_{X_s}+B_s$ is semi-ample.
Then $S'$ is a locally closed subset of $S$ (with reduced structure).
\end{lem}
\begin{proof}
Recall that we defined locally stable families only in the projective setting, so $X\to S$ is assumed projective. Arguing as in the proof of Lemma \ref{l-l.stable-family-lc}, it is enough to show that $S'$ contains a non-empty open subset assuming $S$ is irreducible. We can shrink $S$ to any non-empty open subset if necessary.  

We claim that there is a subset $\Pi\subset S'$ 
of closed points which is dense in $S$.  
Let $s\in S'$ and let $\sigma$ be its closure in $S$. It is enough to show that there is a non-empty open subset $R$
of $\sigma$ such that $R\subset S'$. If $\sigma\neq S$, then the claim follows from Noetherian induction.
Assume then that $\sigma=S$, that is, $s$ is the generic point of $S$. By assumption, $K_{X_s}+B_s$ is semi-ample. Therefore, $K_X+B$ is semi-ample over a non-empty open subset of $S$, by Lemma \ref{l-generic-semi-ampleness}, which implies existence of $R$, and this in turn ensures existence of $\Pi$.

Shrinking $S$ we will assume $S$ is smooth. Then $(X,B)$ is slc.
Let $X^\nu$ be the normalisation of $X$ and let $K_{X^\nu}+B^\nu$ be the pullback of $K_X+B$. 
If $s\in \Pi$ is general and if $(X_s,B_s)$ and $(X_s^\nu,B_{s}^\nu)$ are the log fibres of 
$(X,B)$ and $(X^\nu,B^\nu)$ over $s$, respectively, then $K_{X_s^\nu}+B_s^\nu$ is semi-ample because $K_{X_s}+B_s$ is semi-ample by assumption.
Taking the Stein factorisation of $X^\nu\to S$ and applying Lemma \ref{l-family-dense-lc-s-ample} to the 
irreducible components of $X^\nu$ and shrinking $S$ we can assume that $K_{X^\nu}+B^\nu$ is semi-ample over $S$. 
 Now by \cite{HX16}, $K_{X}+B$ is semi-ample over $S$ which implies the lemma.
 
\end{proof}

\subsection{Vertical volume in families}

\begin{lem}\label{l-l.stable-family-vertical-volume}
Assume that $(X,B),A\to S$ is a family of stable minimal models (as in \ref{defn-stable-minimal-family}) where $S$ is a reduced $k$-scheme of finite type. 
For each $s\in S$, let $X_s\to Z_s$ be the contraction defined by the semi-ample divisor $K_{X_s}+B_s$. Given $\Gamma\subset \Q^{>0}$, let $S'$ be the set of (not necessarily closed) points $s\in S$ such that $\vol(A_s|_{F})\in \Gamma$ for the general fibres $F$ of $X_s\to Z_s$ over each irreducible component of $Z_s$. Then $S'$ is a locally closed subset of $S$ (with reduced structure).
\end{lem}
\begin{proof}
As in the proofs above we can assume that $S'$ is dense in $S$, that $S$ is irreducible, and that it is enough to show that $S'$ contains a non-empty open subset of $S$. 
Moreover, we can assume that $S$ is smooth hence that $K_{X/S}+B$ is semi-ample over $S$ defining a contraction $g\colon X\to Z$. In particular, $A$ is ample over $Z$ as $K_{X/S}+B+tA$ is ample over $S$ for every small $t>0$, by definition of families of stable minimal models. 

Note that $X$ is reduced because $S$ is reduced and all the fibres of $X\to S$ are reduced.
So $Z$ is also reduced as $g$ is a contraction.
There is a dense open subset $U\subset Z$ so that $g$ is flat over $U$, by generic flatness. Then $\vol(A|_F)$ is locally constant on the connected components of $U$ where $F$ is a fibre of $X\to Z$ over some point of $U$, by Lemma \ref{l-vol-loc-constant}. 

Since $U$ is dense in $Z$, the complement $P=Z\setminus U$ does not contain any irreducible component of $Z$. Thus shrinking $S$, we can assume that for each $s\in S$, 
 every irreducible component of $U$ intersects $Z_s$. Moreover, we can assume that $U_s$ is dense in $Z_s$: indeed, $Z_s$ is the union of the $T_s$ where $T$ ranges through the irreducible components of $Z$;  shrinking $S$, $P_s\cap T$ is a proper closed subset of $T_s$ and the components of $T_s$ and $P_s\cap T$ are of different dimension; this just means that $U\cap T_s$ is dense in $T_s$; so $U_s$ is dense in $Z_s$. 
 
 If $s\in S'$, then by assumption $\vol(A_s|_{F})\in \Gamma$ for any general fibre $F$ of $X_s\to Z_s$ over any component of $Z_s$. Thus $\vol(A_s|_{F})\in \Gamma$ for any fibre $F$ of $X_s\to Z_s$ over any point of $U_s$ because $X_s\to Z_s$ is flat over $U_s$ and $U_s$ is dense in $Z_s$. Since each irreducible component of $U$ intersects $Z_s$ and since $U_s$ is dense in $Z_s$, we deduce that $\vol(A|_F)\in \Gamma$ for every fibre $F$ over $U$. Therefore, $\vol(A_s|_{F})\in \Gamma$ for the general fibres $F$ of $X_s\to Z_s$ over any component of $Z_s$ for each $s\in S$ because $U_s$ is dense in $Z_s$.

\end{proof}

\subsection{Coincidence with invertible sheaf}

\begin{lem}\label{l-l.stable-family-inv.sheaf}
Assume that $(X,B)\to S$ is a locally stable family where $S$ is a reduced $k$-scheme of finite type. Assume also that $\mathcal{L}$ is an invertible sheaf on $X$ and $r\in \N$. 
Let $S'$ be the set of (not necessarily closed) points $s\in S$ such that $r(K_{X_s}+B_s)$ is Cartier and 
$\mathcal{O}_{X_s}(r(K_{X_s}+B_s))\simeq \mathcal{L}_s$. Then $S'$ is a locally closed subset of $S$ (with reduced structure).
\end{lem}
\begin{proof}
Replacing $S$ we can assume $S'$ is dense in $S$. 
By Neotherian induction it is enough to replace $S$ with any of its irreduible components and then show that $S'$ contains some non-empty open subset of $S$.
By \cite[Theorem 4.28]{K21}, there is a locally closed partial decomposition $S''\to S$ satisfying the following: for a morphism of schemes $T\to S$, the divisorial pullback of $r(K_{X/S}+B)$ to $T\times_SX$ is Cartier iff $T\to S$ factors through $S''$. In particular, for each $s\in S'$, $\Spec k(s)\to S$ factors through $S''$. Since $S'$ is dense in $S$ we deduce that some component of $S''$ contains a non-empty open subset of $S$, hence replacing $S$ with this open set we can assume that $r(K_{X/S}+B)$ is Cartier.

Now the lemma follows from \cite[Lemma 1.19]{V95}.
Here we are using the following fact: since the fibres $X_s$ of $X\to S$ are reduced and connected, we have $h^0(\mathcal{O}_{X_s})=1$. 

\end{proof}

We will also use the following lemma. 

\begin{lem}\label{l-l.stable-family-vanishing}
Assume that $X\to S$ is a flat projective morphism where $S$ is a reduced $k$-scheme of finite type, and   
$\mathcal{L}$ is an invertible sheaf on $X$. 
Let $S'$ be the set of (not necessarily closed) points $s\in S$ such that $h^q(\mathcal{L}_s)=0$ for every $q>0$. Then $S'$ is an open subset of $S$.
\end{lem}
\begin{proof}
Let $e$ be the maximum of dimension of the fibres of $X\to S$. 
The vanishing $h^q(\mathcal{L}_s)=0$ holds for every $s$ whenever $q>e$. On the other hand, if $s\in S'$, then for each $q>0$, $h^q(\mathcal{L}_t)=0$ holds for $t$ in some open neighbourhood $U_q$ of $s$, by upper semi-continuity of cohomology. Thus the vanishing holds for every $q>0$ in the open neighbourhood $\bigcap_{1\le q\le e} U_q$ of $s$. This shows that $S'$ is an open subset.
   
\end{proof}

%%%%%%%%%%%%%%%%%%%%%%%%%%%%%%%%%%%%%%%%%%%%%%%%%%
%%%%%%%%%%%%%%%%%%%%%%%%%%%%%%%%%%%%%%%%%%%%%%%%%%

\section{\bf Embedded stable minimal models}\label{s-strongly-embedded-s.m.model} 

In this section, we construct fine moduli spaces of embedded stable minimal models which are used to construct the coarse moduli spaces discussed in the next section. We follow a standard strategy in the moduli theory of varieties and borrow some notation and a result from \cite{K21}.

\subsection{Base change of families}\label{ss-base-change-fam.s.m.model}
Assume 
$$
f\colon (X,B),A\to S
$$
is a family of stable minimal models (as in \ref{defn-stable-minimal-family}) where $S$ is a reduced scheme over $k$. Let $S'\to S$ be a morphism of reduced schemes. 
Then base change and divisorial pullback as in Section \ref{s-locally-stable-families} gives the induced locally stable family  
$$
f'\colon (X',B'),A'\to S'.
$$
Since 
$$
(X,B+tA)\to S
$$ 
is locally stable for small $t\ge 0$, 
$$
(X',B'+tA')\to S'
$$ 
is also locally stable for small $t\ge 0$. 
Pick $s'\in S'$ and let $s\in S$ be its image. Since $(X_s,B_s),A_s$ is a stable minimal model over $k(s)$, $(X_{s'}',B_{s'}'),A_{s'}'$ is a stable minimal model over $k(s')$ where we use the fact that $X_{s'}'\simeq X_s\times_{\Spec k(s)}k(s')$. This shows that $f'$ is a  family of stable minimal models. 

Now assume $d\in \N$, $c\in \Q^{\ge 0}$, $\Gamma\subset  \Q^{>0}$, $\sigma\in \Q[t]$. 
As before we use the notation $\Phi_c=c\Z^{\ge 0}$.
Assume $f$ above is a family of 
$(d,\Phi_c,\Gamma,\sigma)$-stable minimal models. 
Then $B=cD$ and $A=cN$ where $D,N$ are relative Mumford divisors, hence $B'=cD'$ and $A'=cN'$ where $D',N'$ are relative Mumford divisors. 

Pick $s'\in S'$ and let $s\in S$ be its image.   
Since $(X_s,B_s),A_s$ is a $(d,\Phi_c,\Gamma,\sigma)$-stable minimal model over $k(s)$, 
$$
\sigma(t)=\vol(K_{X_s}+B_s+tA_s)
$$
for small $t>0$, hence 
$$
\sigma(t)=\vol(K_{X_{s'}'}+B_{s'}'+tA_{s'}')
$$
for small $t>0$, by Lemma \ref{l-base-change-field}. 

If $X_s\to Z_s$ is the contraction defined by $K_{X_s}+B_s$, then the base change $X_{s'}'\to Z_{s'}'$ is the contraction defined by $K_{X_{s'}'}+B_{s'}'$.
Also each irreducible component of $Z_{s'}'$ maps onto an irreducible component of $Z_s$.  
If $F'$ is a general fibre of $X_{s'}'\to Z_{s'}'$ over some point $z'$ of some irreducible component of $Z_{s'}'$, then $F'\simeq F\times_{\Spec k(z)}k(z')$ where $z\in Z_s$ is the image of $z'$ and $F$ is the fibre of $X_s\to Z_s$ over $z$ (here $F$ is a general fibre of $X_s\to Z_s$ over some irreducible component of $Z_s$), hence 
$$
\vol(A_{s'}'|_{F'})=\vol(A_s|_{F})\in \Gamma
$$ 
again by Lemma \ref{l-base-change-field}. Therefore, 
$(X_{s'}',B_{s'}'),A_{s'}'$ is a $(d,\Phi_c,\Gamma,\sigma)$-stable minimal model over $k(s')$. 
This in turn shows that $f'$ is a family of 
$(d,\Phi_c,\Gamma,\sigma)$-stable minimal models.

\subsection{Strongly embedded stable minimal models}

\begin{defn}
Let $d,r,n\in \N$, $c\in \Q^{\ge 0}$, $\Gamma\subset  \Q^{>0}$, $a\in \Q^{>0}$, and $\sigma\in \Q[t]$. 
To simplify notation let 
$$
\Xi=(d,\Phi_c,\Gamma,\sigma,a,r,\PP^n).
$$ 
(1) 
Let $S$ be a reduced scheme over $k$. 
A \emph{strongly embedded family of $\Xi$-stable minimal models over $S$} is a family of $(d,\Phi_c,\Gamma,\sigma)$-stable minimal models $f\colon (X,B),A\to S$ 
 together 
with a closed embedding $g\colon X\to \PP^n_S$ such that 
\begin{itemize}
\item $(X,B+aA)\to S$ is locally stable,

\item $f=\pi g$  where $\pi$ denotes the projection $\PP^n_S\to S$, 

\item letting $\mathcal{L}:=g^*\mathcal{O}_{\PP^n_S}(1)$, we have   
$$
R^qf_*\mathcal{L}\simeq R^q\pi_*\mathcal{O}_{\PP^n_S}(1)
$$ 
for all $q$, and

\item for every $s\in S$, we have 
$$
\mathcal{L}_s\simeq \mathcal{O}_{X_s}(r(K_{X_s}+B_s+aA_s)).
$$
\end{itemize}
We denote the family by 
$$
f\colon (X\subset \PP^n_S,B),A\to S.
$$

(2)
When $(X_s,B_s)$ is lc for every $s\in S$, we say the family is a \emph{family of strongly embedded lc $\Xi$-stable minimal models over $S$}. 

(3)
Define the functor $\mathfrak{S^{\rm e}_{\rm slc}}\Xi$ on the category of reduced $k$-schemes by setting 
$$
\mathfrak{S^{\rm e}_{\rm slc}}\Xi(S)=\{\mbox{families of strongly embedded $\Xi$-stable minimal models over $S$}\}.
$$
Similarly define the functor $\mathfrak{S^{\rm e}_{\rm lc}}\Xi$  by setting 
$$
\mathfrak{S^{\rm e}_{\rm lc}}\Xi(S)=\{\mbox{families of strongly embedded lc $\Xi$-stable minimal models over $S$}\}.
$$
Given a morphism of reduced schemes $S'\to S$, the map 
$$
\mathfrak{S^{\rm e}_{\rm slc}}\Xi(S)\to \mathfrak{S^{\rm e}_{\rm slc}}\Xi(S')
$$ 
is defined as in \ref{ss-base-change-fam.s.m.model} (similarly for $\mathfrak{S^{\rm e}_{\rm lc}}\Xi$). 
To be more precise, assume 
$$
f\colon (X\subset \PP^n_S,B),A\to S
$$
is in $\mathfrak{S^{\rm e}_{\rm slc}}\Xi(S)$.
Then base change and divisorial pullback gives the family   
$$
f'\colon (X',B'),A'\to S'.
$$ 
of $(d,\Phi_c,\Gamma,\sigma)$-stable minimal models 
which comes equipped with an induced embedding $g'\colon X'\to \PP^n_{S'}$. Since 
$$
(X,B+aA)\to S
$$ 
is locally stable, 
$$
(X',B'+aA')\to S'
$$ 
is locally stable. Letting $\mathcal{L}':=g'^*\mathcal{O}_{\PP^n_{S'}}(1)$, we have 
$R^qf'_*\mathcal{L'}\simeq R^q\pi'_*\mathcal{O}_{\PP^n_{S'}}(1)$ for all $q$ by base change of cohomology. Moreover, 
for every $s'\in S'$, we have 
$$
\mathcal{L'}_{s'}\simeq \mathcal{O}_{X_{s'}'}(r(K_{X_{s'}'}+B_{s'}'+aA_{s'}')).
$$  
Thus we get
$$
f'\colon (X'\subseteq \PP^n_{S'},B'),A'\to S'
$$
in $\mathfrak{S^{\rm e}_{\rm slc}}\Xi(S')$.
\end{defn}
\smallskip

\subsection{Moduli of strongly embedded stable minimal models}

\begin{prop}\label{p-moduli-se-pol-cy}
With the above notation, the functors $\mathfrak{S^{\rm e}_{\rm slc}}\Xi$ and $\mathfrak{S^{\rm e}_{\rm lc}}\Xi$ are represented by reduced separated $k$-schemes ${M^{\rm e}_{\rm slc}}\Xi$ and ${M^{\rm e}_{\rm lc}}\Xi$ of finite type over $k$.
\end{prop}
\begin{proof}
First we focus on the functor $\mathfrak{S^{\rm e}_{\rm slc}}\Xi$. In the end we treat $\mathfrak{S^{\rm e}_{\rm lc}}\Xi$. Recall that 
$$
\Xi=(d,\Phi_c,\Gamma,\sigma,a,r,\PP^n).
$$\ 
%First we treat the functor $\mathfrak{S^{\rm e}_{\rm slc}}\Xi$. In the end we discuss $\mathfrak{S^{\rm e}_{\rm lc}}\Xi$.\\
% and by boundedness of stable minimal models (via Lemma \ref{l-bnd-lct-v.ampleness}).\\

\emph{Step 1.}
Put $\alpha:=(c,ac)$. 
Recall the functor  $\mathfrak{MLS}^{\rm e}(d,\alpha,\PP^n)$ and its fine moduli space 
$$
M:={{M\!LS}}^{\rm e}({d,\alpha,\PP^n})
$$ 
from Theorem \ref{t-moduli-embedded-marked-stable-pairs}. 
There is a strongly embedded $(d,\alpha,\PP^n)$-marked locally stable universal family 
$$
(X\subset \PP^n_M, \Delta:=cD+acN)\to M
$$
where $D,N$ are relative Mumford divisors on $X$ over $M$.
The moduli space $M$ is very large, we only need parts of it related to the functor 
$\mathfrak{S^{\rm e}_{\rm slc}}\Xi$. More precisely, we want to identify the points of $M$ which parametrise 
strongly embedded $\Xi$-stable minimal models.\\  

\emph{Step 2}.
Suppose we are given a family of strongly embedded $\Xi$-stable minimal models
$$
(X'\subset \PP^n_{S'},B'),A'\to S'
$$
where $S'$ is a reduced $k$-scheme. We will show that there is a unique morphism $S'\to M$ so that the latter family is the pullback of the universal family over $M$, and 
there is an open subset $M^{(0)}\subset M$ of finite type over $k$ depending only on $\Xi$ such that  
$S'\to M$ factors through $M^{(0)}$.

The divisor $\Delta':=B'+aA'$ has a natural 
$\alpha$-marking because by assumption $B'=cD'$ and $A'=cN'$ for some relative Mumford divisors $D',N'$ over $S'$. Moreover, by definition $(X',\Delta')\to S'$ is locally stable, hence it is a $(d,\alpha)$-marked locally stable family. Thus
 taking into account the given embedding $X'\subset \PP^n_{S'}$ we see that the family 
$$
(X'\subset \PP^n_{S'},\Delta')\to S'
$$
 is a strongly embedded $(d,\alpha,\PP^n)$-marked locally stable family.  
Therefore, there is a unique morphism $S'\to M$ so that the latter family is the pullback of the universal family over $M$ (of Step 1) via $S'\to M$. 

By assumption, 
for every $s'\in S'$, we have 
$$
\mathcal{O}_{X_{s'}'}(1)\simeq \mathcal{O}_{X_{s'}'}(r(K_{X_{s'}'}+\Delta_{s'}'))
$$
where $\mathcal{O}_{X_{s'}'}(1)$ is the pullback of $\mathcal{O}_{\PP^n_{S'}}(1)$ to $X'_{s'}$.
Again by assumption, 
$$
\vol(K_{X_{s'}'}+B_{s'}'+tA_{s'}')=\sigma(t)
$$
for every small $t\ge 0$. Since $K_{X_{s'}'}+\Delta_{s'}'=K_{X_{s'}'}+B'_{s'}+aA'_{s'}$ is ample and $K_{X_{s'}'}+B_{s'}'$ is nef by assumption, we deduce that 
$$
\vol(K_{X_{s'}'}+B_{s'}'+tA_{s'}')=\sigma(t)
$$
for every $t\in [0,a]$ because both sides are polynomials in $t$ on the interval $[0,a]$ which agree for small values of $t$. In particular, 
$$
\vol(\mathcal{O}_{X_{s'}'}(1))=\vol(r(K_{X_{s'}'}+\Delta_{s'}'))=r^d\sigma(a)
$$
is fixed, hence the degrees of $X_{s'}'$ and $\Delta_{s'}'$ with respect to 
$\mathcal{O}_{X_{s'}'}(1)$ are bounded from above. Therefore, there is an open subscheme 
$M^{(0)}\subset M$ of finite type over $k$, depending only on $\Xi$, such that $S'\to M$ factors 
through $M^{(0)}$.\\

\emph{Step 3.}
In this step we apply a process to distinguish the strongly embedded $\Xi$-stable minimal models among the fibres of the universal family over $M$ of Step 1.

(0). First pull back the universal family over $M$ to a family over  $M^{(0)}$ and denote it by
$$
(X^{(0)}\subset \PP^n_{M{(0)}}, \Delta^{(0)}=cD^{(0)}+acN^{(0)})\to M^{(0)}.
$$
To simplify notation we will denote the log fibre of this family over a point $s$ by 
$$
(X_s\subset \PP^n_s,\Delta_s=cD_s+acN_s)
$$ 
that is, we drop the superscript $(0)$. 
We do similarly below.

(i). By definition of locally stable families, $K_{X^{(0)}/M^{(0)}}+\Delta^{(0)}$ is $\Q$-Cartier. The points $s$ of $M^{(0)}$ such that $K_{X_s}+\Delta_s$ is ample  
form an open subset, say $M^{(1)}$. Pull back the family over $M^{(0)}$ to a family over $M^{(1)}$ 
to get 
$$
(X^{(1)}\subset \PP^n_{M{(1)}}, \Delta^{(1)}=cD^{(1)}+acN^{(1)})\to M^{(1)}.
$$ 

(ii). By \cite[Theorem 4.8]{K21}, there is a locally 
closed partial decomposition $M^{(2)}\to M^{(1)}$ satisfying the following: given a morphism $S\to M^{(1)}$ 
from a reduced scheme, remove $acN^{(1)}$ in the family in (i) and 
then pull back 
$$
(X^{(1)}\subset \PP^n_{M{(1)}}, cD^{(1)})\to M^{(1)}
$$ 
to a family over $S$ using divisorial pullback introduced in Section \ref{s-locally-stable-families}, to get  
$$
(X_S\subset \PP^n_{S},cD_S)\to S
$$ 
which is a well-defined family of pairs according to \cite[Definition 4.3]{K21};
then $(X_S,cD_S)\to S$ is a locally stable family iff $S\to M^{(1)}$ factors 
through $M^{(2)}\to M^{(1)}$. 

Pull back the family in (i) to a family over $M^{(2)}$ to get 
$$
(X^{(2)}\subset \PP^n_{M{(2)}}, \Delta^{(2)}=cD^{(2)}+acN^{(2)})\to M^{(2)}.
$$
Then 
$$
(X^{(2)},cD^{(2)})\to M^{(2)}
$$
is a locally stable family. In particular, $K_{X^{(2)}/M^{(2)}}+cD^{(2)}$ is $\Q$-Cartier, hence $N^{(2)}$ is also $\Q$-Cartier. 

(iii).  
By Lemma \ref{l-l.stable-family-s-ample}, 
the set of points $s\in M^{(2)}$ such that $K_{X_s}+cD_s$
is semi-ample is locally closed, say $M^{(3)'}$. Pulling back the family over $M^{(2)}$ to a family over $M^{(3)'}$ we get a family
$$
(X^{(3)'}\subset \PP^n_{M{(3)'}}, cD^{(3)'}+acN^{(3)'})\to M^{(3)'}.
$$ 
Then 
$$
(X^{(3)'}\subset \PP^n_{M{(3)'}}, cD^{(3)'}), cN^{(3)'}\to M^{(3)'}
$$
is a family of stable minimal models because 
$$
(X^{(3)'}\subset \PP^n_{M{(3)'}},cD^{(3)'}+tcN^{(3)'})\to M^{(3)'}
$$ 
is locally stable for every $t\in [0,a]$ and because for each $s\in M^{(3)'}$ we have: 
\begin{itemize}
\item $({X_s},cD_s)$ is slc, 
\item $K_{X_s}+cD_s$ is semi-ample defining a contraction $X_s\to Z_s$, 
\item $K_{X_s}+cD_s+tcN_s$ is ample for every $t\in (0,a]$, and 
\item $({X_s},cD_s+tcN_s)$ is slc for every $t\in [0,a]$.
\end{itemize}

Now the set of points $s\in M^{(3)'}$ such that $\vol(cN_s|_{F})\in \Gamma$ for the general fibres $F$ of $X_s\to Z_s$ over each irreducible component of $Z_s$ is a locally closed subset $M^{(3)} \subset M^{(3)'}$, by Lemma \ref{l-l.stable-family-vertical-volume}.

Pull back the family over  $M^{(3)'}$ to a family over $M^{(3)}$ to get   
$$
(X^{(3)}\subset \PP^n_{M{(3)}}, \Delta^{(3)}=cD^{(3)}+acN^{(3)})\to M^{(3)}.
$$ 
By construction, $K_{X^{(3)}/M^{(3)}}+cD^{(3)}$ is nef over $M^{(3)}$.

(iv). 
Since $K_{X^{(3)}/M^{(3)}}+\Delta^{(3)}$ is ample over $M^{(3)}$ and $K_{X^{(3)}/M^{(3)}}+cD^{(3)}$ is nef over $M^{(3)}$, we see that for each $t\in (0,a]$,
$$
K_{X^{(3)}/M^{(3)}}+cD^{(3)}+tcN^{(3)}
$$ 
is ample over $M^{(3)}$. Thus for each $s\in M^{(3)}$, 
$$
\theta_s(t):=\vol(K_{X_s}+cD_s+tcN_s)=(K_{X_s}+cD_s+tcN_s)^d
$$
is a polynomial in $t$ of degree $\le d$ on the interval $(0,a]$. On the other hand, fixing $t\in (0,a]$, we see that 
$\theta_s(t)$ is a locally constant function on $M^{(3)}$, by Lemma \ref{l-vol-loc-constant}.

Pick numbers 
$$
0<t_1<\cdots <t_{d+1}=a.
$$  
Then the set of points $s\in M^{(3)}$ for which $\theta_s(t_i)=\sigma(t_i)$ for each $i$, is a closed and open subset of $M^{(3)}$. Denote this set by $M^{(4)}$. For each $s\in M^{(4)}$, $\theta_s(t)-\sigma(t)$ is a polynomial in $t$ of degree $\le d$ with $d+1$ distinct roots $t_1,\dots,t_{d+1}$ (it was noted in Step 2 that $\deg \sigma\le d$). Hence $\theta_s=\sigma$.
Thus we deduce that $M^{(4)}$ is the set of points $s\in M^{(3)}$ for which $\theta_s=\sigma$ on the interval $[0,a]$.

Pullback over $M^{(4)}$ gives the induced family 
$$
(X^{(4)}\subset \PP^n_{M{(4)}}, \Delta^{(4)}=cD^{(4)}+acN^{(4)})\to M^{(4)}
$$   
and 
$$
(X^{(4)}, cD^{(4)}),cN^{(4)}\to M^{(4)}
$$ 
which is a family of $(d,\Phi_c, \Gamma,\sigma)$-stable minimal models.
 
(v). 
Consider the set of points $s\in M^{(4)}$, say $M^{(5)}$, such that $r(K_{X_s}+\Delta_s)$ is Cartier and  
$$
\mathcal{O}_{X_s}(1) \simeq  \mathcal{O}_{X_s}(r(K_{X_s}+\Delta_s))
$$ 
where $\mathcal{O}_{X_s}(1)$ is the pullback of $\mathcal{O}_{\PP^n_{M^{(4)}}}(1)$ to $X_s$. 
Then $M^{(5)}$ is a locally closed subset of $M^{(4)}$, by Lemma \ref{l-l.stable-family-inv.sheaf}. 

The induced family  
$$
(X^{(5)}\subset \PP^n_{M{(5)}}, \Delta^{(5)}=cD^{(5)}+acN^{(5)})\to M^{(5)}
$$
is a strongly embedded $(d,\alpha,\PP^n)$-marked locally stable family.
Moreover, the family 
$$
(X^{(5)}\subset \PP^n_{M{(5)}}, cD^{(5)}), cN^{(5)}\to M^{(5)}
$$
is a family of strongly embedded $\Xi$-stable minimal models.\\

\emph{Step 4.} 
We claim that $M^{(5)}$ is the desired moduli space ${M^{\rm e}_{\rm slc}}\Xi$. 
We show that for any family of strongly embedded $\Xi$-stable minimal models
$$
h\colon (X'\subset \PP^n_{S'},B'),A'\to S',
$$
there is a unique morphism $S'\to M^{(5)}$ so that the family $h$ is the pullback of the family 
$$
(X^{(5)}\subset \PP^n_{M{(5)}}, cD^{(5)}),cN^{(5)}\to M^{(5)}.
$$
 As noted in Step 2, 
there is a unique morphism $S'\to M$ so that the marked family 
$$
(X'\subset \PP^n_{S'},\Delta':=B'+aA')\to S'
$$
is the pulback of the marked family
$$
(X\subset \PP^n_{M}, \Delta=cD+acN)\to M.
$$ 

By Step 2, $S'\to M$ factors through $M^{(0)}$. Since $K_{X'/S'}+\Delta'$ is ample over $S'$ by definition, $S'\to M^{(0)}$ factors through $M^{(1)}$, by Lemma \ref{l-base-change-field}(1). From the previous paragraph we see that 
$
(X',B')\to S'  
$
is the pulback of
$
(X, cD)\to M.
$ 
Since the former is a locally stable family by definition, $S'\to M^{(1)}$ factors through $M^{(2)}$. 

To check that $S'\to M^{(2)}$ factors through $M^{(i)}$ for $i\ge 3$, it is enough to show that the image of $S'\to M^{(2)}$ is contained in $M^{(i)}$ because $M^{(i)}$ is a locally closed subset of $M^{(i-1)}$ for $i\ge 3$ and so it is a locally closed subset of $M^{(2)}$. 

Pick an arbitrary point $s'\in S'$ and let $s\in M^{(2)}$ be its image. Let $(X_{s'}',B_{s'}'),A_{s'}'$ be the log fibre of the family 
$$
(X',B'),A'\to S'  
$$
over $s'$ and $(X_{s},cD_{s}),cN_{s}$ the log fibre of the family 
$$
(X^{(2)}, cD^{(2)}),cN^{(2)}\to M^{(2)}
$$
 over $s$. Here $X_{s'}'\simeq \Spec k(s')\times_{\Spec k(s)}X_s$.
 By definition of $(d,\Phi_c, \Gamma,\sigma)$-stable minimal models, $K_{X_{s'}'}+B_{s'}'$ is semi-ample defining a contraction $X_{s'}'\to Z_{s'}'$. Then by Lemma \ref{l-base-change-field}, $K_{X_{s}}+cD_{s}$ is semi-ample defining a contraction $X_{s}\to Z_{s}$. 
This shows $s\in M^{(3)'}$. Moreover, since $X_{s'}'\simeq Z_{s'}'\times_{Z_s}X_s$ (as we saw in the proof of \ref{l-base-change-field}) and since $\vol(A_{s'}'|_{F'})\in \Gamma$ for any general fibre of $X_{s'}'\to Z_{s'}'$ over any irreducible component of $Z_{s'}'$, we have $\vol(cN_{s}|_{F})\in \Gamma$ for any general fibre of $X_{s}\to Z_{s}$ over any irreducible component of $Z_{s}$, by Lemma \ref{l-base-change-field}(3), hence $s\in M^{(3)}$. Thus $S'\to M^{(2)}$ factors through $M^{(3)}$.

On the other hand,
$$
\vol(K_{X_{s'}'}+B_{s'}'+tA_{s'}')=\sigma(t)
$$ 
for any $t\in [0,a]$ by definition of $(d,\Phi_c, \Gamma,\sigma)$-stable minimal models and by the assumption that  $K_{X_{s'}'}+B_{s'}'+aA_{s'}'$ is ample. This implies that 
$$
\vol(K_{X_{s}}+cD_{s}+tcN_{s})=\sigma(t)
$$ 
for any $t\in [0,a]$, by Lemma \ref{l-base-change-field}(3), hence $S'\to M^{(3)}$ factors through $M^{(4)}$.  

In addition, 
$$
\mathcal{O}_{X_{s'}'}(1) \simeq  \mathcal{O}_{X_{s'}'}(r(K_{X_{s'}'}+\Delta_{s'}'))
$$ 
holds by definition where $\Delta_{s'}'=B_{s'}'+aA_{s'}'$. In particular,  $r(K_{X_{s'}'}+\Delta_{s'}')$ is Cartier. Then $r(K_{X_{s}}+\Delta_{s})$ is also Cartier. Therefore,  
$$
\mathcal{O}_{X_s}(1) \simeq  \mathcal{O}_{X_s}(r(K_{X_s}+\Delta_s))
$$ 
holds by Lemma \ref{l-base-change-field}(5). So $S'\to M^{(4)}$ factors through $M^{(5)}$ as claimed, hence ${M^{\rm e}_{\rm slc}}\Xi=M^{(5)}$. 
\\

\emph{Step 5.}
We now treat the functor $\mathfrak{S^{\rm e}_{\rm lc}}\Xi$. Consider the universal family 
$$
(X^{(5)}\subset \PP^n_{M{(5)}}, cD^{(5)}), cN^{(5)}\to M^{(5)}
$$
of Step 3. The set of points $s\in M^{(5)}$ such that $(X_s,cD_s)$ is lc is a locally closed subset $M^{(6)}\subset M^{(5)}$, by Lemma \ref{l-l.stable-family-lc}. Given a family of   
strongly embedded lc $\Xi$-stable minimal models
$$
h\colon (X'\subset \PP^n_{S'},B'),A'\to S',
$$
there is a unique morphism $S'\to M^{(5)}$ so that the family $h$ is the pullback of the family over $M^{(5)}$. Moreover, if $s'\in S'$ and if $s\in M^{(5)}$ is its image, then since  
$(X_{s'}',B_{s'}')$ is lc by assumption, we deduce that $(X_{s},cD_{s})$ is also lc because $X_{s'}'$ is normal so it is smooth in codimension one which implies $X_s$ is also smooth in codimension one so it is normal by Serre's criterion. This shows  
$s\in M^{(6)}$. Therefore, $S'\to M^{(5)}$ factors through $M^{(6)}$ which in turn implies ${M^{\rm e}_{\rm lc}}\Xi:=M^{(6)}$ is a fine moduli space for the functor $\mathfrak{S^{\rm e}_{\rm lc}}\Xi$ and the universal family over $M^{(6)}$ is the pullback of the above family over $M^{(5)}$.

\end{proof}

%%%%%%%%%%%%%%%%%%%%%
%%%%%%%%%%%%%%%%%%%%%

\section{\bf Coarse moduli spaces of stable minimal models}

In this section, we prove our main result on existence of moduli spaces, that is, Theorem \ref{t-moduli-s.min.models}. 

\subsection{Thresholds}

\begin{lem}\label{l-bnd-lct-v.ampleness}
Let $d\in\N$, $c\in \Q^{\ge 0}$, $\Gamma\subset \Q^{>0}$ be a finite set, and $\sigma\in\Q[t]$ be a polynomial.  
Then there exist a positive rational number $a$ and a natural number $r$ such that $rc,ra\in \N$ 
and they satisfy the following.  Assume $(X,B),A$ is a $(d,\Phi_c,\Gamma,\sigma)$-stable minimal model over a field $K$ of characteristic zero. 
Then 
\begin{itemize}
\item $(X,B+aA)$ is slc, and

\item $r(K_X+B+aA)$ is very ample with 
$$
h^q(mr(K_X+B+aA))=0
$$ 
for any $m,q>0$.
\end{itemize} 
\end{lem}
\begin{proof}
Recall that $\Phi_c=c\Z^{\ge 0}$.
By definition, the coefficients of $B$ and $A$ belong to the DCC set $\Phi_c$. It is enough to find $a,r$ when $K=\C$ by the Lefschetz principle. 
Indeed, we can replace $(X,B),A$ with a $(d,\Phi_c,\Gamma,\sigma)$-stable minimal model that is defined over some subfield of $\C$. Next we can extend that subfield to the whole $\C$.  
Note that properties such as being Cartier, very ample,  
vanishing of cohomology, all ascend and descend under field extensions in characteristic zero (such extensions are faithfully flat). 

In the case $K=\C$, the set of such $(X,B),A$ is bounded, by Theorem \ref{t-bnd-s-mmodels-slc}. So 
existence of $a$ follows from the theorem and its proof. Existence of $r$ also follows from the boundedness.

% For descent of being Cartier and more generally flatness of sheaves, see [EGA IV2, 2.5.1].
% For descend of very ampleness see [EGA IV2, 2.7.2].

\end{proof}

\subsection{Images of slc pairs}

\begin{lem}\label{l-image-slc-pairs}
Let $(X,B)$ be an slc pair and $f\colon X\to Z$ be a contraction such that $Z$ is pure dimensional and $K_X+B\sim_\Q 0/Z$. Then any codimension one component of the singular locus of $Z$ is the image of a codimension one component of the singular locus of $X$.
\end{lem}
\begin{proof}
Note that by a codimension one component of the singular locus of $Z$ we mean a component that has codimension one in $Z$ (similarly for $X$).
Assume $T$ is a codimension one component of the singular locus of $Z$. 
Since the lemma is local over $Z$, we can shrink $Z$ around the generic point $\eta_T$, hence assume $Z$ is connected in which case $X$ is also connected as $f$ is a contraction. We can then assume $X$ is pure dimensional. 

Shrinking $Z$, we can assume that $\eta_T$ is contained in the image of every irreducible component of $X$, and that if $X_1,X_2$ are components of $X$, then either $X_1\cap X_2=\emptyset$ or that $\eta_T$ is contained in the image of every component of $X_1\cap X_2$. If some component $X_i$ of $X$ maps onto $T$, then we can find another component $X_j$ and a codimension one component of the singular locus of $X$ inside $X_i\cap X_j$ (note that $X$ remains connected if any codimension $2$ closed subset is removed), which maps onto $T$ by the previous sentence. Thus we can assume that each component of $X$ maps onto a component of $Z$. 
   
Assume that $Z$ is not irreducible near $\eta_T$, say $Z_1,\dots,Z_r$ are its irreducible components near $\eta_T$.    
Let $Y_i$ be the union of the irreducible components of $X$ mapping onto $Z_i$. 
Let $F$ be the fibre of $f$ over $\eta_T$.  
Since $F$ is connected and since $F\subset \bigcup _1^r Y_i$, we can find intersecting irreducible components $X_i,X_j$ (in $Y_i,Y_j$, respectively) and a codimension one component of the singular locus of $X$ inside $X_i\cap X_j$. But then $T= f(X_i\cap X_j)$ because $f(X_i)\neq f(X_j)$, so some codimension one component of the singular locus of $X$ maps onto $T$.   

Now we can assume that $Z$ is irreducible, that every component of $X$ maps onto $Z$, and that the image of every component of the singular locus of $X$ contains $T$. Assume that there is no codimension one component of the singular locus of $X$ mapping onto $T$. Shrinking $Z$ around $\eta_T$, we can assume that $Z$ is affine and smooth outside $T$. Let $U=Z\setminus T$. 
Let $\pi\colon Z^\nu\to Z$ be the normalisation of $Z$. Pick any regular function $\alpha^\nu$ on $Z^\nu$. We claim that $\alpha^\nu$ descends to $Z$ as a regular function. 
Since $U$ is smooth, $\pi$ is an isomorphism over $U$, so $\alpha^\nu$ descends to $U$ as a regular function. 

Let $\mu\colon X^\nu\to X$ be the normalisation of $X$. Then $X^\nu\to Z$ factors as $X^\nu\to Z^\nu\to Z$. Let $\beta^\nu$ be the pullback of $\alpha^\nu$ to $X^\nu$. Then $\beta^\nu$ descends to $f^{-1}U$  as a regular function. 
On the other hand, let $R$ be the intersection of $f^{-1}T$ with the singular locus of $X$, with reduced structure. Since no codimension one component of the singular locus of $X$ maps onto $T$, $R$ is of codimension $\ge 2$ in $X$. Letting $V=X\setminus R$, we have 
$$
f^{-1}U=X\setminus f^{-1}T\subset V
$$ 
and the singular locus of $V$ coincides with the singular locus of $f^{-1}U$.

Let $W$ be the smooth locus of $V$. Then $V=W\cup f^{-1}U$. Since $W$ is smooth, $\beta^\nu$ descends to $W$, say as $\lambda$. Moreover, by construction, $\beta^\nu$ descends to $f^{-1}U$, say as $\gamma$. Since $W\cap f^{-1}U$ is inside the smooth locus of $V$, we can see that $X^\nu\to X$ is an isomorphism over $W\cap f^{-1}U$ and from this we get 
$$
\lambda|_{W\cap f^{-1}U}=\gamma|_{W\cap f^{-1}U}.
$$ 
So $\lambda,\gamma$ glue to a regular function on $V$. That is, $\beta^\nu$ descends to $V$. 
This implies that $\beta^\nu$ descends to $X$, say as $\beta$, because $X$ is $S_2$ and $R=X\setminus V$ has codimension $\ge 2$. On the other hand, since $f$ is a contraction, $\beta$ is the pullback of some regular function $\alpha$ on $Z$. Then $\alpha^\nu$ descends to $Z$ as $\alpha$. 

We have  then proved that every regular function on $Z^\nu$ descends to $Z$. This is possible only if 
 $Z^\nu\to Z$ is an isomorphism. Then $Z$ is normal, a contradiction. 

\end{proof}

\subsection{Extension of families over curves}

\begin{lem}\label{l-s.min.model-extension-over-curves}
Let $d\in\N$, $c\in \Q^{\ge 0}$, $\Gamma\subset \Q^{>0}$, and $\sigma\in\Q[t]$ be a polynomial. 
Assume that $\Phi$ is closed under addition. 
Assume that   
$$
f\colon (X,B),A\to S
$$ 
is a family of $(d,\Phi_c,\Gamma,\sigma)$-stable minimal models where $S$ is a smooth curve over $k$.  Then after a finite base change, $f$ can be extended to a family of $(d,\Phi_c,\Gamma,\sigma)$-stable minimal models
$$
f'\colon (X',B'),A'\to S'
$$  
where $S'$ is the smooth compactification of $S$.
\end{lem}
\begin{proof}
First assume that $X$ is normal. After a finite base change, there exist a compactification $(\overline{X},\overline{B}),\overline{A}$ of $(X,B),A$ over $S'$ and a semi-stable log resolution $W\to \overline{X}$, by \cite{KKMS73}. Let $B_W$ be the birational transform of $\overline{B}$ plus the horizontal$/S'$ exceptional divisors of $W\to \overline{X}$. Let $A_W$ be the birational transform of $\overline{A}$. 
Then $(X,B)$ is a weak lc model of $(W,B_W)$ over $S$. So we can run an MMP on $K_W+B_W$ over $S'$ with scaling of some ample divisor which terminates over $S$ with some model $(Y,B_Y)$, by the results of \cite{B12}. Then $(Y,B_Y)$ is a dlt model of $(X,B)$ over $S$. 
Moreover, since the generic point of every non-klt centre of $(Y,B_Y)$ maps into $S$, we can run a further MMP on $K_Y+B_Y$ which terminates over $S'$ with a good minimal model $(X'',B'')$, by \cite{HX13}. Note that $Y\bir X''$ is an isomorphism over $S$. 

Pick a small number $a>0$. Since $A$ contains no non-klt centre of $(X,B)$, $K_Y+B_Y+aA_Y$ is the pullback of $K_X+B+aA$ over $S$, hence $K_Y+B_Y+aA_Y$ is semi-ample over $S$. This implies that $K_{X''}+B''+aA''$ is also semi-ample over $S$ as $Y\bir X''$ is an isomorphism over $S$.  Moreover, we can assume that $W\bir Y$ is a partial MMP on $K_W+B_W+aA_W$ and that $Y\bir X''$ is a partial MMP on $K_Y+B_Y+aA_Y$, so $(X'',B''+aA'')$ is dlt.
Since the generic point of every non-klt centre of $(X'',B''+aA'')$ maps into $S$, $(X'',B''+aA'')$ has a good minimal model over $S'$, so we can run an MMP on $K_{X''}+B''+aA''$ over $S'$ so that it terminates with a good minimal model $(X',B'+aA')$ over $S'$. Since $K_{X''}+B''$ is semi-ample over $S'$, we can choose $a$ small enough so that the MMP on $K_{X''}+B''+aA''$ is $K_{X''}+B''$-trivial.

By construction, $(X',B'+tA')$ is lc and $K_{X'}+B'+tA'$ is semi-ample over $S'$ for every sufficiently small $t\ge 0$. Replace $(X',B'+aA')$ with its lc model. Then $(X',B'+tA')$ is lc and $K_{X'}+B'+tA'$ is semi-ample over $S'$ for every sufficiently small $t\ge 0$, and $K_{X'}+B'+tA'$  is ample over $S'$ for small $t>0$. Then $(X',B'),A'\to S'$ is the desired family.

When $X$ is not necessarily normal, we take the normalisation $(X^\nu,B^\nu),A^\nu$ of $(X,B),A$, extend each irreducible component $(X^\nu_j,B^\nu_j),A^\nu_j$ to a family $({X^\nu_j}',{B^\nu_j}'),{A^\nu_j}'$ over $S'$, and then glue them together using \cite{K13} to get a family 
$$
f'\colon (X',B'),A'\to S'
$$ 
where $(X',B'+tA')$ is slc and $K_{X'}+B'+tA'$ is semi-ample (resp. ample) over $S'$ for every small $t\ge 0$ (resp. $t>0$). These follow from arguments similar to those of the proof of Lemma \ref{l-l.stable-family-s-ample}.
For every $s'\in S'$, the log fibre $(X_{s'}',B_{s'}'),A_{s'}'$ is a stable minimal model. Thus $f'$ is a family of stable minimal models.

Now we show that the extended family $f'$ is a family of $(d,\Phi_c,\Gamma,\sigma)$-stable minimal models. 
Pick a closed point $s'\in S'$ and let $(X_{s'}',B_{s'}'),A_{s'}'$ be the log fibre of the family over $s'$.
Since $f'$ is flat, $\dim X_{s'}'=d$. Also since the coefficients of $B',A'$ belong to $\Phi_c$, 
the coefficients of $B_{s'}'$ (resp. $A_{s'}'$) are $\Z^{\ge 0}$-linear combination of the coefficients of $B'$ (resp. $A'$), hence they belong to $\Phi_c$. 
Moreover, for any small $t>0$, 
$$
\vol(K_{X_{s'}'}+B_{s'}'+tA_{s'}')=\vol(K_{X_s'}+B_s'+tA_s')=\sigma(t)
$$
where $s$ is a general point of $S'$. 

Let $X_{s'}'\to Z_{s'}'$ be the contraction defined by $K_{X_{s'}'}+B_{s'}'$. It remains 
to show that for any general fibre $F'$ of $X_{s'}'\to Z_{s'}'$ over any irreducible component $T'$ of $Z_{s'}'$, we have $\vol(A'_{s'}|_{F'})\in \Gamma$. 
Let $X'\to Z'/S'$ be the contraction defined by $K_{X'}+B'$. Then $Z_{s'}'$ is the fibre of $Z'\to S'$ over ${s'}$. Since $T'$ is an irreducible component of $Z_{s'}'$, $Z'$ is pure dimensional near the generic point $\eta_{T'}$: indeed, if an irreducible component $R'$ of $Z'$ contains $\eta_{T'}$, then each irreducible component of the fibre $R_{s'}'$ is of dimension $\dim R'-1$; since $T'$ is an irreducible component of $Z_{s'}'$ and since $T'\subset R_{s'}'$, $T'$ is one of the components of $R_{s'}'$; so $\dim T'=\dim R'-1$; this shows that all the components $R'$ containing $\eta_{T'}$ have the same dimension. 

On the other hand, since $(X',B')\to S'$ is a locally stable family, no fibre contains a non-klt centre of $(X',B')$. In particular, no codimension one component of the singular locus of $X$ maps onto $T'$. Applying Lemma \ref{l-image-slc-pairs}, we deduce that $Z'$ is smooth near $\eta_{T'}$. 
Therefore, $X'\to Z'$ is flat over $\eta_{T'}$ as none of the components of $X'$ map onto $T'$. Thus $\vol(A'|_{G'})$ is constant for fibres $G'$ of $X'\to Z'$ over some open neighbourhood $U$ of $\eta_{T'}$ (note that $X'$ is pure dimensional, so $G'$ are pure dimensional too). 

The fibre $U_s$ of $U\to S'$ is a non-empty open subset of the fibre $Z_s'$ of $Z'\to S'$, for general $s\in S$. Moreover, for such $s$, $X_s'=X_s\to Z_s'=Z_s$ coincides with the contraction defined by $K_{X_s'}+B_s'=K_{X_s}+B_s$. Therefore, 
$$
\vol(A_{s'}'|_{F'})=\vol(A_s|_{F})\in \Gamma,
$$ 
for the general fibres $F'$ of $X_{s'}'\to Z_{s'}'$ over $T'$ and general fibres $F$ of $X_s\to Z_s$ over $U_s$. 
This implies that $(X_{s'}',B_{s'}'),A_{s'}'$ is a $(d,\Phi_c,\Gamma,\sigma)$-stable minimal model 
as desired. 

\end{proof}

\subsection{Existence of moduli spaces}

\begin{proof}(of Theorem \ref{t-moduli-s.min.models})
We follow a standard argument that is frequently used to construct moduli spaces.\\

\emph{Step 1.}
Put
$
\Xi=(d,\Phi_c,\Gamma,\sigma).
$
Let $a,r$ be the numbers given by Lemma \ref{l-bnd-lct-v.ampleness}, for the data $d,\Phi_c,\Gamma,\sigma$.
Let 
$$
f\colon (X,B),A\to S
$$ 
be a family of $(d,\Phi_c,\Gamma,\sigma)$-stable minimal models where $S$ is a reduced $k$-scheme. 
Put $\Delta=B+aA$. For each $s\in S$, $(X_s,B_s),A_s$ is a $(d,\Phi_c,\Gamma,\sigma)$-stable minimal model, so by the lemma,  $(X_s,\Delta_s)$ is slc, 
$r(K_{X_s}+\Delta_s)$ is very ample, and  
$$
h^q(r(K_{X_s}+\Delta_s))=0
$$ 
for $q>0$.  
Moreover, 
there are finitely many possibilities for the number 
$$
n_s:=\mathcal{X}(r(K_{X_s}+\Delta_s))-1=h^0(r(K_{X_s}+\Delta_s))-1
$$ 
depending only on $d,c,\Gamma,\sigma,a,r$: as in the proof of \ref{l-bnd-lct-v.ampleness}, 
this finiteness can be reduced to the case $s=\Spec \C$ in which 
case we can apply the boundedness of Theorem \ref{t-bnd-s-mmodels-slc}. 
In addition, by \cite[4.34 and 5.29]{K21}, $r(K_{X/S}+\Delta)$ is Cartier. 
Thus by flatness of $X\to S$, we see that $n_s$ is locally constant on $S$. 

Now fix $n$ and let $\mathfrak{S}_{\rm slc}\Xi(n)$ be the restriction of the functor 
$\mathfrak{S}_{\rm slc}\Xi$ to families $f$ in which 
$$
n=h^0(r(K_{X_s}+\Delta_s))-1
$$ 
for every $s\in S$.  Put 
$$
\Xi'=(d,\Phi_c,\Gamma,\sigma,a,r,\PP^n).
$$\

\emph{Step 2.}
Let 
$$
f\colon (X,B),A\to S
$$ 
be a family of $\Xi(n)$-stable minimal models where $S$ is a reduced $k$-scheme. 
Put $\Delta=B+aA$. 
By base change of cohomology, 
$$
R^qf_*\mathcal{O}_X(r(K_{X/S}+\Delta))=0
$$ 
for $q>0$. Moreover, for each $s\in S$, the natural map 
$$
R^qf_*\mathcal{O}_X(r(K_{X/S}+\Delta))\otimes k(s) \to H^q(r(K_{X_s}+\Delta_s)) 
$$
is an isomorphism for $q\ge 0$. In particular, 
$$
f_*\mathcal{O}_X(r(K_{X/S}+\Delta))
$$
is locally free and 
$$
f^*f_*\mathcal{O}_X(r(K_{X/S}+\Delta)) \to \mathcal{O}_X(r(K_{X/S}+\Delta))
$$
is surjective, so we get a closed embedding 
$$
g\colon X\to \PP(f_*\mathcal{O}_X(r(K_{X/S}+\Delta)))
$$ 
over $S$. Therefore, for each point $s\in S$, we can shrink $S$ around $s$ so that we get a closed embedding into $\PP^n_S$ in which case 
$$
f\colon (X\subset \PP^n_S,B),A\to S
$$ 
is a strongly embedded family of $\Xi'$-stable minimal models: indeed,
\begin{itemize}
\item $(X,B+aA)\to S$ is locally stable by our choice of $a$,

\item $f=\pi g$  where $\pi$ denotes the projection $\PP^n_S\to S$, 

\item we have 
$$
\mathcal{L}:=g^*\mathcal{O}_{\PP^n_S}(1)\simeq \mathcal{O}_X(r(K_{X/S}+\Delta))
$$ 
and  $R^qf_*\mathcal{L}\simeq R^q\pi_*\mathcal{O}_{\PP^n_S}(1)$ for all $q$, and

\item for every $s\in S$,  
$$
\mathcal{L}_s\simeq \mathcal{O}_{X_s}(r(K_{X_s}+B_s+aA_s)).
$$ 
\end{itemize}
However, note that the embedding $X\to \PP^n_S$ is not unique: it is determined 
up to an automorphism of $\PP^n_S$ over $S$.\\  

\emph{Step 3.}
Now recall from Proposition \ref{p-moduli-se-pol-cy} the functor $\mathfrak{S^{\rm e}_{\rm slc}}\Xi'$ and its fine moduli space ${M^{\rm e}_{\rm slc}}\Xi'$ which is reduced, separated and of finite type over $k$. 
There is a natural action of ${\rm PGL}_{n+1}(k)$ on ${M^{\rm e}_{\rm slc}}\Xi'$ inherited from the action of ${\rm PGL}_{n+1}(k)$ on $\PP^n_k$ and the action on associated Hilbert schemes. Indeed, each closed point $\lambda$ of ${M^{\rm e}_{\rm slc}}\Xi'$ corresponds to a strongly embedded $\Xi'$-stable minimal model $(X,B),A$ over $k$; any automorphism of $\PP^n_k$ maps $(X,B),A$ to a $\Xi'$-stable minimal model $(X',B'),A'$ over $k$, hence mapping $\lambda$ to a closed point of ${M^{\rm e}_{\rm slc}}\Xi'$; note in particular, that the KSBA-stable pair $(X,B+aA)$ is mapped to the KSBA-stable pair $(X',B'+aA')$. 

 The stabilisers of the action of ${\rm PGL}_{n+1}(k)$ on ${M^{\rm e}_{\rm slc}}\Xi'$ are finite because  automorphism groups of the pairs $(X,B+aA)$ are finite. Moreover, this action is proper because stable families over a smooth curve can be extended over the compactification of the curve in at most one way.   
The quotient 
$$
{M_{\rm slc}}\Xi(n):={M^{\rm e}_{\rm slc}}\Xi'/{\rm PGL}_{n+1}(k)
$$ 
then exists as an algebraic space \cite{KM97}\cite{K97} which is a coarse moduli 
space for the functor  $\mathfrak{S}_{\rm slc}\Xi(n)$. 

Now note that by Steps 1 and 2, given any family 
$$
f\colon (X,B),A\to S
$$ 
for the functor $\mathfrak{S}_{\rm slc}\Xi(n)$, there is an open covering 
$S=\bigcup S_i$ and morphisms $S_i\to {M^{\rm e}_{\rm slc}}\Xi'$ such that the induced morphisms 
$S_i\to {S_{\rm slc}}\Xi(n)$ are uniquely determined and they agree on overlaps, 
hence they determine a unique morphism $S\to {M_{\rm slc}}\Xi(n)$; moreover, the map 
$$
{\mathfrak{S}_{\rm slc}}\Xi(n)(\Spec k) \to {M_{\rm slc}}\Xi(n)(k)
$$ 
is bijective. 

The moduli space ${M_{\rm slc}}\Xi(n)$ is proper because every family of $\Xi(n)$-stable minimal models 
 over a smooth curve can be extended to a family 
over the compactification of the curve, after a finite base change, by Lemma \ref{l-s.min.model-extension-over-curves}. 
The moduli space is projective by \cite{K90}\cite{F18}\cite{KP17}.
The moduli space ${M_{\rm slc}}\Xi$ for the functor $\mathfrak{S}_{\rm slc}\Xi$ 
is the disjoint union of ${M_{\rm slc}}\Xi(n)$ 
for the finitely many possible $n$.\\

\emph{Step 4}. 
Since being normal is an open condition in a falt family, there is an open subset ${M_{\rm lc}}\Xi\subset {M_{\rm slc}}\Xi$ which is a quasi-projective coarse moduli space for the functor $\mathfrak{S}_{\rm lc}\Xi$.  

\end{proof}

%%%%%%%%%%%%%%%%%%%%%%%
%%%%%%%%%%%%%%%%%%%%%%%

\section{\bf Examples of moduli spaces and some questions}\label{s-examples}

In this section, we discuss various examples of classes of varieties and pairs and their moduli spaces illustrating the general existence theorem \ref{t-moduli-s.min.models}. We also add remarks related to stable minimal models and list a few questions.

\subsection{Moduli of Fano pairs}
 A particular case of stable Calabi-Yau pairs is that of Fano pairs 
polarised by effective anti-pluri-log-canonical divisors. 
A \emph{stable Fano pair} is of the form $(X,\Lambda),A$ where $(X,\Lambda+A),A$ is a stable Calabi-Yau pair and $\Lambda\ge 0$. Since $-(K_X+\Lambda)\sim_\Q A$, the pair $(X,\Lambda)$ is indeed Fano (with slc singularities) which is polarised by $A$. 

As an example, consider $X=\PP^n$, $\Lambda=0$, and $K_X+A\sim_\Q 0$ where $A=cN$ for some hypersurface $N$ and $(X,A)$ is klt. Then $(X,A),A$ is a stable Calabi-Yau pair, so $(X,0),A$ is a stable Fano pair. Note that fixing $c$, $\deg N=\frac{n+1}{c}$ is fixed. So there is a one-to-one correspondence between the $(X,0),A$ and the hypersurfaces $N=\frac{1}{c}A$. Hacking \cite{Ha04} considered this setting when $n=2$ and $N$ is smooth and used a compactification of the moduli space of the $(X,0),A$ to get a compactifcation of the moduli space of the plane curves $N$.  

Let $d\in\N$, $c\in \Q^{> 0}$, $\Gamma\subset \Q^{>0}$ a finite set, and $\sigma\in\Q[t]$ be a polynomial.  
A \emph{$(d,\Phi_c,\Gamma,\sigma)$-stable Fano pair} is a stable Fano pair $(X,\Lambda),A$ where $(X,\Lambda+A),A$ is a $(d,\Phi_c,\Gamma,\sigma)$-stable Calabi-Yau pair.  

We define stable Fano families.
Let $S$ be a reduced scheme over $k$. A \emph{family of $(d,\Phi_c,\Gamma,\sigma)$-stable Fano pairs over $S$} 
is of the form $(X,\Lambda),A\to S$ where $(X,\Lambda+A),A\to S$ is a family of $(d,\Phi_c,\Gamma,\sigma)$-stable Calabi-Yau pairs and $\Lambda\ge 0$. 

We define the moduli functor $\mathfrak{F}_{\rm slc}(d,\Phi_c,\Gamma,\sigma)$ from the category of reduced $k$-schemes to the category of sets by setting  
$$
{\mathfrak{F}}_{\rm slc}(d,\Phi_c,\Gamma,\sigma)(S)=\{\mbox{families of $(d,\Phi_c,\Gamma,\sigma)$-stable Fano pairs over $S$}, 
$$
$$
\hspace{8cm} \mbox{up to isomorphism over $S$}\}. 
$$ 
 
\begin{thm}\label{t-moduli-pol-Fano}
The functor $\mathfrak{F}_{\rm slc}(d,\Phi_c,\Gamma,\sigma)$ has a projective coarse moduli space.
\end{thm}  

The theorem can be proved using the same ideas of the proof of Theorem \ref{t-moduli-s.min.models} and the following lemma. 

\begin{lem}\label{l-l.stable-Fano}
Assume that $(X,B),A\to S$ is a family of stable Calabi-Yau pairs where $S$ is a reduced $k$-scheme of finite type. Let $S'$ be the set of (not necessarily closed) points $s\in S$ such that $B_s\ge A_s$.
Then $S'$ is a locally closed subset of $S$ (with reduced structure).
\end{lem}
\begin{proof}
Arguing as in the proof of Lemma \ref{l-l.stable-family-lc}, it is enough to show that $S'$ contains a non-empty open subset when $S$ is irreducible and $S'$ is dense in $S$. But in this case, $B\ge A$, so $S'=S$.   

\end{proof}

Applying the lemma to the universal family of stongly embedded stable minimal models (which are then Calabi-Yau pairs by our choice of $\sigma$) over $M^e_{\rm slc}(d,\Phi_c,\Gamma,\sigma)$ of Section \ref{s-strongly-embedded-s.m.model} we get a locally closed subset $F_{\rm slc}^e(d,\Phi_c,\Gamma,\sigma)$ parametrising srongly embedded stable Fano pairs. Taking quotient by the action of ${\rm PGL}_{n+1}(k)$ as in the proof of \ref{t-moduli-s.min.models} we can construct a projective moduli space for $\mathfrak{F}_{\rm slc}(d,\Phi_c,\Gamma,\sigma)$. To get projectivity one shows first that families of stable Fano pairs can be extended over smooth curves, after a finite base change, using the same arguments of Lemma \ref{l-s.min.model-extension-over-curves}.  

Restricting the functor $\mathfrak{F}_{\rm slc}(d,\Phi_c,\Gamma,\sigma)$ to special pairs, e.g. those where $X$ deform to $\PP^d$, gives more interesting moduli spaces.

\subsection{Moduli of log Fano fibrations}

A \emph{stable log Fano fibration} is of the form $(X,\Lambda),A\to Z$ where $(X,\Lambda+A),A$ is a stable minimal model and $\Lambda\ge 0$, and $X\to Z$ is the contraction defined by $K_X+\Lambda+A$. Then $-(K_X+\Lambda)\sim_\Q A/Z$, so we can see that $(X,\Lambda)\to Z$ is indeed a log Fano fibration (with slc singularities) which is polarised by $A$. 

For an example, consider a Fano fibration $X\to Z$ where $X$ is projective with klt singularities, $-K_X$ is ample over $Z$, and $X\to Z$ is a contraction. Then we can always find $A$ so that $(X,0),A\to Z$ is a stable Fano fibration. To be more specific, for example, assume $\mathcal{E}$ is a locally free sheaf of rank $n+1$ on $Z$ and $f\colon X=\Proj(\mathcal{E})\to Z$ is the associated projective bundle. Then taking a sufficiently ample divisor $L$ on $Z$, we can find $0\le G\sim lH+\pi^*L$ where $H$ is the tautological divisor on $X$ and $n+1<l\in \N$ so that letting $A=\frac{n+1}{l}G$ we get $(X,0),A\to Z$ which is a stable Fano fibration.

Let $d\in\N$, $c\in \Q^{> 0}$, $\Gamma\subset \Q^{>0}$ a finite set, and $\sigma\in\Q[t]$ be a polynomial.  
A \emph{$(d,\Phi_c,\Gamma,\sigma)$-stable log Fano fibration} is a stable log Fano fibration $(X,\Lambda),A\to Z$ where $(X,\Lambda+A),A$ is a $(d,\Phi_c,\Gamma,\sigma)$-stable minimal model. 

Let $S$ be a reduced scheme over $k$. A \emph{family of $(d,\Phi_c,\Gamma,\sigma)$-stable log Fano fibrations over $S$} is of the form $(X,\Lambda),A\to S$ where $(X,\Lambda+A),A\to S$ is a family of $(d,\Phi_c,\Gamma,\sigma)$-stable minimal models and $\Lambda\ge 0$. For each $s\in S$, we get a stable log Fano fibration $(X_s,\Lambda_s),A_s\to Z_s$ where $X_s\to Z_s$ is the contraction defined by $K_{X_s}+\Lambda_s+A_s$.

We define the moduli functor $\mathfrak{F}^{\rm fib}_{\rm slc}(d,\Phi_c,\Gamma,\sigma)$ from the category of reduced $k$-schemes to the category of sets by setting  
$$
{\mathfrak{F}}^{\rm fib}_{\rm slc}(d,\Phi_c,\Gamma,\sigma)(S)=\{\mbox{families of $(d,\Phi_c,\Gamma,\sigma)$-stable log Fano fibrations over $S$}, 
$$
$$
\hspace{8cm} \mbox{up to isomorphism over $S$}\}. 
$$ 
 
\begin{thm}\label{t-moduli-pol-Fano-fib}
The functor $\mathfrak{F}^{\rm fib}_{\rm slc}(d,\Phi_c,\Gamma,\sigma)$ has a projective coarse moduli space.
\end{thm}  

This can be proved similar to Theorem \ref{t-moduli-pol-Fano}. 

\subsection{Moduli of abelian varieties with theta divisor}
Consider the data $d=1$, $c=1$, $\Gamma=\{1\}$, and $\sigma(t)=t$. 
Let $X$ be an elliptic curve and $A$ a closed point on $X$. Then $(X,0),A$ is a $(d,\Phi_c,\Gamma,\sigma)$-stable Calabi-Yau pair. The set of points in the moduli space $M_{\rm slc}(d,\Phi,\Gamma,\sigma)$ corresponding to such pairs is an open subset $U$ which is isomorphic to the moduli space $M_{1,1}$ of curves of genus $1$ with one marked point: this can be seen at the strongly embedded level discussed in Section \ref{s-strongly-embedded-s.m.model}. The closure $\overline{U}$ in $M_{\rm slc}(d,\Phi,\Gamma,\sigma)$ gives the compactification $\overline{M}_{1,1}$ which is the moduli space of stable curves of genus $1$ with a marked point. However, note that $\overline{U}\neq M_{\rm slc}(d,\Phi,\Gamma,\sigma)$ 
because the latter also parametrises $(\PP^1,B_1+B_2),S$ where $B_1,B_2,S$ are three distinct points.

More generally, consider the data $d$, $c=1$, $\Gamma=\{u\}$ where $u\in \N$, and $\sigma(t)=ut^d$. 
Let $X$ be an abelian variety of dimension $d$ and $A\ge 0$ an ample Cartier divisor of volume $u$. Then $(X,0),A$ is a $(d,\Phi_c,\Gamma,\sigma)$-stable Calabi-Yau pair. The set of points in the moduli space $M_{\rm slc}(d,\Phi_c,\Gamma,\sigma)$ corresponding to such pairs is an open subset $U$ and its closure $\overline{U}$ gives a meaningful compactification. 

Now additionally assume $u=d!$. Then by Kodaira vanishing and the Riemann-Roch theorem, 
$$
h^0(A)=\mathcal{X}(A)=\frac{\vol(A)}{d!}=1.
$$ 
So $A$ is a Theta divisor and it determines a principal polarisation, and the invertible sheaf $\mathcal{O}_X(A)$ uniquely determines $A$. So $U$ above is the moduli space of the ``pairs" $(X,\mathcal{O}_X(A))$ and its closure $\overline{U}$ is a meaningful compactification. See \cite{Al02} for more general relevant results.

\subsection{Moduli of K3 surfaces}
Consider the data $d=2$, $c=1$, $\Gamma=\{u\}$, and $\sigma(t)=ut^2$. 
Let $X$ be a K3 surface and $A\ge 0$ an ample Cartier divisor on $X$ with $\vol(A)=u$. Then $(X,0),A$ is a $(d,\Phi_c,\Gamma,\sigma)$-stable Calabi-Yau pair. The set of points in the moduli space $M_{\rm slc}(d,\Phi,\Gamma_c,\sigma)$ corresponding to such pairs is an open subset $U$ because being a K3 surface is an open condition in a family of stable Calabi-Yau pairs. The closure $\overline{U}$ gives a meaningful compactification of the moduli space of the $(X,0),A$. For related results see \cite{La16}.

Sometimes one can choose a more special $A$ in its linear system $|A|$ and consider the moduli space of the $(X,0),A$. For more on this see \cite{AET19}.

\subsection{Moduli of stable surfaces}
Let $(X,B),A\to Z$ be a stable lc minimal model. We have the following possibilities: 
\begin{itemize}
\item $\kappa(K_X+B)=2$: $X\to Z$ is birational,
\item $\kappa(K_X+B)=1$: $X\to Z$ is an elliptic fibration or a conic bundle,
\item $\kappa(K_X+B)=0$: $X\to Z$ is constant.
\end{itemize}
We can write the associated polynomial $\sigma$ explicitly as 
$$
\sigma(t)=(K_X+B+tA)^2=(K_X+B)^2+2(K_X+B)\cdot At+A^2t^2.
$$

In case $\kappa(K_X+B)=2$ and $A=0$, then $(X,B)$ is a KSBA stable surface and their moduli is constructed in \cite{Al94} and \cite{K21}. But if $A\neq 0$, then we get new moduli functors not covered by the mentioned works. 

In case $\kappa(K_X+B)=1$, $A$ is ample over $Z$, so the horizontal part of $\Supp A$ is a union of multi-sections of $X\to Z$. We have 
$$
\sigma(t)=2(K_X+B)\cdot At+A^2t^2.
$$
A special case is when $A$ is a section, in which case moduli of similar pairs were studied in \cite{ABI17}. 

In case $\kappa(K_X+B)=0$, $(X,B),A$ is a stable Calabi-Yau. 
We have $\sigma(t)=A^2t^2$.
Special cases of this, abelian surfaces and K3 surfaces, were already mentioned above. 

Although moduli of surfaces has been extensively studied in the literature but still it seems that our results tackle new moduli problems for surfaces.
 
\subsection{Moduli of stable 3-folds}
Let $(X,B),A\to Z$ be a stable lc minimal model where $d=2$. We have the following possibilities: 
\begin{itemize}
\item $\kappa(K_X+B)=3$: $X\to Z$ is birational,
\item $\kappa(K_X+B)=2$: $X\to Z$ is an elliptic fibration or a conic bundle,
\item $\kappa(K_X+B)=1$: $X\to Z$ is a K3 or abelian or del Pezzo etc fibration,
\item $\kappa(K_X+B)=0$: $X\to Z$ is constant.
\end{itemize} 

In case $\kappa(K_X+B)=2$ and $A=0$, then $(X,B)$ is a KSBA stable 3-fold and their moduli is constructed in \cite{K21}. But again as in the case of surfaces, if $A\neq 0$, then we get new moduli functors. 
In case $\kappa(K_X+B)=0$, special situations, e.g. abelian 3-folds, have been studied in the literature but it is likely that many more cases remain to be treated.   

But it seems that the cases $\kappa(K_X+B)=2$ and $\kappa(K_X+B)=1$ are largely unexplored 
in the literature. Our results open the door for the study of moduli spaces of a vast range of special cases in dimension $3$ and similarly in higher dimension. 

\subsection{Relative stable minimal models} 
The notion of stable minimal models also makes sense in the relative case. More precisely, fix a variety or scheme $Y$. A \emph{stable minimal model} $(X,B),A$ over $Y$ consists of $(X,B)$, an $\R$-divisor $A\ge 0$, and a projective moprhism $X\to Y$ such that   
\begin{itemize}
\item $(X,B)$ is a connected slc pair,
\item $K_X+B$ is semi-ample over $Y$ defining a contraction $f\colon X\to Z/Y$, 
\item $K_X+B+tA$ is ample over $Y$ for some $t>0$, and 
\item $(X,B+tA)$ is slc for some $t>0$.
\end{itemize} 
Note that this is not the same as a family of stable minimal models defined in \ref{defn-stable-minimal-family}. For example, we do not assume $X\to Y$ to be flat. But the two notions are equivalent in case $Y$ is smooth and irreducible.

\subsection{Generalisations}
One can consider generalisations of Theorem \ref{t-moduli-s.min.models} in several different directions.
\begin{itemize}
\item \emph{Non-reduced bases}: With more work one should be able to extend \ref{t-moduli-s.min.models} to the case when the moduli functor is defined for arbitrary base schemes, not just reduced ones. 

\item \emph{Real coefficients:} Although for simplicity we only considered stable minimal models $(X,B),A$ where the coefficients of $B,A$ are rational but \ref{t-moduli-s.min.models} should hold for real coefficients. 

\item \emph{Marked divisors:} It is possible to prove a version of \ref{t-moduli-s.min.models} for stable minimal models $(X,B),A$ where both $B,A$ are written as a linear combination of marked divisors. 

\item \emph{Relative case:} One can consider stable minimal models $(X,B),A$ relative to a fixed variety or scheme $T$. That is, we have a morphism $X\to T$ (not necessarily surjective) and $(X,B),A$ is stable over $T$. It should be possible to extend \ref{t-moduli-s.min.models} to this situation. 
\end{itemize}

\subsection{Questions}
We list a few questions related to stable minimal models and their moduli. 

\begin{quest}
For which set of data $d,\Phi,\Gamma,\sigma$, the set $\mathcal{S}_{\rm slc}(d,\Phi,\Gamma,\sigma)$ is non-empty?
\end{quest}

It is easy to get examples where the set is empty. For example, take $d=1$, $\Phi=\emptyset$, $\Gamma=\{1\}$. If $\deg \sigma>1$, then clearly $\mathcal{S}_{\rm slc}(d,\Phi,\Gamma,\sigma)=\emptyset$. Or if we take $\sigma=3+t$, then the set is again empty because if $X$ is a stable minimal model for the data, then $\sigma(t)=\deg K_X+t(\deg A)$, so the constant term is always even (this is clear if $X$ is normal; in the non-normal case the contribution of the nodes to $\deg K_X$ is again even). 

\begin{conj}\label{conj-no-vertical-volume}
Let $d\in\N$, $\Phi\subset \Q^{\ge 0}$ be a DCC set, and $\sigma\in\Q[t]$ be a polynomial. Then $\mathcal{S}_{\rm slc}(d,\Phi,\sigma)$ is a bounded family.
\end{conj}

Here, $\mathcal{S}_{\rm slc}(d,\Phi,\sigma)$ is defined similar to $\mathcal{S}_{\rm slc}(d,\Phi,\Gamma,\sigma)$ without the condition on $\vol(A|_F)$. 
The conjecture is related to the effective adjunction conjecture of Prokhorov-Shokurov \cite[Conjecture 7.13]{PSh09} on moduli divisors.

\begin{quest}
What kind of singularities does the moduli space $M_{\rm slc}(d,\Phi_c,\Gamma,\sigma)$ have?
\end{quest}

\begin{quest}
What is the Kodaira dimension of the components of $M_{\rm slc}(d,\Phi_c,\Gamma,\sigma)$? 
\end{quest}

One can ask other related questions, for example, when the components are uniruled or rationally connected or rational.

\begin{quest}
How the moduli spaces $M_{\rm slc}(d,\Phi_c,\Gamma,\sigma)$ vary when the data $\Phi_c,\Gamma,\sigma$ vary? 
\end{quest}

%%%%%%%%%%%%%%%%%%%%%%%%%%%%%%
%%%%%%%%%%%%%%%%%%%%%%%%%%%%%%

\section{\bf DCC of Iitaka volumes}

In this section, we prove some results about the DCC property of Iitaka volumes. We will not use them to prove the main results of this paper but they are of independent interest and most likely find applications elsewhere.

\subsection{Toroidal pairs}\label{ss-toroidal-pairs}
 
A pair $(Y,\Delta)$ is \emph{toroidal} if for each closed point $y\in Y$ there is a toric pair $(Y',\Delta')$ and a closed point $y'\in Y'$ such that $(Y,\Delta)$ and $(Y',\Delta')$ are formally isomorphic near $y,y'$. In particular, $\Delta$ is reduced, 
$(Y,\Delta)$ is lc, $Y$ is smooth outside $\Delta$, and 
$K_Y+\Delta$ is Cartier. For the theory of toroidal pairs, see \cite{KKMS73} where they are called toroidal embeddings.
A toroidal pair $(Y,\Delta=\sum_1^n S_i)$ is \emph{strict} if the $S_i$ are normal. 
In this case, for each subset $I\subseteq \{1,\dots,n\}$, $\bigcap_{i\in I} S_i$ is normal and $\bigcap_{i\in I} S_i\setminus \bigcup_{i\notin I} S_i $ is smooth  \cite[page 57]{KKMS73}. 
The irreducible components of such $\bigcap_{i\in I} S_i$ are the strata of $(Y,\Delta)$, hence are normal.

\begin{lem}\label{l-toroid-div-crep-bir}
Let $(Y,\Delta)$ be a log smooth toroidal pair. Let $S$ be a toroidal divisor over $Y$, i.e. $a(S,Y,\Delta)=0$, and let $V$ be its centre on $Y$. 
Let $(S,\Delta_S)$ and $(V,\Delta_V)$ be the induced toroidal structures on $S$ and $V$. Then there is a crepant birational map 
$$
(S,\Delta_S)\bir (\PP^n,\Delta_{\PP^n})\times (V,\Delta_V)
$$ 
over $V$ where $n+1$ is the codimension of $V$ in $Y$ and $\Delta_{\PP^n}$ is the toric boundary divisor on $\PP^n$. 
\end{lem}
\begin{proof}
Note that the toroidal structure on $\PP^n\times V$ is given by the sum of pullbacks of $\Delta_{\PP^n}$ and $\Delta_V$ via the corresponding projections.

It is enough to treat the case when every component of $\Delta$ that intersects $V$, contains $V$. In this case $\Delta_V=0$. Also, we can assume that $V$ is not a divisor otherwise the statement would be trivial.

First assume that $S$ is the exceptional divisor of the blowup of $Y$ along $V$. Then there is a birational map 
$$
(S,\Delta_S)\bir (\PP^n,\Delta_{\PP^n})\times V=:(T,\Delta_T)
$$ 
which is an isomorphism over the generic point $\eta_V$. Since every component of $\Delta$ contains $V$, every non-klt centre on both sides is horizontal over $V$. Thus taking a common resolution $\phi \colon U\to S$ and $\psi\colon U\to T$, we see that 
$$
K_S+\Delta_S\ge \phi_*\psi^*(K_T+\Delta_T) ~~\mbox{and}~~\psi_*\phi^*(K_S+\Delta_S)\le K_T+\Delta_T
$$
both hold: indeed, writing $\phi_*\psi^*(K_T+\Delta_T)=K_T+\Theta_T$, every component of $\Theta_T$ with positive coefficient is a compoenent of $\rddown{\Theta_T}$, hence is a non-klt place of $(T,\Delta_T)$, so is horizontal over $V$ which means it is a component of $\Delta_S$; a similar remark applies on $T$. Applying the negativity lemma implies $\phi^*(K_S+\Delta_S)=\psi^*(K_T+\Delta_T)$, hence the above map is crepant birational.

Now we treat the general case. Let $E$ be the exceptional divisor of the blowup of $Y$ along $V$ and let $W$ be the centre of $S$ on the blowup, which is contained in $E$. If $\dim W=\dim V$, then $W\to V$ is an isomorphism, hence we may replace $Y$ with the blowup and repeat the argument. So assume $\dim W>\dim V$. 
Let $(W,\Delta_W)$ be the induced toroidal pair. By reverse induction on codimension, we can assume that we already have a crepant birational map 
$$
(S,\Delta_S)\bir (\PP^m,\Delta_{\PP^m})\times (W,\Delta_W) 
$$ 
for some $m$. On the other hand, by the previous paragraph, we have a  crepant birational map 
$$
(E,\Delta_E)\bir (\PP^n,\Delta_{\PP^n})\times V.
$$ 
Moreover, $W$ is a horizontal$/V$ non-klt centre of $(E,\Delta_E)$, so we get a crepant birational map 
$$
(W,\Delta_W)\bir (R,\Delta_R) 
$$ 
where $R$ is some non-klt centre of $(\PP^n,\Delta_{\PP^n})\times V$. But such centres are again products, that is, we get a crepant birational map 
$$
(W,\Delta_W)\bir (R,\Delta_R)\simeq (\PP^l,\Delta_{\PP^l})\times V 
$$ 
for some $l$. Combining the above we get a crepant birational map 
$$
(S,\Delta_S)\bir (\PP^m,\Delta_{\PP^m})\times (\PP^l,\Delta_{\PP^l})\times V
$$ 
which in turn gives a crepant birational map 
$$
(S,\Delta_S)\bir (\PP^{m+l},\Delta_{\PP^{m+l}})\times V.
$$ 
Note that $n=m+l$.
\end{proof}

\subsection{DCC over fixed pair}

\begin{lem}\label{l-dcc-ivol-fixed-bir-model}
Let $\Phi\subset \Q^{\ge 0}$ be a DCC set.
Assume that $(Y,\Delta)$ is a projective strictly toroidal pair. Consider projective pairs $(X,B)$ such that 
\begin{itemize}
\item $(X,B)$ is lc and has a good minimal model, 
\item the coefficients of $B$ are in $\Phi$, and 
\item we have a birational contraction $\phi\colon X\to Y$ with $\phi_*B\le \Delta$.  
\end{itemize}
Then the set of $\Ivol(K_X+B)$ for such $(X,B)$ satisfies the DCC.
\end{lem}
\begin{proof}
Assume not. Then we can find a sequence $\phi_i\colon (X_i,B_i)\to Y$ of pairs as in the proposition so that the Iitaka volumes $v_i=\Ivol(K_{X_i}+B_i)$ form a strictly decreasing sequence. We can assume that $(X_i,B_i)$ is $\Q$-factorial dlt, and running an MMP on $K_{X_i}+B_i$ over $Y$ with scaling of some ample divisor we reach a model on which the pushdown of $K_{X_i}+B_i$ is a limit of movable/$Y$ $\Q$-divisors. Replacing $X_i$ with that model, we can assume that $K_{X_i}+B_i$ is a limit of movable/$Y$ $\Q$-divisors. In particular, by \cite[Lemma 3.3]{B12},
$$
K_{X_i}+B_i+G_i=\phi_i^*(K_Y+\Delta)
$$
where $G_i\ge 0$. Thus every component of $B_i$ is toroidal with respect to $(Y,\Delta)$.

Let $(X_i'',B_i'')$ be a good minimal model of $(X_i,B_i)$ and $X_i''\to Z_i''$ be the contraction defined by $K_{X_i''}+B_i''$. 
First assume that for each $i$, $(X_i'',B_i'')$ is not klt over $\eta_{Z_i''}$. In this case, we first change models and assume $(Y,\Delta)$ is log smooth. By Lemma \ref{l-stein-deg-lc-bnd-comp}, there is a crepant birational model $(X_i',B_i')$ of $(X_i'',B_i'')$ and a normal component $S_i'$ of $\rddown{B_i'}$ such that 
$$
v_i=\Ivol(K_{X_i''}+B_i'')=\Ivol(K_{X_i'}+B_i')=\frac{1}{\sdeg(S_i'/Z_i'')}\Ivol((K_{X_i'}+B_i')|_{S_i'})
$$
where $\sdeg(S_i'/Z_i'')$ is bounded from above depending only on $\dim X_i''=\dim Y$. By the previous paragraph, 
$$
0=a(S_i',X_i'',B_i'')\ge a(S_i',X_i,B_i)\ge a(S_i',Y,\Delta),
$$
hence $S_i'$ is ``toroidal" with respect to $(Y,\Delta)$. More precisely, taking a common resolution $\alpha_i\colon W_i\to X_i'$ and $\beta_i\colon W_i\to X_i$ and $\gamma_i\colon W_i\to Y$, and letting $S_i$ be the birational transform of $S_i'$ on $W_i$, then $S_i$ is toroidal with respect to $(Y,\Delta)$, and \
$$
\alpha_i^*(K_{X_i'}+B_i')\le \beta_i^*(K_{X_i}+B_i)\le \gamma_i^*(K_Y+\Delta),
$$ 
hence 
$$
{\bar{\alpha}_i}^*(K_{S_i'}+B_{S_i'}:=(K_{X_i'}+B_i')|_{S_i'})\le (K_{S_i}+\Delta_{S_i}):=\gamma_i^*(K_Y+\Delta)|_{S_i}
$$ 
where $\bar{\alpha}_i$ denotes $S_i\to S_i'$. On the other hand, by Lemma \ref{l-toroid-div-crep-bir}, 
$({S_i},\Delta_{S_i})$ is crepant birational to a toroidal pair $(T,\Delta_T)$ which we can assume to be independent of $i$. If $\mu_i\colon V_i\to S_i'$ and $\nu_i\colon V_i\to T$ is a common resolution, then 
$$
{\nu_i}_*\mu_i^*(K_{S_i'}+B_{S_i'})\le K_T+\Delta_T.
$$
Therefore, applying induction on dimension we deduce that $\{\Ivol(K_{S_i'}+B_{S_i'})\}$ satisfies DCC. Since $\sdeg(S_i'/Z_i'')$ is bounded from above, $\{\Ivol(K_{X_i}+B_{i})\}$ satisfies DCC, a contradiction. 

From now on we assume that $(X_i'',B_i'')$ is klt over $\eta_{Z_i''}$. In this case we can apply the proof of \cite[?]{HMX14}. We follow the presentation of the proof of \cite[?]{B21b} and give a sketch. Let ${\bf{M}}_i$ be the b-divisor so that the trace ${\bf{M}}_{i,X_i}=B_i$ but its coefficient in any exceptional$/X_i$ prime divisor is $1$. One reduces to the case such that for any prime divisor $D$ over $Y$, the coefficients $\mu_D{\bf{M}}_i$ form an increasing sequence for $i\gg 0$. We can then define a limit b-divisor ${\bf{C}}=\lim {\bf{M}}_i$. Let $C=\bf{C}_Y$. In the end one reduces to the situation in which if we write $K_{X_i}+\Theta_i=\phi_i^*(K_Y+C)$, then for each prime divisor $D$ on $X_i$, we have $\mu_D{\bf{C}}\ge \mu_{D}\Theta_i$.   

Assume that $S$ is a component of $C$ such that $\mu_S{\phi_i}_*B_i=1$ for every $i$. Since $S_i$ is not horizontal over $Z_i''$, decreasing the coefficient of $S$ in $B_i$ slightly can change the Iitaka volume $\Ivol(K_{X_i}+B_i)$ slightly, hence we can assume that $\rddown{{\phi_i}_*B_i}=0$ for every $i$. 

Pick a rational number $t\in (0,1)$ so that ${\phi_i}_*B_j\le tC$ for some $j$. Then $K_{X_i}+B_i\ge \phi_i^*(K_Y+tC)$ for $i\gg 0$ and $K_{X_j}+B_j\le \phi_j^*(K_Y+tC)$. Therefore, for $i\gg j$, 
$$
v_i=\Ivol(K_{X_i}+B_i)\ge v_j=\Ivol(K_{X_j}+B_j).
$$ 
This contradicts the assumption that the Iitaka volumes $\Ivol(K_{X_i}+B_i)$ form a strictly decreasing sequence. 
 
\end{proof}

\subsection{DCC in general}

\begin{lem}\label{l-Ivol-constant-in-log-smooth-families}
Let $(X,B)\to Z$ be a relatively log smooth family where $Z$ is smooth and $B$ is a $\Q$-divisor. 
Assume that a log fibre $(X_0,B_0)$ has a good minimal model. Then $\Ivol(K_{X_z}+B_z)$ is constant for the closed points $z\in Z$.  
\end{lem}
\begin{proof}
By \cite[Corollary 1.4]{HMX18}, $(X,B)$ has a good minimal model over $Z$ hence it has lc model $V$ over $Z$, and $V_z$ is the lc model of $({X_z},B_z)$ for every closed point $z$. 
Moreover, if $H$ is the ample$/Z$ divisor on $V$ determined by $(X,B)$, then $H|_{V_z}$ is the ample divisor on $V_z$ determined by $(X_z,B_z)$. Therefore, $\Ivol(K_{X_z}+B_z)=\vol(H|_{V_z})$ is independent of $z$.

\end{proof}

\begin{prop}\label{p-dcc-ivol-bnd-bir-model}
Let $\Phi\subset \Q^{\ge 0}$ be a DCC set and $\mathcal{P}$ be a bounded set of couples.
Consider projective pairs $(X,B)$ such that 
\begin{itemize}
\item $(X,B)$ is lc and has a good minimal model,
\item the coefficients of $B$ are in $\Phi$, and 
\item there exist $(Y,\Delta)\in \mathcal{P}$ and a birational map $\psi\colon Y\bir X$ such that $\psi^*(K_X+B)\le K_Y+\Delta$.  
\end{itemize}
Then the set of $\Ivol(K_X+B)$ for such $(X,B)$ satisfies the DCC.
\end{prop}
\begin{proof}
Replacing each $(Y,\Delta)\in \mathcal{P}$ with a bounded log resolution, we can assume that the $(Y,\Delta)$ are log smooth. Let $\phi\colon W\to Y$ and $\beta\colon W\to X$ be a common resolution. Let $E$ be the sum of the exceptional$/X$ prime divisors $D$ on $W$ such that $a(D,X,B)<1$. Let $B_W$ be the sum of the birational transform of $B$ and $E$. Then we can write 
$$
K_W+B_W=\beta^*(K_X+B)+G
$$
where $G\ge 0$. In particular, the good minimal model of $(X,B)$ is a good weak lc model of $(W,B_W)$  from which we can produce a good minimal model of $(W,B_W)$, by the results of \cite{B12}. Moreover,  
$$
\Ivol(K_W+B_W)=\Ivol(K_X+B).
$$ 
Since $\psi^*(K_X+B)\le K_Y+\Delta$ and since $\Delta$ is reduced, $\phi_*B_W\le \Delta$. Replacing $(X,B)$ with $(W,B_W)$, we can assume that $\phi\colon X\bir Y$ is a morphism with $\phi_*B\le \Delta$. 

The rest of the proof is similar to the proof of \cite[Theorem 1.3(1)]{HMX14} (see also the presentation in the proof of \cite[Theorem 1.3]{B21b}). We then only give a sketch.

Running an MMP on $K_X+B$ over $Y$, we can assume that each component of $B$ is toroidal with respect to $(Y,\Delta)$. Moreover, we can replace $(X,B)$ hence assume that $X$ is obtained from $Y$ by a sequence of smooth blowups of $Y$, toroidal with respect to $(Y,\Delta)$.   

Now assume that there is a sequence $(X_i,B_i)\to (Y_i,\Delta_i)$ of pairs and maps as in the proposition, satisfying the properties of the last paragraph, so that the $v_i=\Ivol(K_{X_i}+B_i)$ form a strictly decreasing sequence of numbers. Since the $(Y_i,\Delta_i)$ are bounded, we can assume that there is a relatively log smooth family $(\overline{Y},\overline{\Delta})\to T$ to a smooth variety such that each $(Y_i,\Delta_i)$ is the log fibre of $(\overline{Y},\overline{\Delta})\to T$ over a closed point $t_i$. After a base change we can in addition assume that the strata of $(\overline{Y},\overline{\Delta})$ have irreducible fibres over $T$. Fix a log fibre $(F,\Delta_F)$ of $(\overline{Y},\overline{\Delta})\to T$. 

Each $(X_i,B_i)\to Y_i$ determines a sequence of smooth blowups $\overline{X}_i\to \overline{Y}$ and a boundary $\overline{B}_i$ with coefficients in $\Phi$ so that $B_i=\overline{B}_i|_{X_i}$.  
This in turn induces a sequence of smooth blowups $\alpha_i\colon F_i\to F$ and a boundary $B_{F_i}$ with ${\alpha_i}_*B_{F_i}\le \Delta_F$. By Lemma \ref{l-Ivol-constant-in-log-smooth-families}, 
$$
\Ivol(K_{F_i}+B_{F_i})=\Ivol(K_{X_i}+B_i)
$$ 
for each $i$. Moreover, since $(X_i,B_i)$ has a good minimal model, $(F_i,B_{F_i})$ has a good minimal model, by \cite[Theorem 1.2]{HMX18}. But then applying Lemma \ref{l-dcc-ivol-fixed-bir-model} gives a contradiction.

\end{proof}

%%%%%%%%%%%%%%%%%%%%%%%%%%%%%%%
%%%%%%%%%%%%%%%%%%%%%%%%%%%%%%%%%%%%%

\vspace{2cm}
%%%%%%%%%%%%%%%%%%%%%

\small
\textsc{Yau Mathematical Sciences Center, JingZhai, Tsinghua University, Hai Dian District, Beijing, China 100084  } \endgraf
\vspace{0.5cm}
\email{Email: birkar@tsinghua.edu.cn\\}

\end{document}